\documentclass[11pt]{amsart}
\usepackage[margin=1in]{geometry}
\usepackage{amsmath,amssymb,amsthm, hyperref, amscd}
\usepackage{enumitem}
\usepackage[T1]{fontenc}
\usepackage{textcomp}
\usepackage{tikz}
\usetikzlibrary{matrix,arrows,decorations.pathmorphing}
 \usepackage{tikz-cd}
\usepackage[scale=0.95,ttdefault]{AnonymousPro}
\hypersetup{
   colorlinks=true,
   linkcolor=blue,
    urlcolor=cyan,
    }

\usepackage[
    backend=biber,
    style=alphabetic,
    useprefix=true,
    maxnames=100,
    maxcitenames=100,
    sorting=anyt
]{biblatex}
\newcommand{\fa}{\mathfrak{a}}

\newcommand{\fr}{\mathfrak{r}}
\newcommand{\fm}{\mathfrak{m}}

\newcommand{\fN}{\mathfrak{N}}
\newcommand{\fp}{\mathfrak{p}}
\newcommand{\fq}{\mathfrak{q}}
\newcommand{\fP}{\mathfrak{P}}
\newcommand{\fQ}{\mathfrak{Q}}
\newcommand{\Max}{\operatorname{Max}}
\newcommand{\Spec}{\operatorname{Spec}}

\newcommand{\Frac}{\operatorname{Frac}}

\newcommand{\colim}{\operatorname{colim}}

\newcommand{\coker}{\operatorname{coker}}
\newcommand{\Tor}
{\operatorname{Tor}}
\newcommand{\Hom}
{\operatorname{Hom}}

\newcommand{\Coh}{\operatorname{Coh}}
\newcommand{\cO}{\mathcal{O}}

\newcommand{\cG}{\mathcal{G}}
\newcommand{\cF}{\mathcal{F}}
\newcommand{\cE}{\mathcal{E}}
\newcommand{\cK}{\mathcal{K}}

\newcommand{\cL}{\mathcal{L}}

\newcommand{\bA}{\mathbf{A}}
\newcommand{\bF}{\mathbf{F}}
\newcommand{\bP}{\mathbf{P}}

\newcommand{\bR}{\mathbf{R}}
\newcommand{\bQ}{\mathbf{Q}}

\newcommand{\bZ}{\mathbf{Z}}

\newcommand{\Tag}[1]{\href{https://stacks.math.columbia.edu/tag/#1}{\texttt{#1}}}
\newcommand{\citestacks}[1]{\cite[Tag \Tag{#1}]{stacks}}
\newcommand{\citetwostacks}[2]{\cite[Tags \Tag{#1} and \Tag{#2}]{stacks}}

\numberwithin{equation}{subsection}

\newtheorem{Thm}[equation]{Theorem}

\newtheorem{Lem}[equation]{Lemma}
\newtheorem{Cor}[equation]{Corollary}

\theoremstyle{definition}
\newtheorem{Def}[equation]{Definition}

\newtheorem{Condit}[equation]{Condition}

\newtheorem{Discu}[equation]{Discussion}
\newtheorem{Exam}[equation]{Example}

\theoremstyle{remark}
\newtheorem{Situ}[equation]{Situation}
\newtheorem{Ques}[equation]{Question}

\newtheorem{Rem}[equation]{Remark}

\newtheorem{StepRegQCOpen}{Step}

\title{%Finitely presented algebras
Regular rings over valuation rings} %and Kunz's Theorem}

\author{Shiji Lyu}
\address{Department of Mathematics, Statistics, and Computer Science\\University of Illinois at Chicago\\Chicago, IL
60607-7045\\USA}
\email{\href{mailto:slyu@alumni.princeton.edu}{slyu@alumni.princeton.edu}}
\urladdr{\url{https://homepages.math.uic.edu/~slyu/}}

\begin{document}

\begin{abstract}
Bertin (1972) defined regularity for coherent local rings, and Knaf (2004) studied the property for a local ring $A$ essentially finitely presented over a valuation ring $V$. 
We discuss several properties of this notion of regularity for such $A$, obtaining results parallel to results for regularity of Noetherian local rings. 
We include classical and modern topics: openness of loci, perfectoid big Cohen--Macaulay algebras, and cotangent complexes. 
We also give an application to Noetherian rings, showing a version of Kodaira's vanishing theorem in large enough residue characteristics.
\end{abstract}

\maketitle

Throughout, all rings are commutative with unit.
$\Max(R)$ denotes the set of maximal ideals of a ring $R$.
For $f\in R$,
$D(f)=D_R(f)$ denotes the principal open subset of $\Spec(R)$ defined by $f$,
$V(I)=V_R(I)$ denotes the closed subset of $\Spec(R)$ defined by the ideal $I$.
The \emph{constructible topology} of $\Spec R$ has a basis consisting of all sets of the form $V(J)\cap D(f)$,
where $J$ is a finitely generated ideal and $f\in R$.
For $\fp\in\Spec(R)$, $\kappa(\fp)$ denotes the residue field of $\fp$.
The notation $\dim R$ is always for the Krull dimension of $R$.

A \emph{valuation ring} is a local \emph{domain} whose finitely generated ideals are principal.
A \emph{Pr\"ufer domain} is an integral domain whose localizations at prime ideals are valuation rings.
For basic properties we refer the reader to \cite[p. 25]{Coherent-Rings}.

An integral domain is \emph{absolutely integrally closed} if it is normal and its fraction field is algebraically closed.
The \emph{absolute integral closure} of an integral domain $R$, denoted by $R^+$, is its integral closure in an algebraic closure of its fraction field.

For a ring $A$ and an $A$-module $M$,
the weak dimension $\operatorname{w.dim}_AM\in\bZ_{\geq 0}\cup\{\infty\}$ is the smallest integer $e$ so that $\operatorname{Tor}_A^{e+1}(M,-)=0$, or $\infty$.
The weak global dimension $\operatorname{w.dim}A$ is the sup of all $\operatorname{w.dim}_AM$, finite or not.
In other words, $\operatorname{w.dim}A\in\bZ_{\geq 0}\cup\{\infty\}$ is the smallest integer $e$ so that $\operatorname{Tor}_A^{e+1}(-,-)=0$, or $\infty$.

 \section{Introduction}

\subsection{Overview}
Valuation rings play an essential role in classical and modern algebraic and arithmetic geometry.
They can be used to tell integral dependence and properness, to detect singularities in birational geometry, to define the $v$- and arc-topologies, and as base rings for adic and perfectoid geometry.
On the other hand, the actual commutative algebra over (non-Noetherian) valuation rings seems to be rather undeveloped.

This article is the second of a projected series on this subject, the first being \cite{lyu-Prufer-Japanese}.  
This article mainly studies ``regular local rings'' over valuation rings, cf. \cite{Bertin-def-Regular} and \cite{Knaf-RLR-over-Prufer}.
Precisely, we are interested in a local ring $A$ essentially finitely presented over a valuation ring $V$ that has finite weak global dimension (see Discussion \ref{discu:NorRegloci}).
We will reserve the phrase ``regular local ring'' for Noetherian regular local rings.

We will establish various properties parallel to those of regularity of Noetherian rings.
We study openness of loci (\S\ref{sec:loci}), the direct summand theorem (\S\ref{sec:SplinterandKunz}), variants of Kunz's Theorem (\S\ref{sec:SplinterandKunz} and \S\ref{sec:KunzRefined}), cotangent complexes (\S\ref{sec:Cotangent1} and \S\ref{sec:Cotangent2}),
and vanishing theorems (\S\ref{sec:VanishingTh}).
These topics are mostly independent of each other.

Apart from the known ingredients from \cite{Coherent-Rings} and \cite{Knaf-RLR-over-Prufer},
we also provide and utilize two important technical inputs: big Cohen--Macaulay algebras (\S\ref{sec:BCM}) and approximation of valuation rings by (Noetherian) regular local rings (\S\ref{sec:approximation-regular-v}).
A discussion of their roles in this article can be found at the beginning of the corresponding sections.

In future work, we will study cohomological properties of finitely presented algebras over valuation rings,
such as associated primes, depth, and Cohen--Macaulayness.
We will make essential use of the approximation techniques discussed here.
We will also work towards Macaulayification (cf. \cite{Macaulay-Cesnavi}) over valuation rings.

\subsection{Openness of loci}
In \S\ref{sec:loci} we study openness of the normal locus $\operatorname{Nor}(B)$ and the ``regular'' locus $\operatorname{Reg}(B)=\{\fq\in\Spec B\mid \operatorname{w.dim} B_\fq<\infty\}$ for a finitely generated integral domain $B$ over a valuation ring $V$.
For a Noetherian ring we have 
 \begin{Thm}[{\cite[(32.B) Th.~73]{Mat-CA}} and {\cite[Corollaire 6.13.5]{EGA4_2}}]
     Let $R$ be a Noetherian ring. The following are equivalent.
     \begin{enumerate}
         [label=$(\roman*)$]
         \item Every finitely generated $R$-algebra has open regular locus.
         \item Every finite $R$-algebra has open regular locus.
         \item\label{J2Noe:finiteInsepHasOpen}
         For every $\fp\in\Spec R$ and every finite purely inseparable extension $F/\kappa(\fp)$,
         there exists a finite $R$-algebra $S$ that is an integral domain with fraction field $F$ and a nonzero $g\in S$ so that $S_g$ is regular.
     \end{enumerate}
     If $R$ satisfies the equivalent conditions, then every finitely generated $R$-algebra has open normal  locus.
 \end{Thm}

 Condition \ref{J2Noe:finiteInsepHasOpen} turns out to be not enough for general valuation rings (Remark \ref{rem:RegNotOpenVal}).
 Our main result for the normal and ``regular'' loci is
 \begin{Thm}[Theorem \ref{thm:RegQCopen} and Corollary \ref{cor:NorQCopen}]\label{thm:RegNorQCOpen-intro}
     For a valuation ring $V$, the following are equivalent.
     \begin{enumerate}
         [label=$(\roman*)$]
         \item $\operatorname{Reg}(B)$ is a quasi-compact open for
         every finitely generated $V$-algebra $B$ that is an integral domain.
         %has quasi-compact open regular locus.
         \item $\operatorname{Reg}(B)$(=$\operatorname{Nor}(B)$) is  open for every finite $V$-algebra $B$ that is an integral domain.
         %has open regular locus (=normal locus).
         \item\label{intro-J2-condition} For every $\fp\in\Spec V$ and every finite purely inseparable extension $F/\kappa(\fp)$,
         there exists a finite $V$-algebra $B$ that is an integral domain with fraction field $F$ and a nonzero $g\in B$ so that $B_g$ is regular;
         moreover, for every $\fp\in\Spec V$ such that $\fp V_\fp$ is principal,
        there exists $f\in V\setminus\fp$ so that $V_\fp=V_f$.
     \end{enumerate}
 If $V$ satisfies the equivalent conditions, then every finitely generated $V$-algebra that is an integral domain has quasi-compact open normal locus. 
 \end{Thm}

The proof of Theorem \ref{thm:RegQCopen} goes by standard d\'evissage arguments that can be found in the proof of the Noetherian case, suitably adapted, plus a reduction to the case $V$ is absolutely integrally closed (Lemma \ref{lem:N2isJ2} and Step \ref{step:RegOpenFinal}).
This latter reduction, as well as the proof of Corollary \ref{cor:NorQCopen},
requires input from the author's previous work on finiteness of integral closure \cite{lyu-Prufer-Japanese}.

For completeness, we also include a discussion on the reduced and integral domain loci.
\begin{Thm}[=Theorem \ref{thm:ReducedLocusQCOpen}]\label{thm:ReducedLocusQCOpen-intro}
    Let $R$ be an integral domain and let $X$ be a flat $R$-scheme of finite presentation.
    Then locus of points $x\in X$ so that $\cO_{X,x}$ is reduced (resp. an integral domain) is a quasi-compact open.
\end{Thm}
This yields %a corresponding statement for finitely presented algebra
 \begin{Cor}[=Corollary \ref{cor:RegNorQCopenForFlat}]\label{cor:RegNorQCopenForFlat-intro}
     Let $V$ be a valuation ring  that satisfies the equivalent conditions in Theorem \ref{thm:RegNorQCOpen-intro}.
     Let $X$ be a flat $V$-scheme of finite presentation.
     Then $\operatorname{Reg}(X)$ and $\operatorname{Nor}(X)$ are quasi-compact opens.
 \end{Cor}
Finally, we note that \ref{intro-J2-condition} in Theorem \ref{thm:RegNorQCOpen-intro} can be easily verified for finite rank valuation rings (and more, Remark \ref{rem:RegNorQCopenForWellOrdered}), so we get
  \begin{Cor}[=Corollary \ref{cor:RegNorQCopenForFiniteRank}]\label{cor:RegNorQCopenForFRK-intro}
     Let $V$ be a valuation ring of finite rank.
     Let $X$ be a flat $V$-scheme of finite presentation, or an integral $V$-scheme of finite type.
     Then $\operatorname{Reg}(X)$ and $\operatorname{Nor}(X)$ are quasi-compact opens.
 \end{Cor}

\subsection{The direct summand theorem}
Parallel to the Noetherian case, we show
\begin{Thm}[=Theorem \ref{thm:RLRisSplinter}]\label{thm:RLRisSplinter-intro}
Let $V$ be a valuation ring and let $A$ be an essentially finitely presented local $V$-algebra. If $\operatorname{w.dim}A<\infty$, then $A$ is a splinter.
 \end{Thm}

In the Noetherian case, the result is first proved in full generality by Andr\'e \cite{Andre-summand}, who also constructed (perfectoid) BCM algebras for all mixed-characteristic Noetherian local rings.
Classical observations by Hochster and Huneke (cf. \cite[(6.7)]{HH92-BCMisFlat})
shows that BCM algebras are flat over a regular local ring and their existence for all Noetherian local rings implies the direct summand theorem.
In \S\ref{sec:BCM} we study a similar notion for flat finitely presented algebras over a valuation ring and prove a similar flatness result (Theorem \ref{thm:BCMisflat}).
This notion behaves well with respect to Noetherian approximation (\S\ref{subsec:ConstructBCM}),
so we can use the construction in the Noetherian case \cite{Bhatt-int-clos-BCM,BMPSTWW} to get a faithfully flat $A$-algebra that contains an absolute integral closure of $A$,
which proves Theorem \ref{thm:RLRisSplinter-intro}.

\subsection{Variants of Kunz's Theorem}
Combining the discussions on BCM algebras in \S\ref{sec:BCM} and techniques from \cite{Kunz-Bhatt-Ma}, we have the following variant of the main result of \cite{Kunz-Bhatt-Ma}, where $\widehat{(-)}^p$ stands for $p$-adic completion.
 \begin{Thm}[=Theorem \ref{thm:PerfdKunz}]\label{thm:PerfdKunz-intro}
 Let $V$ be a valuation ring and let $A$ be an essentially finitely presented local $V$-algebra.
Assume that there exists a prime number $p$ not invertible in $V$.
The following are equivalent.
\begin{enumerate}[label=$(\roman*)$]
\item $\operatorname{w.dim} A<\infty$.
         \item $A$ is an integral domain and 
         \begin{align*}
             \operatorname{w.dim}_A\widehat{A^+}^p=\begin{cases}
                 0 & \text{the image of }V_{\fm_A\cap V}\text{ in }A\text{ is }p\text{-adically separated;}\\
                 1 & \text{otherwise.}
             \end{cases}
         \end{align*}
         \item $A$ is an integral domain and  $\operatorname{w.dim}_A\widehat{A^+}^p<\infty$.
         \item  There exists an $A$-algebra $S$ which is perfectoid,
         and an $S$-module $U$
         that satisfies $U/\sqrt{\fm_A S}U\neq 0$ and $\operatorname{w.dim}_AU<\infty$. 
     \end{enumerate}
 \end{Thm}

In equal characteristic $p>0$ we have the following.
\begin{Thm}[=Theorem \ref{thm:KunzALL} and Corollary \ref{cor:OneFeFiniteTorDim=Regular}]\label{thm:KunzALL-intro}
    Let $(V,\fm_V,k_V)$ be a valuation ring of characteristic $p>0$ and let $(A,\fm_A,k_A)$ be an essentially finitely presented local $V$-algebra.
    Then the following are equivalent.
    \begin{enumerate}[label=$(\roman*)$]
        \item\label{KunzALLintro-fwgdim} $\operatorname{w.dim} A<\infty$.
        %\item\label{Kunz-RLR} For every finitely generated ideal $I$ of $A$, $\operatorname{w.dim}_A A/I<\infty$; that is, $A$ is regular in the sense of Bertin \cite[D\'ef.~3.5]{Bertin-def-Regular}.
        \item\label{KunzALLintro-all-F-flat} $F^e_*A$ is flat over $A$ for all $e$.
        \item\label{KunzALLintro-some-F-flat} $F^e_*A$ is flat over $A$ for some $e>0$.
        \item\label{KunzALLintro-some-F-ffd} $\operatorname{w.dim}_AF^e_*A<\infty$ for some $e>0$.
        \item\label{KunzALLintro-perf-flat} The map $A\to A_{perf}=\colim_e F^e_*A$ is flat.
        \item\label{KunzALLintro-plus-flat}
        $A$ is an integral domain and the map $A\to A^+$ is flat.
        \item\label{KunzALLintro-some-perf-ffd} There exists an $A$-algebra $S$ which is perfect,
        and an $S$-module $U$
        that satisfies $U/\sqrt{\fm_A S}U\neq 0$ and $\operatorname{w.dim}_AU<\infty$. 
    \end{enumerate}
\end{Thm}
Here, it is straightforward to show all items except for \ref{KunzALLintro-some-F-ffd} are equivalent; indeed, in view of Theorem \ref{thm:PerfdKunz-intro}, the only nontrivial implication is \ref{KunzALLintro-fwgdim} implies \ref{KunzALLintro-all-F-flat},
which we settle in the proof of Theorem \ref{thm:KunzALL}.
For Noetherian local rings \ref{KunzALLintro-some-F-ffd} characterizes regularity \cite{Koh-Lee-exactness-complex-Kunz-theorem}, and we use a much more refined Noetherian approximation argument (\S\ref{sec:approximation-regular-v}) than we used in \S\ref{subsec:ConstructBCM} to get the equivalence of our \ref{KunzALLintro-fwgdim} and \ref{KunzALLintro-some-F-ffd}, see \S\ref{sec:KunzRefined}.
This argument somewhat resembles the approximation techniques in the joint work in progress \cite{Lyu-Bogdan-approximation},
where we approximate Noetherian local rings with other Noetherian local rings.

\subsection{Cotangent complexes}
The following question was posted at a workshop in Banff, October 2025, by Javier Carvajal-Rojas.
\begin{Ques}\label{ques:FflatOmegaflat}
    Let $A$ be an $\bF_p$-algebra.
    What is the relation between the flatness of $F_*A$ and that of $\Omega_{A/\bF_p}$ as $A$-modules?
\end{Ques}
When $A$ is Noetherian, $F$-finite, and reduced, both $F_*A$ and $\Omega_{A/\bF_p}$ are finite $A$-modules, and their flatness are equivalent, whereas for non-reduced rings there are easy counterexamples where $\Omega_{A/\bF_p}$ is free but $F_*A$ is not flat (cf. \cite[\S 6]{Andre-F-flatness} or \cite[Chap.~11]{Ma-Polstra-book}). %the flatness of $F_*A$ is equivalent to regularity by Kunz's Theorem \cite[Th.~107]{Mat-CA}.
The author is not aware of an answer to Question \ref{ques:FflatOmegaflat} for Noetherian reduced rings outside of the $F$-finite case.

On the other hand, recall the following  classical theorem, 
cf. \cite[Cor.~7.27 and Supplement, Th.~30]{Andre-cotangent-complex}.
\begin{Thm}
    Let $k$ be a field and let $(A,\fm_A,k_A)$ be a Noetherian local $k$-algebra.
    The following are equivalent.
    \begin{enumerate}[label=$(\roman*)$]
        \item $L_{A/k}$ has tor-amplitude in $[0,0]$ (\emph{i.e.} quasi-isomorphic to a flat module placed at degree $0$).
        \item $H^{-1}(L_{A/k}\otimes_A^Lk_A)=0$.
        \item $A\otimes_k F$ is regular for all finite extensions $F/k$.
        \item $A\otimes_k F$ is regular for all finite extensions $F/k$ inside $k^{1/p}$, where $p=\operatorname{char}k$.
    \end{enumerate}
    Here, by convention, $k^{1/p}=k$ if $p=0$.
\end{Thm}

We show (Theorems \ref{thm:Regular-Has-Flat-L}, \ref{thm:Regular-Has-Flat-L-m-not-fg}, and \ref{thm:Regular-Has-Flat-L-m-fg})
\begin{Thm}
    Let $k$ be a field and let $(A,\fm_A,k_A)$ be a local $k$-algebra.
    Assume that $A$ is essentially finitely presented over a valuation ring $V$.
    The following are equivalent.
    \begin{enumerate}[label=$(\roman*)$]
        \item\label{flatL-intro:flatL} $L_{A/k}$ has tor-amplitude in $[0,0]$.
        \item $H^{-1}(L_{A/k}\otimes_A^Lk_A)=0$.
        \item\label{flatL-intro:GReg}
        $\operatorname{w.dim}(A\otimes_kF)<\infty$ for all finite extensions $F/k$.
        \item $\operatorname{w.dim}(A\otimes_kF)<\infty$ for all finite extensions $F/k$ inside $k^{1/p}$, where $p=\operatorname{char}k$.
    \end{enumerate}
    Here, by convention, $k^{1/p}=k$ if $p=0$.
\end{Thm}
Note that we do not require $k\subseteq V$.

We loosely follow the proof ideas in \cite{Andre-cotangent-complex}.
To prove \ref{flatL-intro:flatL} implies \ref{flatL-intro:GReg} in the case the maximal ideal of $V$ is finitely generated (\emph{i.e.} principal),
we need to generalize the standard characterization of regular immersions to our setup, see Theorem \ref{thm:lci=Lin-1}, which requires the refined Noetherian approximation in \S\ref{sec:approximation-regular-v}.

\subsection{Vanishing theorems}
So far we have discussed counterparts of results known for Noetherian rings.
The following result is new to the author's knowledge even for Noetherian $V$, \emph{e.g.} $V=\bZ_p$.
It says that (a version of) Kodaira's vanishing theorem, which is false in positive or mixed characteristics, holds over valuation rings of large residue characteristics relative to the data defining the variety.
\begin{Thm}[=Theorem \ref{thm:VanishingForLargep-fixdegree}]\label{thm:VanishingForLargep-fixdegree-intro}
Let $r,m,n,d\in\bZ_{>0}$.
There exists a constant $C=C(r,m,n)$ with the following property.

    Let $V$ be a valuation ring of residue characteristic $p>C$,
    and let $X$ be a closed subscheme of $\bP^r_V$ defined by at most $m$ homogeneous polynomials of degree at most $n$.
    Assume that $X$ is flat over $\Spec V$ of relative dimension $d$,
    and assume that $\operatorname{Reg}(X)=X$.

    Let $\omega_X=\mathcal{E}xt^{r-d}_{\cO_{\bP^r_V}}(\cO_X,\cO_{\bP^r_V}(-r-1))$.
    Then %for every ample invertible sheaf $\cL$ on $X$ and every $j>0$ we have 
    $H^j(X,\omega_X(1))=0$ for all $j>0$.
\end{Thm}
There is another version of the theorem stated with Hilbert polynomials, see Theorem \ref{thm:VanishingForLargep}.

The idea of the proof is that if you have a counterexample $X_p$ over a valuation ring of $V_p$ of residue characteristic $p$ for infinitely many $p$,
then you can get a counterexample over an ultraproduct $V_{\natural}$ of $V_p$, which is a valuation ring of residue characteristic zero.
We make this ultraproduct idea precise (and the language of ultraproduct implicit) with infinite direct product.
It should be noted that we are not assuming $X$ is smooth over $V$, otherwise the result follows from a simple spreading-out argument (Remark \ref{rem:VanishingEasyForSmooth}).

Note that even when $V_p$ are Noetherian, $V_\natural$ is not, so we always need a similar vanishing theorem over non-Noetherian valuation rings of residue characteristic zero.
We show
\begin{Thm}[=Theorem \ref{thm:VanishingOverQ}]\label{thm:VanishingOverQ-intro}
    Let $V$ be a valuation ring of residue characteristic $0$,
    and let $X$ be a closed subscheme of $\bP^r_V$ for some $r$.
    Assume that $X$ is flat over $\Spec V$ of relative dimension $d$
    and assume that $\operatorname{Reg}(X)=X$.
    %$\operatorname{w.dim}\cO_{X,x}<\infty$ for all $x\in X$.

    Let $\omega_X=\mathcal{E}xt^{r-d}_{\bP^r_V}(\cO_X,\cO_{\bP^r_V}(-r-1))$.
    Then $\omega_X$ is an invertible sheaf, and for every ample invertible sheaf $\cL$ on $X$ and every $j>0$ we have $H^j(X,\omega_X\otimes\cL)=0$.
\end{Thm}

Theorems \ref{thm:VanishingOverQ}, \ref{thm:VanishingForLargep}, and \ref{thm:VanishingForLargep-fixdegree} are in their primary forms.
Similar results corresponding to other vanishing theorems can be proved essentially the same way.
However, to talk intelligently on such results, we need to establish a framework for duality and canonical sheaves on schemes of finite presentation over a valuation ring.
We plan to explore this in future work, based on the approximation techniques in $\S\ref{sec:approximation-regular-v}.$

\subsection{Acknowlegement}
The author thanks 
Rankeya Datta,
Shizhang Li,
Longke Tang,
Kevin Tucker,
Takehiko Yasuda, and
Jiahong Yu for helpful discussions. The author was supported by an AMS-Simons Travel Grant.

\section{Preliminary facts}
\subsection{Ideals of valuation rings}
\begin{Lem}\label{lem:RadicalIsPrimeInVal}
    Let $V$ be a valuation ring.
    A radical ideal of $V$ is prime.
\end{Lem}
\begin{proof}
    As the ideals of $V$ are totally ordered by inclusion, a minimal prime divisor of an ideal $I$ is necessarily contained in all prime ideals containing $I$.
    As a radical ideal is an intersection of primes it must then be prime itself. 
\end{proof}
\begin{Lem}\label{lem:SquareIsPrimeInVal}
    Let $V$ be a valuation ring and let $I$ be an ideal of $V$ such that $0\neq I\neq V$.
If $I= I^2$, then $I$ is a prime ideal and is not finitely generated.
    % The following are equivalent.
    % \begin{enumerate}[label=$(\roman*)$]
    %     \item $I$ is a prime ideal and is not finitely generated.
    %     \item $I=I^2$.
    % \end{enumerate}
\end{Lem}
\begin{proof}
    $I$ is not finitely generated by Nakayama's Lemma.
    To show $I$ is prime, it suffices to show $x^2\in I$ implies $x\in I$ (Lemma \ref{lem:RadicalIsPrimeInVal}).
    Since $I=I^2$, we have $x^2\in I^2=\bigcup_{a\in I}a^2V$,
    so there exists $a\in I$ with $x^2\in a^2V$, so $x\in aV$.
\end{proof}

\subsection{Essentially finitely generated algebras over Pr\"ufer domains}
We collect some basic facts about Pr\"ufer domains, cf. \cite[\S3.2]{ADVal}.
\begin{Thm}\label{thm:ModAndAlgOverV}
    Let $D$ be a Pr\"ufer domain.
    \begin{enumerate}[label=$(\roman*)$]
        \item\label{OverV:torsion-free-is-flat} Every torsion-free $D$-module is flat.
        \item\label{OverV:fp-is-coherent} Let $A$ be an essentially finitely presented $D$-algebra.
        Then $A$ is coherent.
        %\item\label{OverV:flat-fg-mod-is-fp} Let $A$ be an essentially finitely presented $V$-algebra and let $M$ be a finite $A$-module flat over $V$.
        %Then $M$ is a finitely presented $A$-module.
        %\item\label{OverV:flat-fg-alg-is-fp} Let $A$ be an essentially finitely generated flat $V$-algebra.
        %Then $A$ is an essentially finitely presented $V$-algebra.
    \end{enumerate}
\end{Thm}
\begin{proof}
For \ref{OverV:torsion-free-is-flat} see \citestacks{0539}.
    For \ref{OverV:fp-is-coherent} see \cite[Cor.~7.3.4, Th.~2.4.1, and Th.~2.4.2]{Coherent-Rings}.
\end{proof}

\begin{Discu}\label{discu:DomainsOverPrufer}
    Let $D$ be a Pr\"ufer domain and let $A$ be an essentially finitely generated $D$-algebra that is an integral domain.
    Then $A=S^{-1}B$ for some  finitely generated $D$-algebra $B$; and we may replace $B$ by its image in $A$ to assume $B$ is  an integral domain.

    Now, $B$ is finitely generated and flat over the image $\overline{D}$ of $D$ in $B$, and $\overline{D}$ is a Pr\"ufer domain.
    Therefore $B$ is flat over $\overline{D}$, hence finitely presented over $\overline{D}$ \citestacks{053G}.
    Therefore, $A$ is flat and essentially finitely presented over $\overline{D}$.
\end{Discu}
\subsection{Derived category of coherent rings}\label{subsec:DofCohRing}
We use basic homological properties of finitely presented modules coherent rings.
In particular, we use the fact that over a coherent ring $A$, finitely presented $A$-modules are exactly coherent $A$-modules \cite[Th.~2.3.2]{Coherent-Rings}.
Let $\Coh(A)$ denote the full subcategory of finitely presented $A$-modules.
Then $\Coh(A)$ is a Serre subcategory \cite[Th.~2.5.1]{Coherent-Rings} that is closed under the Tor and Ext functors \cite[Cor.~2.5.3]{Coherent-Rings}.
In particular, for all $X,Y\in D^b_{\Coh}(A)$, we have $X\otimes_A^LY\in D^-_{\Coh}(A)$ and $R\Hom_{A}(X,Y)\in D^+_{\Coh}(A)$,
and we have that $D^-_{\Coh}(A)$ is the subcategory of pseudo-coherent objects \citestacks{064N}, cf. \cite[Cor.~2.5.2]{Coherent-Rings}.

On a scheme $X$ such that $\cO_X(U)$ is coherent for all affine opens $U$,
we may similarly consider the abelian category $\Coh(X)$ of coherent sheaves on $X$,
which is the collection of finitely presented $\cO_X$-modules.

We record the following results on weak dimensions. Note that we do not require $I$ to be finitely generated or $A/I$ to be coherent in  \ref{wdim:quotient}.
\begin{Thm}\label{thm:ALLwdimFacts}
    Let $(A,\fm,k)$ be a coherent local ring.
    
\begin{enumerate}[label=$(\roman*)$]
    \item\label{wdim:testk}  Let $M$ be a coherent $A$-module. Then $\operatorname{pd}_AM=\operatorname{w.dim}_A M=\sup\{c\mid H^{-c}(k\otimes^L_A M)\neq 0\}\in\bZ_{\geq 0}\cup\{\infty\}$.

    \item\label{wdim:quotient} Let $I\subseteq\fm$ be an ideal of $A$.
    Then $\operatorname{w.dim}A\leq\operatorname{w.dim}A/I+\operatorname{w.dim}_A (A/I)$.

    \item\label{wdim:flatDescent} Let $B$ be a faithfully flat $A$-algebra. Then $\operatorname{w.dim}A\leq\operatorname{w.dim}B$.

    \item\label{wdim:flatAscent} Let $A\to B$ be a flat local map of coherent local rings. Then $\operatorname{w.dim}B\leq\operatorname{w.dim}A+\operatorname{w.dim}(B/\fm B)$.
\end{enumerate}
    
\end{Thm}
\begin{proof}
    For \ref{wdim:testk} see \cite[Th.~2.5.9]{Coherent-Rings}.
    For \ref{wdim:quotient} see \cite[Th.~3.1.3]{Coherent-Rings}.
    For \ref{wdim:flatDescent} see the proof of \cite[Th.~3.1.2]{Coherent-Rings}.
    Finally, \ref{wdim:flatAscent} follows from  \ref{wdim:quotient}, as $\operatorname{w.dim}_B(B/\fm B)\leq \operatorname{w.dim}_Ak$ \citestacks{066M}.
\end{proof}

\subsection{Finite rank valuation rings}\label{subsec:FiniteRankV}
For valuation rings of finite rank see for example \cite[Chap.~VI, \S 4]{Bourbaki-CA-5-7}.
We shall use the fact that the spectrum of a finite rank valuation ring $V$ is finite.
It follows that every %$(\Spec 
$V$-scheme essentially of finite type is a finite-dimensional %\footnote{It is trivial that $\dim R[T]\leq 2\dim R-1$.} 
Noetherian topological space, as it is the finite union of its fibers over $\Spec V$.

We also use the fact that if the fraction field of a valuation ring $V$ has finite transcendence degree over the prime field, or over a subfield of $V$, then $V$ has finite rank, cf. \cite[Chap.~VI, \S 10]{Bourbaki-CA-5-7}.

\subsection{An elementary reduction}\label{subsec:ReducedToSmallV}
Let $V$ be a valuation ring and let $B$ be a finitely presented $V$-algebra.
Let $\{K_\lambda\}_{\lambda}$ be
a family of subfields of $K:=\Frac(V)$ filtered by inclusion, so that each $V_{\lambda}:=V\cap K_{\lambda}$ is a valuation ring of finite rank.
As seen in \S\ref{subsec:FiniteRankV},
there are many choices of such families.
When $V$ contains a field $k$, we can take $\{K_\lambda\}_{\lambda}$ to be all subfields of $K$ of finite transcendence degree over $k$;
when $V$ dominates $\bZ_{(p)}$, we can take $\{K_\lambda\}_{\lambda}$ to be all subfields of $K$ of finite transcendence degree over $\bQ$;
when $V=V^+$, we can take  $\{K_\lambda\}_{\lambda}$ to be all algebraically closed subfields of $K$ of finite transcendence degree over the prime field, \emph{etc}.

By \citestacks{05N9} there exists a $\lambda_0$ and finitely presented ring map $V_{\lambda_0}\to B_{\lambda_0}$ so that $B=B_{\lambda_0}\otimes_{V_{\lambda_0}} V$,
so denoting by $B_\lambda$ the base change
$B_{\lambda_0}\otimes_{V_{\lambda_0}} V_\lambda$
we have
$B=\bigcup_{\lambda\geq\lambda_0}B_{\lambda}$
and $B_{\lambda}\to B$ is faithfully flat.
As noted in \S\ref{subsec:FiniteRankV} every $\Spec B_{\lambda}$ is a finite-dimensional Noetherian topological space.

For every localization $A=B_\fq$, we have
$A=\bigcup_{\lambda\geq\lambda_0}(B_{\lambda})_{\fq \cap B_{\lambda}}$
and $(B_{\lambda})_{\fq \cap B_{\lambda}}\to A$ is faithfully flat.
Every $\Spec ((B_{\lambda})_{\fq \cap B_{\lambda}})$ is a finite-dimensional Noetherian topological space.

\section{Openness of loci}\label{sec:loci}

In this section, we discuss openness of the reduced, integral domain, normal, and finite weak dimension (regular in the sense of \cite{Bertin-def-Regular,Knaf-RLR-over-Prufer}) loci.
We will mostly consider \emph{flat} finitely presented algebras over an integral domains $R$.
Flatness is necessary as illustrated in the following example.

\begin{Exam}\label{exam:BadValueGroup-PrincipalIsLimit}
    Let $\Gamma=\bZ\oplus\bZ^{\oplus \bZ_{\leq 0}}$ ordered lexicographically, that is, elements of $\Gamma$ are strings $(a_m)_{m\in\{-\infty\}\cup\bZ_{\leq 0}}$ of integers all but finitely many are $0$, and $(a_m)\geq (b_m)$ if and only if for the smallest $m$ so that $a_m\neq b_m$, we have $a_m>b_m.$
    Let $\gamma_e=(\delta_{em})_m$ (Kronecker $\delta$) for $e\in\{-\infty\}\cup\bZ_{\leq 0}$,
    so $\gamma_{0}< \gamma_{-1}<\ldots< \gamma_{-\infty}$.

    Let $V$ be a valuation ring with value group $\Gamma$ and for each $e\in\{-\infty\}\cup\bZ_{\leq 0}$ let $x_e\in V$ be an element of value $\gamma_e$.
    Then $\fp_e:=\sqrt{x_{e}V}$ is a prime ideal of $V$ and $\fp_e V_{\fp_e}=x_{e}V_{\fp_e}$.
    We also write $\fp$ for $\fp_{-\infty}$ to simplify our notations.
     Then $B=V/x_{-\infty} V$ is a finite finitely presented $V$-algebra such that  $B_\fp$ is a field.
    However, for no $f\in V\setminus\fp$, the ring $B_f$ is reduced.
    This is because the value of $f$ is less than $\gamma_{e+1}$ for some $e$,
    so $x_{-\infty}x^{-1}_{e}\in\fp$ but $f^Nx_{-\infty}x^{-1}_{e}\not\in x_{-\infty} V$ for any $N$.
\end{Exam}

\subsection{Reduced and integral domain loci}\label{subsec:redIntLoci}
\begin{Lem}\label{lem:ReducedLocusOpen}
    Let $B$ be a flat finitely presented algebra over an integral domain $R$ and let $\fq\in\Spec B$.

    If $B_\fq$ is reduced (resp. an integral domain), then there exists $g\in B\setminus\fq$  such  that $B_g$ is reduced (resp. an integral domain).
    %and a  prime element $\pi\in V$ so that $B_g$ is flat over $V/\pi V$ and is reduced.
\end{Lem}
\begin{proof}
Let $K=\Frac R$.
Then $B_K=B\otimes_RK$ is Noetherian, and  $(B\setminus\fq)^{-1}(B_K)$ is reduced (resp. an integral domain).
Therefore, for the nilradical (resp. some prime ideal) $I$ of $B_K$, we have
$(B\setminus\fq)^{-1}I=0$.
As $B_K$ is Noetherian, $I$ is finitely generated,
so there exists $g\in B\setminus \fq$ so that $I_g=0$, in other words, $(B_K)_g$ is reduced (resp. an integral domain).
Then $B_g$ is reduced (resp. an integral domain) by flatness.
\end{proof}
\begin{Thm}\label{thm:ReducedLocusQCOpen}
    Let $R$ be an integral domain and let $X$ be a flat $R$-scheme of finite presentation.
    Then locus of points $x\in X$ so that $\cO_{X,x}$ is reduced (resp. an integral domain) is a quasi-compact open.
\end{Thm}
\begin{proof}
    We may assume $X=\Spec B$ affine.
    Write $R=\bigcup_{\lambda\in \Lambda}R_{\lambda}$,
    the  direct union of a (filtered) family of subrings finitely generated over $\bZ$.
    We may assume there exists a minimal element $0$ of $\Lambda$ and a flat finitely presented $R_{0}$-algebra $B_{0}$ so that $B=R\otimes_{R_0}B_{0}$,
    see \citestacks{02JO}.
Write $B_\lambda=R_{\lambda}\otimes_{R_{\lambda_0}}B_{\lambda_0}$.

Write $K_\lambda=\Frac(R_\lambda)$ and $K=\Frac(R)$.
Write $C_\lambda=B_\lambda\otimes_{R_\lambda}K_\lambda$ and 
$C=B\otimes_{R}K$; they are Noetherian rings.
Then $C=C_{\lambda}\otimes_{K_\lambda}K$ is faithfully flat over $C_{\lambda}$.
Let $\fr_1,\ldots,\fr_m$ be all the associated primes of $C$, so $\fr_1\cap C_\lambda,\ldots,\fr_m\cap C_\lambda$ are all the associated primes of $C_\lambda$ (cf. \citestacks{0337}).
After replacing $\Lambda$ by $\Lambda_{\geq\lambda}$ for some $\lambda$,
we may assume $\fr_1\cap C_0,\ldots,\fr_m\cap C_0$ are distinct,
so $\fr_1\cap C_\lambda,\ldots,\fr_m\cap C_\lambda$ are distinct for all $\lambda$.
Furthermore, we may assume $\fr_j=(\fr_j\cap C_0)C$ for all $j$.

Given $\fq\in\Spec B$,
let $J=\{j\mid \fr_j\cap B\subseteq\fq\}$.
As $\fr_j=(\fr_j\cap C_0)C=(\fr_j\cap B_0)C$ for all $j$, we have $J=\{j\mid \fr_j\cap B_{\lambda}\subseteq\fq\cap B_\lambda\}$ for all $\lambda$.
Therefore we have a commutative diagram
\[\begin{CD}
    (B_{\lambda})_{\fq\cap B_\lambda}@>>> B_\fq\\
    @V{\alpha_\lambda}VV  @VV{\alpha}V\\
    %(B_{\lambda})_{\fq\cap B_\lambda}\otimes_{R_{\lambda}}\Frac(R_{\lambda})@>>> B_{\fq}\otimes_{R_{}}\Frac(R_{})\\
    (B_{\lambda}\setminus \fq\cap B_\lambda)^{-1}C_{\lambda}@>>> (B\setminus\fq)^{-1}C\\
     @V{\beta_\lambda}VV  @VV{\beta}V\\
     \prod_{j\in J}(C_{\lambda})_{\fr_j\cap C_\lambda} @>{\psi}>>\prod_{j\in J}C_{\fr_j}
\end{CD}\]
where $\alpha_\lambda$ and $\alpha$  are injective  by flatness, $\beta_\lambda$ and $\beta$  are injective by \citetwostacks{05BZ}{0311},
and $\psi$ is injetive by flatness again,
as $(C_{\lambda})_{\fr_j\cap C_\lambda}\to C_{\fr_j}$ is faithfully flat for all $j$.
We conclude that $(B_{\lambda})_{\fq\cap B_\lambda}\to B_\fq$ is injective.
    Consequently,
    if $\Sigma$ (resp. $\Sigma_{\lambda}$) is the reduced or integral domain locus of $B$ (resp. $B_\lambda$),
    then
    \[
    \Sigma=\bigcap_{\lambda\geq\lambda_0}v_{\lambda}^{-1}(\Sigma_{\lambda})
    \]
    where $v_{\lambda}:\Spec B\to\Spec B_\lambda$ is the canonical map.
    By Lemma \ref{lem:ReducedLocusOpen} $\Sigma$ and $\Sigma_\lambda$ are open,
    and $\Sigma_\lambda$ is quasi-compact as $\Spec B_\lambda$ is Noetherian.
    Therefore $v_{\lambda}^{-1}(\Sigma_{\lambda})$ is quasi-compact (as $v_{\lambda}$ is affine), hence closed in the constructible topology.
    Since the constructible topology is compact \citestacks{0901} we see the intersection must be a finite intersection, so $\Sigma$ is quasi-compact.
\end{proof}

\subsection{Normal and regular loci}
We consider these loci together for the following reason.
\begin{Lem}\label{lem:QFNormal=Val}
Let $(V,\fm_V,k_V)\to (A,\fm_A,k_A)$ be an essentially finitely generated local map of local rings.
Assume that $V$ is a valuation ring, $A$ is normal, and $\dim (A/\fm_V A)=0$.
Then $A$ is a valuation ring,
and the value group extension induced by the ring map $\overline{V}\to A$ has finite index,
where $\overline{V}$ is the image of $V$ in $A$.
    %Let $B$ be a flat quasi-finite algebra over a Pr\"ufer domain $D$.
    %Let $\fq\in\Spec B$.
    %hen the following hold.
    %\begin{enumerate}[label=$(\roman*)$]
        %\item\label{QF:normal=Val} 
        %Then $B_\fq$ is normal if and only if $B_\fq$ is a valuation ring.
        %\item\label{QF:normalOpen} If $B_\fq$ is normal,
        %then there exists $g\in B\setminus\fq$  such  that $B_g$ is normal (therefore a Pr\"ufer domain by \ref{QF:normal=Val}).
    %\end{enumerate}
\end{Lem}
\begin{proof}
We may assume $V=\overline{V}$.
Let $t_1,\ldots,t_r\in A$ be such that their images in $k_A$ is a transcendental basis of $k_A$ over $k_V$.
As $V$ is a valuation subring of $A$, it is easy to verify that $P:=V[t_1,\ldots,t_r]$ is a polynomial $V$-algebra, and that $\fm_A\cap P=\fm_V P$.
We may therefore replace $V$ by $P_{\fm_VP}$ to assume $k_A$ is finite over $k_V$.
In this case, $A/\fm_VA$ is finite over $k_V$,
so $A=B_\fq$ where $B$ is a finite $V$-subalgebra of $A$ and $\fq\in\Spec B$, by Zariski's Main Theorem \citetwostacks{00QA}{00QB}.

    %We may assume $B$ is an integral domain by Lemma \ref{lem:ReducedLocusOpen}.
    %(and \citestacks{053G}).
    %There exists a finite $D$-subalgebra $C$ of $B$ so that $\Spec B\to \Spec C$ is an open immersion \citestacks{00QB}.
    Let $E$ be the normalization of $B$.
    Then $E$ is the integral closure of $V$ in a finite extension of its fraction field, hence a Pr\"ufer domain, and 
    this gives the result,
    as the normal domain $A=B_\fq$ is localization of $E$. % if $B_\fq$ is normal.
\end{proof}

\begin{Discu}\label{discu:NorRegloci}
Let $A$ be a ring. We write $\operatorname{Nor}(A)$ for the normal locus of $A$, that is, the set of all primes $\fp$ so that $A_\fp$ is a normal domain.

    Let $A$ be a ring with coherent local rings. We write $\operatorname{Reg}(A)$ for the set of all primes $\fp$ so that every finitely generated ideal of $A_\fp$ has finite projective dimension (see \cite[D\'ef.~3.5]{Bertin-def-Regular}).

    We have $\operatorname{Reg}(A)\subseteq \operatorname{Nor}(A)$ by \cite[Cor.~4.3]{Bertin-def-Regular},
    with equality when $A$ is flat quasi-finite over a Pr\"ufer domain  (Lemma \ref{lem:QFNormal=Val}).
    When $A$ is flat and essentially finite presented over a Pr\"ufer domain  we have $\sup\{\operatorname{w.dim}(A_\fp)\mid\fp\in\operatorname{Reg}(A)\}<\infty$, see \cite[Th.~1]{Knaf-RLR-over-Prufer}.
    The same is therefore true for any integral domain $A$ essentially finitely generated over a Pr\"ufer domain $D$, see Discussion \ref{discu:DomainsOverPrufer}.

    We denote by $\operatorname{Nor}(X)$ and $\operatorname{Reg}(X)$ the corresponding notion for a scheme $X$,
    the latter for schemes with coherent local rings.
\end{Discu}

\begin{Lem}[cf. {\cite[(32.B) Th.~73]{Mat-CA}}]
\label{lem:J2isEveryFiniteJ0}
    Let $D$ be a Pr\"ufer domain. The following are equivalent.
    \begin{enumerate}[label=$(\roman*)$]
        \item\label{J2:RegLocHasOpen} 
        For every essentially finitely generated integral domain $A$ over $D$,
        $\operatorname{Reg}(A)$ contains a non-empty open subset of $\Spec A$.
        \item\label{J2:everyfiniteJ0}
        For every $\fp\in\Spec D$ and finite purely inseparable extension $F$ of $\kappa(\fp)$,
        there exists a integral domain $B$ finite over $D$ with fraction field $F$,
         so that  $\operatorname{Nor}(B)$ contains a non-empty open subset of $\Spec B$.
    \end{enumerate}
    If $D$ satisfies the equivalent conditions, so does any Pr\"ufer domain essentially finitely generated over $D$.
\end{Lem}
\begin{proof}
    The last contention is trivial for \ref{J2:RegLocHasOpen}, and we have \ref{J2:RegLocHasOpen} trivially implies \ref{J2:everyfiniteJ0}.
    Assume \ref{J2:everyfiniteJ0}. We need to show \ref{J2:RegLocHasOpen}.

    Let $A$ be an essentially finitely generated integral domain over $D$.
    Since \ref{J2:everyfiniteJ0} passes to quotients we may assume $\ker(D\to A)=0$,
    so $A$ is flat over $D$.
    Let $L=\Frac A$ and $K=\Frac D$.
    As $L$ is finitely generated over $K$,
    there exists $F/K$ finite purely inseparable (inside some algebraic closure of $L$) so that the composite $LF$ is separable over $F$, cf. \citestacks{04KM}.
    By \ref{J2:everyfiniteJ0},
    there exists an integral domain $B$ finite over $D$ with fraction field $F$,
         so that  $\operatorname{Nor}(B)$ contains a non-empty open subset of $\Spec B$.
    In particular there exists $0\neq g\in B$ so that $B_g$ is a Pr\"ufer domain, see Lemma \ref{lem:QFNormal=Val}.
    %Since $F/K$ is purely inseparable, $B_g$ is a valuation ring.
    
    Let $C$ be the image of $A\otimes_D B$ in $LF$, picture
    \[
    \begin{CD}
        D@>>> B@>>> B_g \\
        @VVV @VVV @VVV \\
        A@>>>C @>>> C_g.\\
    \end{CD}
    \]
    By construction, $B_g\to C_g$ is injective and essentially finitely generated.
    By \citestacks{051T} we may localize further to assume  $B_g\to C_g$ is flat and essentially finitely presented.
    Moreover, %$B_g\to C_g$ is smooth at $0\in\Spec C_g$, as 
    $LF/F$ is separable, so there exists a nonzero $h\in C$ so that $C_{gh}$ is smooth over $B_g$ \citestacks{00TF}.
    As $B_g$ is Pr\"ufer, Theorem \ref{thm:ALLwdimFacts}\ref{wdim:flatAscent} tells us $D(gh)\subseteq \operatorname{Reg}(C)$.

    Finally, by construction, $A\to C_{gh}$ is an injective map of integral domains of finite type (in fact quasi-finite).
    Therefore there exists $0\neq f\in A$ so that $A_f\to C_{fgh}$ is faithfully flat \citestacks{051T}.
    Therefore $D(f)\subseteq \operatorname{Reg}(A)$ by flat descent, Theorem \ref{thm:ALLwdimFacts}\ref{wdim:flatDescent}.
\end{proof}

In contrast to the Noetherian case,
these conditions do not guarantee openness of $\operatorname{Reg}(B)$ even for a finite flat $D$-algebra $B$ that is an integral domain, and the failure comes from two directions.
The first direction is the complicated structure the spectrum of a (non-Noetherian) valuation ring can have, see Condition \ref{condit:IsolatedPrimeInV} and Remark \ref{rem:RegNotOpenVal}.
The second direction, which we discuss right now, is the interference of different maximal ideals.
We take our example from \cite{lyu-Prufer-Japanese}.
In fact, all localizations of this $D$ at maximal ideals are DVRs.
\begin{Exam}\label{exam:J1.5notalwaysJ2}
    In \cite[Example~31]{lyu-Prufer-Japanese}, assume $\Frac V$ has characteristic zero.
    Then $D$ satisfies the equivalent conditions of Lemma \ref{lem:J2isEveryFiniteJ0},
    but $\operatorname{Nor}(A)=\operatorname{Reg}(A)$ is not open, where $A$ is the finite free $D$-algebra constructed in Step 3 of  \cite[Example~31]{lyu-Prufer-Japanese}.
\end{Exam}
\begin{proof}
As $\Frac V$ has characteristic zero, so does $\Frac D$, hence Lemma \ref{lem:J2isEveryFiniteJ0}\ref{J2:everyfiniteJ0} is trivial when $\fp=0$.
When $\fp\neq 0$, $D/\fp$ is a field, as $D_\fm$ is a DVR for all $\fm\in\operatorname{Max}(D)$,
hence  Lemma \ref{lem:J2isEveryFiniteJ0}\ref{J2:everyfiniteJ0} is trivial again.
Therefore $D$ satisfies Lemma \ref{lem:J2isEveryFiniteJ0}\ref{J2:everyfiniteJ0}.
The fact $\operatorname{Nor}(A)$ is not open follows from the argument in the last paragraph in Step 3 of  \cite[Example~31]{lyu-Prufer-Japanese}.
\end{proof}
On the positive side, we have the following result,
which shares the same ideas with \cite[Proof of Th.~4 and Th.~6]{lyu-Prufer-Japanese}.
A Pr\"ufer domain is said to be \emph{N-2} if its integral closure in every finite extension of its fraction field is finite.
An N-2 Pr\"ufer domain $D$ automatically satisfies Lemma \ref{lem:J2isEveryFiniteJ0}\ref{J2:everyfiniteJ0}, as the integral closure of $D$ in $F$ is finite by \cite[Lem.~28]{lyu-Prufer-Japanese}.
\begin{Lem}\label{lem:N2isJ2}
    Let $D$ be an N-2 Pr\"ufer domain, $B$ a flat finitely presented $D$-algebra, $\fq\in\Spec B$ a generic point in its fiber (\emph{i.e.} $\dim B_\fq/ (\fq\cap D)B_\fq=0$).
    Assume that $\fq\in \operatorname{Nor}(B)$.
    Then there exists $g\not\in\fq$ such that $D(g)\subseteq \operatorname{Reg}(B)$.
\end{Lem}
\begin{proof}
We may assume $B$ is an integral domain by Lemma \ref{lem:ReducedLocusOpen}, and we may assume that there exists a finite injective map $P\to B$, where $P$ is a polynomial algebra over $D$, cf. \cite[Proof of Lem.~29]{lyu-Prufer-Japanese}.

    Let $F=\Frac D$, $\overline{F}$ an algebraic closure of $F$, $(F_i)_i$ all finite subextensions of $\overline{F}/F$.
    Let $D_i$ (resp. $D_\infty$) be the integral closure of $D$ in $F_i$ (resp. $\overline{F})$ and let $P_i=D_i\otimes_DP$ (resp. $P_\infty=D_\infty\otimes_DP$).
    Let $C_i$ (resp. $C_{\infty}$) be the normalization of $D_i\otimes_D B$ (resp. $D_{\infty}\otimes_D B$), so we have a commutative diagram
    \[
    \begin{CD}
        D@>>> P@>>> B@.\\
        @VVV @VVV @VVV @.\\
        D_i@>>> P_i@>>> D_i\otimes_D B@>>> C_i\\
        @VVV @VVV @VVV @VVV\\
        D_\infty@>>> P_\infty@>>> D_\infty\otimes_D B@>>> C_\infty.\\
    \end{CD}
    \]
By \cite[Th.~7(iii)]{lyu-Prufer-Japanese},
$C_\infty$ is finite over $P_{\infty}$.
By \cite[Lem.~17]{lyu-Prufer-Japanese}, for large $i$, $C_\infty=P_{\infty}\otimes_{P_i}C_i$,
in particular $P_i\to C_i$ is finite by descent \citestacks{03C4}.
We fix such an $i$.

Let $\fr_{\infty}$ be a prime ideal of $C_{\infty}$ lying above $\fq$ that is  generic in its fiber over $D_{\infty}$. 
Let $\fr_i=\fr_{\infty}\cap C_i$.
As seen in the proof of \cite[Th.~6]{lyu-Prufer-Japanese}, $D_\infty\to C_\infty$ is smooth at $\fr_{\infty}$,
so $D_i\to C_i$ is smooth at $\fr_{i}$.
Fix $h\in C_i\setminus\fr_i$ so that $D_i\to (C_i)_h$ is smooth.

As $D\to D_i$ and $P_i\to C_i$ are finite,
so is $B\to C_i$.
As $B_\fq$ is a valuation ring (Lemma \ref{lem:QFNormal=Val}), $(C_i)_{\fq}$ is flat over $B_\fq$ as it is torsion-free, so we have $(C_i)_g$ is flat over $B$ for some $g\in B\setminus\fq$, see for example \citetwostacks{05IR}{01WM}.

Finally, as $D_i\to (C_i)_{gh}$ is smooth,
$D(gh)\subseteq \operatorname{Reg}(C_i)$ by Theorem \ref{thm:ALLwdimFacts}\ref{wdim:flatAscent},
hence the image of $D(gh)$ in $\Spec B$, which is an open subset \citestacks{00I1} of $\Spec B$ containing $\fq=\fr_i\cap B$, is contained in $\operatorname{Reg}(B)$ by Theorem \ref{thm:ALLwdimFacts}\ref{wdim:flatDescent}, as desired.
\end{proof}

We now specialize to valuation rings and give some further examples. % of J-1$\frac{1}{2}$ and non-J-1$\frac{1}{2}$ valuation rings.
\begin{Lem}\label{lem:J1.5classes}
    Let $V$ be a valuation ring. Assume either
    \begin{enumerate}
    \item\label{J1.5:Q} $V$ is a $\bQ$-algebra; or
        \item\label{J1.5:frk}$V$ has finite rank; or
        \item\label{J1.5:SphComp} $V$ is  spherically complete.
    \end{enumerate}
    Then $V$ %is J-2.
    satisfies the equivalent conditions of Lemma \ref{lem:J2isEveryFiniteJ0}.
\end{Lem}
\begin{proof}
As there are no nontrivial purely inseparable extensions in characteristic zero case \eqref{J1.5:Q} is trivial.

Let $K=\Frac V$. We may assume $\operatorname{char}K=p>0$.
    As both \eqref{J1.5:frk}\eqref{J1.5:SphComp} pass to quotients, finite extensions, and localizations,
    it suffices to show for every $x\in V\setminus V^p$,
    there exists $0\neq f\in V$ so that $V_f[T]/(T^p-x)$ is normal.
    In case \eqref{J1.5:frk} this is trivial as $V_f=K$ for some $0\neq f\in V$.
    
    Assume $V$ is spherically complete.
    As $x\in V\setminus V^p$, there exists a nonzero prime ideal $\fp$ of $V$ so that $x\not\in V^p+\fp $; for, if not, take $z_\fp\in V$ so that $x\in z_\fp^p+\fp $,
    any element in $\bigcap_{\fp\neq 0} (z_\fp+\fp)$ will be a $p$th root of $x$, and this intersection is nonempty, as $z_\fp-z_\fq\in \fq$ whenever $\fp\supseteq\fq$ and as $V$ is spherically complete.
    Now, take $0\neq f\in \fp$.
    For $\fq\in D(f)$ we have $\fp\not\subseteq\fq$, so $\fq\subseteq\fp$, so $x$ is not a $p$th power in $V/\fq$, as $x\not\in V^p+\fq$.
    Therefore $(V_\fq/\fq V_\fq)[T]/(T^p-x)$ is a field.
    By Theorem \ref{thm:ALLwdimFacts}\ref{wdim:flatAscent}, $\operatorname{w.dim}(V_\fq[T]/(T^p-x))<\infty$,
    so $V_f[T]/(T^p-x)$ is normal, as desired.
\end{proof}

\begin{Exam}\label{exam:notJ1.5}
    Let $\Gamma=\bQ^{\oplus \bZ_{\leq 0}}$ ordered lexicographically, that is, elements of $\Gamma$ are strings $(a_m)_{m\leq 0}$ of rational numbers all but finitely many are $0$, and $(a_m)\geq (b_m)$ if and only if for the smallest $m$ so that $a_m\neq b_m$, we have $a_m>b_m.$
    Let $\gamma_e=(\delta_{em})_m$ (Kronecker $\delta$),
    so $\gamma_{0}\leq \gamma_{-1}\leq\ldots$.

    Let $k=\bF_p(T_0,T_1,\ldots)$, and let $V=\{f=\sum_\gamma a_\gamma t^\gamma\in k[[t^\Gamma]]\mid k(a^{1/p}_{\gamma}\mid \gamma\in\Gamma)\text{ is finite over }k\}$.
    Then $V$ is a valuation subring of the Hahn series ring $k[[t^\Gamma]]$.
    Consider $x=\sum_{e}T_et^{\gamma_{-e}}\in V$.
    It is then clear that $x\not\in V^p$ but $x\in V^p+\fp$ for all $0\neq \fp\in\Spec V$.
    Therefore $(V_\fp/\fp V_\fp)[T]/(T^p-x)$ is not reduced for all $0\neq \fp\in\Spec V$,
    so $V_\fp[T]/(T^p-x)$ is not a valuation ring by \cite[Th.~2.9]{Knaf-RLR-over-Prufer},
    hence $V$ does not satisfy the equivalent conditions of Lemma \ref{lem:J2isEveryFiniteJ0}.
\end{Exam}

We consider the following condition for a valuation ring $V$ which clearly  depends only on the value group of $V$.
\begin{Condit}\label{condit:IsolatedPrimeInV}
   For every $0\neq\fp\in \Spec V$ so that $\fp V_\fp$ is principal, there exists $f\in V\setminus\fp$ so that $\fp V_f$ is principal; in other words, $V_f=V_\fp$\footnote{A nonzero principal prime ideal $\fp$ in a valuation ring $(V,\fm)$ must be maximal. To see this, let $\fp =x V$ and assume there exist $y\in \fm\setminus\fp$. Then $y\not\in xV$, so $x\in yV$, therefore $xy^{-1}\in V\setminus xV$, contradicting the assumption $\fp$ is prime.}. 
\end{Condit}
\begin{Rem}\label{rem:RegNotOpenVal}
    In Example \ref{exam:BadValueGroup-PrincipalIsLimit}, $V$ and $\fp$ do not satisfy this condition.
Note that we can make $V$ contain $\bQ$ and be spherically complete, hence satisfying the equivalent conditions of Lemma \ref{lem:J2isEveryFiniteJ0} by Lemma \ref{lem:J1.5classes}.
Therefore, the first paregraph of the proof of Theorem \ref{thm:RegQCopen} below tells us for $B=V[T]/(T^2-x_{-\infty})$ the locus $\operatorname{Nor}(B)=\operatorname{Reg}(B)$ is not open, showing that the equivalent conditions of Lemma \ref{lem:J2isEveryFiniteJ0} do not guarantee openness even for valuation rings.
\end{Rem}

\begin{Rem}\label{rem:isolatedpQuotient}
    Note that for prime ideals $\fp\subsetneq\fP$ of  a valuation ring $V$,
    $\fP V_\fP/\fp V_\fP$ is principal if and only if $\fP V_\fP$ is principal.
    Indeed, we may assume $\fP$ is the maximal ideal of $V$.
    If $x\in\fP$ is such that $\fP=xV+\fp$, then for every $y\in \fP$ we can write $y=xz+r$ where $z\in V,r\in \fp$.
    As $x\not\in\fp$, the valuation of both $xz$ and $r$ are no less than that of $x$, so $y\in xV$.
    Therefore Condition \ref{condit:IsolatedPrimeInV} passes to quotients.
    % As \eqref{J2QC:J1.5} in Theorem \ref{thm:RegQCopen} also passes to quotients, see Lemma \ref{lem:J2isEveryFiniteJ0}, it follows from  Theorem \ref{thm:RegQCopen} and Corollary \ref{cor:NorQCopen} that for every integral domain $B$ finitely generated over $V$ that satisfies \eqref{J2QC:J1.5}\eqref{J2QC:PrincipalIsolated} in Theorem \ref{thm:RegQCopen}, $\operatorname{Reg}(B)$ and $\operatorname{Nor}(B)$ are quasi-compact opens, as $B$ is flat and finitely presented over the image of $V$ in $B$,
    % see \citestacks{053G}.
\end{Rem}

Let us give the main result of this subsection.

\begin{Thm}\label{thm:RegQCopen}
    Let $V$ be a valuation ring.
    %Then the following are equivalent.
    %\begin{enumerate}[label=$(\roman*)$]
    %\item $\operatorname{Reg}(A)$ is open for every  $V$-algebra $A$ that is finite free as a $V$-module.
    %    \item $\operatorname{Reg}(X)$ is a quasi-compact open for every flat $V$-scheme of finite presentation.
    %\item 
    The following are equivalent.
    \begin{enumerate}[label=$(\roman*)$]
        \item\label{J2QC:J1.5+Isolatedp} $V$ satisfies Condition \ref{condit:IsolatedPrimeInV} and the equivalent conditions of Lemma \ref{lem:J2isEveryFiniteJ0}.
        \item\label{J2QC:RegQCOpen} For every integral $V$-scheme $X$ of finite type, $\operatorname{Reg}(X)$ is a quasi-compact open.
        \item\label{J2QC:RegQCOpenFINITE} For every finite $V$-algebra $B$ that is an integral domain, $\operatorname{Reg}(B)=\operatorname{Nor}(B)$ is open.
    \end{enumerate}
    %Moreover, if $X$ satisfies the 
     % Let $X$ be a flat $V$-scheme of finite presentation.
     % Then $\operatorname{Reg}(X)$ is a quasi-compact open.
\end{Thm}
We have \ref{J2QC:RegQCOpen} trivially implies \ref{J2QC:RegQCOpenFINITE} and \ref{J2QC:RegQCOpenFINITE} trivially implies Lemma \ref{lem:J2isEveryFiniteJ0}\ref{J2:everyfiniteJ0}.
To see why \ref{J2QC:RegQCOpenFINITE} implies Condition \ref{condit:IsolatedPrimeInV},
let $\fp\in\Spec V$ and $x\in\fp$ be such that $\fp V_{\fp}=xV_\fp$.
Then $B:=V[T]/(T^2-x)$ is such that $B_\fp$ is a valuation ring (apply Theorem \ref{thm:ALLwdimFacts}\ref{wdim:quotient} with $A=B_\fp,I=TB_\fp$).
By \ref{J2QC:RegQCOpenFINITE} (and, say, \citestacks{01WM}) there exists $f\in V\setminus\fp$ such that $B_f$ is normal.
If $B_f\neq B_\fp$, then there exists $\fq\in D(f)\subseteq \Spec V$ properly containing $\fp$, so we can take $y\in \fq\setminus\fp$,
so $z:=xy^{-2}\in V$,
and $B_\fq=V_\fq[T]/(T^2-x)=V_\fq[T]/(T^2-zy^2)$.
Then $Ty^{-1}\in \Frac(B_\fq)\setminus B_{\fq}$ and $(Ty^{-1})^2=z\in B_\fq$, showing $B_\fq$ not normal.
 
It remains to show \ref{J2QC:J1.5+Isolatedp} implies \ref{J2QC:RegQCOpen}.
     We may assume $X=\Spec B$ affine.
     We may also assume $V$ is a subring of $B$ (Remark \ref{rem:isolatedpQuotient}),
     so $B$ is flat over $V$,
     hence finitely presented over $V$ \citestacks{0GSE}.
     Let $\fq\in\operatorname{Reg}(B)$ and let $\fp=\fq\cap V$.
     Consider the statement
     \begin{align}\label{RegOpen:Constructible}
     \operatorname{Reg}(B)\text{ is a neighborhood of }\fq\text{ in the constructible topology},
     \end{align}
in other words, there exists a \emph{finitely generated} ideal $J$ of $B$ contained in $\fq$ and $g\in B\setminus\fq$ such that $V(J)\cap D(g)\subseteq\operatorname{Reg}(B)$.
The rest of the proof consists of several steps around \eqref{RegOpen:Constructible}.

\begin{StepRegQCOpen}\label{step:pFG}
    In this step, we show if $\fp V_\fp$ is principal, %finitely generated (\emph{i.e.} principal),
    then \eqref{RegOpen:Constructible} is true.

    As $V$ satisfies Condition \ref{condit:IsolatedPrimeInV}, we may assume $\fp$ is principal.
    The ring map $V/\fp\to B/\fq$ is finitely generated and flat, so it is finitely presented \citestacks{0GSE}, so $V\to B/\fq$ is finitely presented as $\fp$ is finitely generated.
Therefore $B\to B/\fq$ is also finitely presented \citestacks{00F4}, so
 $\fq$ is  finitely generated \citestacks{00R2}.
As $B_\fq$ has finite weak global dimension, it now follows that there exists a sequence
\[
\begin{CD}
    0@>>> B^{r_l}@>>> ... @>>> B^{r_1}@>>> B@>>> B/\fq@>>> 0 
\end{CD}
\]
of coherent $B$-modules that is exact after localizing at $\fq$.
Therefore the sequence is exact after localizing at $g$ for some $g\not\in\fq$, as the homology modules of the sequence are also coherent (cf. \S\ref{subsec:DofCohRing}). 
Therefore $B_g/\fq B_g$ has finite tor dimension over $B_g$.
Furthermore, we may also assume $\operatorname{Reg}(B/\fq)\subseteq D(g)$,
as $V$ satisfies  Lemma \ref{lem:J2isEveryFiniteJ0}\ref{J2:RegLocHasOpen}.
Therefore $\operatorname{Reg}(B)$ contains $V(\fq)\cap D(g)$ by Theorem \ref{thm:ALLwdimFacts}\ref{wdim:quotient}.
\end{StepRegQCOpen}

\begin{StepRegQCOpen}\label{step:slice}
    In this step, we show that we can slice $B$; more precisely, we show that there exists a sequence $\underline{x}=x_1,\ldots,x_m\in \fq$ so that 
    %$B/\underline{x}B$ is  flat over $V$,
    $B_\fq/\underline{x}B_\fq$ is a valuation ring, 
    %whose value group is a finite index extension of that of $V$, 
    and that there exists $g\in B\setminus \fq$ such that $B_g/\underline{x}B_g$ is  flat over $V$ and that $\operatorname{Reg}(B)\cap V(\underline{x})\cap D(g)\supseteq \operatorname{Reg}(B/\underline{x} B)\cap D(g)$;
    in particular, \eqref{RegOpen:Constructible} for $B_g/\underline{x}B_g$ and  $\fq B_g/\underline{x}B_g$ implies \eqref{RegOpen:Constructible} for $B$ and $\fq$. 

    Let $\overline{O}=B_\fq/\fp B_\fq$, a (Noetherian) Cohen--Macaulay local ring \cite[Th.~2.9]{Knaf-RLR-over-Prufer}.
    Let $m=\dim \overline{O}$ and let $x_1,\ldots,x_m$ be elements in $\fq$ that are linearly independent in $\fq\overline{O}/\fq^2\overline{O}$ and form a regular sequence in $\overline{O}$.
    Then $\underline{x}$ is a regular sequence in $B_\fq$ and $(B/\underline{x}B)_\fq$ is flat over $V$ \citestacks{0470}.
    By \cite[Th.~2]{Knaf-RLR-over-Prufer},
    $(B/\underline{x}B)_\fq$ has finite weak global dimension, hence a valuation ring whose value group is a finite index extension of that of $V_\fp$
    (\cite[Cor.~4.3]{Bertin-def-Regular} and Lemma \ref{lem:QFNormal=Val}).

    %To show \eqref{RegOpen:Constructible} for $B/\underline{x}B$ and  $\fq/\underline{x}B$ implies \eqref{RegOpen:Constructible} for $B$ and $\fq$,
    % By \citestacks{05IR}, there exists $g_1\not\in\fq$ such that $(B/\underline{x}B)_{g_1}$ is flat over $V$.
    % By \citetwostacks{045U}{00RL},
    % we can choose $g_2\not\in\fq$ so that all fibers of $V\to B_{g_2}$ and $V\to (B/\underline{x}B)_{g_2}$
    % are Cohen--Macaulay and equidimensional.
    % Comparing dimensions, this forces $\underline{x}$  to be  a regular sequence in the Cohen--Macaulay local ring $B_\fQ/\fP B_\fQ$ for every $\fQ\in V(\underline{x})\cap D(g_1g_2)$, where $\fP=\fQ\cap V$.
    % By \citestacks{0470}, $\underline{x}$ is  a regular sequence in $B_\fQ$ for every $\fQ\in V(\underline{x})\cap D(g_1g_2)$. 
    %We choose $g_2$ so that $\underline{x}$ is a regular sequence in $B_\fQ$ for every $\fQ\in V(\underline{x})\cap D(g)$, Corollary \ref{cor:SpreadOutRegSeq}\footnote{It is also possible to force this using Cohen--Macaulayness instead, cf. \citetwostacks{045U}{00RL}}.
Since $B$ is coherent, the module $N_1:=B[x_1]=\ker(B\xrightarrow{\times x_1}B)$ is coherent, and so are $N_2:=(B/x_1B)[x_2],N_3:=(B/(x_1,x_2))[x_3],\ldots, N_m:=(B/(x_1,\ldots,x_{m-1}))[x_m]$.
As $\underline{x}$ is a regular sequence in $B_\fq$, we have $(N_j)_\fq=0$,
so after localization, we may assume $N_j=0$ for all $j.$
Then $\underline{x}$ is a regular sequence in $B_\fQ$ for every $\fQ\in V(\underline{x})$, so
     $\operatorname{Reg}(B)\cap V(\underline{x})\supseteq \operatorname{Reg}(B/\underline{x} B)$ by Theorem \ref{thm:ALLwdimFacts}\ref{wdim:quotient}, as desired.
    %it suffices to show that there exists $g\in B\setminus \fq$ such that $\operatorname{Reg}(B)\cap V(\underline{x})\cap D(g)\supseteq \operatorname{Reg}(B/\underline{x} B)\cap D(g)$.
    %We choose $g$ so that 
\end{StepRegQCOpen}

\begin{StepRegQCOpen}\label{step:ALMOST}
In this step, we show if $\fq B_\fq = \fp B_\fq$ (which is always the case if $B_\fq$ is a valuation ring and $\fp V_\fp$ is not principal, cf. \cite[Th.~2.9]{Knaf-RLR-over-Prufer}), then
there exists $g\in B\setminus\fq$ such that 
$\operatorname{Reg}(B_g/\fp B_g)=\Spec(B_g/\fp B_g)$ and that
$V(\fp B)\cap D(g)\subseteq\operatorname{Reg}(B)$.

We know $B_\fq/\fp B_\fq$ is a field.
After localizing some $g$ we may assume $B_\fp/\fp B_\fp$ is a regular Noetherian domain, as the field $\kappa(\fp)$ is J-2 \citestacks{07PJ}.
As $B$ is flat over $V$, $B/\fp B$ is flat over $V/\fp$, thus a subring of $B_\fp/\fp B_\fp$.
We conclude that $B/\fp B$ is an integral domain,
so $\fp B=\fq$.
We may then assume, after localization, that $\operatorname{Reg}(B/\fp B)=\Spec (B/\fp B)$, %for some $g\not\in\fp B$ 
as $V$ satisfies  Lemma \ref{lem:J2isEveryFiniteJ0}\ref{J2:RegLocHasOpen}.
Therefore $\operatorname{Reg}(B)$ contains $V(\fp B)$ by Theorem \ref{thm:ALLwdimFacts}\ref{wdim:quotient}.
\end{StepRegQCOpen}

\begin{StepRegQCOpen}\label{step:frkV}
    In this step, we prove the theorem for all finite rank $V$.

    We know $\fp=\sqrt{\pi V}$ for some $\pi\in V$, so $V(\fp B)=V(\pi B)$.
    Therefore, the combination of Steps \ref{step:pFG}--\ref{step:ALMOST} tells us \eqref{RegOpen:Constructible} holds.
    As %$\Spec B$ is a Noetherian topological space (\S\ref{subsec:FiniteRankV}) and as 
    $\operatorname{Reg}(B)$ is stable under generalization,
    it follows that $\operatorname{Reg}(B)$ is open \citestacks{0903},
    and it is quasi-compact as $\Spec B$ is a Noetherian topological space (\S\ref{subsec:FiniteRankV}).
    %    the theorem follows from general topology \citestacks{0541}.
\end{StepRegQCOpen}

\begin{StepRegQCOpen}\label{step:ConstrIsEnough}
    In this step, we prove that if \eqref{RegOpen:Constructible} is true for all $\fq\in\operatorname{Reg}(B)$ (\emph{i.e.} $\operatorname{Reg}(B)$ is open in the constructible topology), then $\operatorname{Reg}(B)$ is a quasi-compact open (\emph{i.e.} the theorem holds).
    %the theorem for all N-2 $V$.

    %Step \ref{step:slice} and Lemma \ref{lem:N2isJ2} tell us \eqref{RegOpen:Constructible} is true for all $\fq\in\operatorname{Reg}(B)$, \emph{i.e.} 
    
    Write $V=\bigcup_{\lambda}V_{\lambda}$ and $B=\bigcup_{\lambda\geq\lambda_0}B_{\lambda}$  as in \S\ref{subsec:ReducedToSmallV}.
    Then flat descent (Theorem \ref{thm:ALLwdimFacts}\ref{wdim:flatDescent}) and limit \cite[Th.~6.2.2]{Coherent-Rings} tell us %that if $\operatorname{Reg}(B)$ (resp. $\Sigma_{\lambda}$) is the reduced or integral domain locus of $B$ (resp. $B_\lambda$),
    %then
    \[
    \operatorname{Reg}(B)=\bigcap_{\lambda\geq\lambda_0}v_{\lambda}^{-1}(\operatorname{Reg}(B_{\lambda}))
    \]
    where $v_{\lambda}:\Spec B\to\Spec B_\lambda$ is the canonical map.
    By Step \ref{step:frkV} 
    $\operatorname{Reg}(B_{\lambda})$ is a quasi-compact open,
    therefore $v_{\lambda}^{-1}(\operatorname{Reg}(B_{\lambda}))$ is a quasi-compact open.
    Since the constructible topology is compact \citestacks{0901} we see the intersection must be a finite intersection, so $\operatorname{Reg}(B)$ is a quasi-compact open.
\end{StepRegQCOpen}

\begin{StepRegQCOpen}\label{step:N-2V}
    In this step, we note that the theorem is proved for all N-2 $V$. 
    Indeed, Step \ref{step:slice} and Lemma \ref{lem:N2isJ2} give us \eqref{RegOpen:Constructible},
    so we win by Step \ref{step:ConstrIsEnough}.
\end{StepRegQCOpen}

\begin{StepRegQCOpen}\label{step:RegOpenFinal}
    In this step we finish the proof.
    
    By Step \ref{step:ConstrIsEnough} it suffices to prove \eqref{RegOpen:Constructible}.
    By Steps \ref{step:pFG} and \ref{step:slice} we may assume $B_\fq$ is a valuation ring and $\fp V_\fp$ is not principal, so 
    we may assume $\operatorname{Reg}(B/\fp B)=\Spec(B/\fp B)$ and that
$V(\fp B)\subseteq\operatorname{Reg}(B)$ by Step \ref{step:ALMOST}.

As $V$ is a valuation ring, any prime ideal of $V$ not containing $\fp$ must be contained in $\fp$.
As $V(\fp B)\subseteq\operatorname{Reg}(B)$,
we see that \eqref{RegOpen:Constructible} for $B$ and $\fq$ is equivalent to \eqref{RegOpen:Constructible} for $B_\fp$ and $\fq B_\fp$, so we may assume $\fp$ is the maximal ideal of $V$.

Let $W$ be a spherically complete valuation ring dominating $V$ whose value group is the divisible closure of that of $V$ and whose residue field is $V/\fp$.
Then $W$ is defectless and has divisible value group, hence is N-2 by \cite[Lem.~18]{lyu-Prufer-Japanese}.
Consequently $\operatorname{Reg}({B\otimes_V W})$ is open by Step \ref{step:N-2V}.
On the other hand, as $W/\sqrt{\fp W}=V/\fp$ and as $\operatorname{Reg}(B/\fp B)=\Spec(B/\fp B)$, it follows from Theorem \ref{thm:ALLwdimFacts}\ref{wdim:flatAscent} that $\operatorname{Reg}({B\otimes_V W})\supseteq V(\sqrt{\fp W}(B\otimes_V W))=V(\fp(B\otimes_V W))=\bigcap_{\pi\in\fp} V(\pi(B\otimes_V W))$.
As the constructible topology is compact \citestacks{0901} we have $\operatorname{Reg}({B\otimes_V W})\supseteq V(\pi(B\otimes_V W))$ for some $\pi\in\fp$, therefore $\operatorname{Reg}(B)\supseteq V(\pi B)$ by flat descent, Theorem \ref{thm:ALLwdimFacts}\ref{wdim:flatDescent}, as desired. 
\end{StepRegQCOpen}

The proof of Theorem \ref{thm:RegQCopen} is now finished.

 \begin{Cor}\label{cor:NorQCopen}
     Let $V$ be a valuation ring  that satisfies the equivalent conditions in Theorem \ref{thm:RegQCopen}.
     Let $X$ be an integral $V$-scheme of finite type.
     Then $\operatorname{Nor}(X)$ is a quasi-compact open.
 \end{Cor}
 \begin{proof}
 As in the proof of Theorem \ref{thm:RegQCopen} we may assume $X=\Spec B$ affine and $V\to B$ flat finitely presented.
    Using the same argument as in Steps \ref{step:frkV} and \ref{step:ConstrIsEnough} of the proof of Theorem \ref{thm:RegQCopen},
    using \citestacks{033G} instead of Theorem \ref{thm:ALLwdimFacts}\ref{wdim:flatDescent},
    we see
    it suffices to show $\operatorname{Nor}(B)$ is open.

    Let $\fq\in\operatorname{Nor}(B)$.
    We must show there exist $g\in B\setminus \fq$ such that $D(g)\subseteq\operatorname{Nor}(B)$, \emph{i.e.}, $B_g$ is normal.
    %We may assume $B$ is an integral domain by Lemma \ref{lem:ReducedLocusOpen}.
    
    Let $K=\Frac V$.
    Then $(B\setminus \fq)^{-1}B_K$ is normal,
    so
    \[
\bigcap_{g\in B\setminus \fq} D_{B_K}(g)\subseteq\operatorname{Nor}(B_K).
\]
We know $\operatorname{Nor}(B_K)$ is open by \cite[Prop.~6.13.4]{EGA4_2}, as $K$ is J-2 \citestacks{07PJ}.
As the constructible topology is compact \citestacks{0901}
there exists $g\in B\setminus \fq$ so that $(B_K)_g$ is a normal domain.
Replacing $B$ by $B_g$ we may assume $B_K$ is a normal domain.

As seen in the proof of \cite[Lem.~29]{lyu-Prufer-Japanese},
we may assume that there exists a finite injective map $P\to B$, where $P$ is a polynomial ring over $V$.
Let $B'$ be as in \cite[Th.~22]{lyu-Prufer-Japanese},
so $B'$ is finite over $B$, and $B_\fq=B'_\fq$ as $B_\fq$ is normal.
Fix $g_1\in B\setminus\fq$ so that $B_{g_1}=B'_{g_1}$.

Let $\fm$ be the maximal ideal of $V$.
Consider the ring $B_{\fm P}$, a finite flat algebra over the valuation ring $P_{\fm P}$.
As $(B\setminus\fq)^{-1}B_{\fm P}$ is a normal domain, we see from Lemma \ref{lem:QFNormal=Val} that it is Pr\"ufer, so 
\[
\bigcap_{g\in B\setminus \fq} D_{B_{\fm P}}(g)\subseteq\operatorname{Reg}(B_{\fm P}).
\]
We know $P_{\fm P}$ satisfies the equivalent conditions in Lemma \ref{lem:J2isEveryFiniteJ0}
as it is essentially finitely generated over $V$.
We know  $P_{\fm P}$ satisfies Condition \ref{condit:IsolatedPrimeInV}
as it has the same value group as $V$.
By Theorem \ref{thm:RegQCopen}, $\operatorname{Reg}(B_{\fm P})$ is open,
so there exists 
$g_2\in B\setminus \fq$ so that $(B_{\fm P})_{g_2}$ is a normal domain as the constructible topology is compact \citestacks{0901}.
Therefore, by \cite[Th.~22]{lyu-Prufer-Japanese}, $B_{g_1g_2}=B'_{g_1g_2}$ satisfies the assumptions of \cite[Lem.~23]{lyu-Prufer-Japanese} for $S=P\setminus\fm P$ the set of primitive polynomials in $P$ and $T=V\setminus\{0\}$,
so $B_{g_1g_2}$ is normal.
\end{proof}

 \begin{Cor}\label{cor:RegNorQCopenForFlat}
     Let $V$ be a valuation ring  that satisfies the equivalent conditions in Theorem \ref{thm:RegQCopen}.
     Let $X$ be a flat $V$-scheme of finite presentation.
     Then $\operatorname{Reg}(X)$ and $\operatorname{Nor}(X)$ are quasi-compact opens.
 \end{Cor}
\begin{proof}
By Theorem \ref{thm:ReducedLocusQCOpen} we may assume $\cO_{X,x}$ is an integral domain for all $x\in X$, so the irreducible components of $X$ are disjoint.
By flatness, irreducible components of $X$ are in one-to-one correspondence with  irreducible components of the Noetherian scheme $X\times_{\Spec V}\Spec(\Frac V)$ (cf. \citestacks{00HS}),
hence there are only finitely many.
Therefore we may assume $X$ is integral, and we conclude by Theorem \ref{thm:RegQCopen} and Corollary \ref{cor:NorQCopen}.
\end{proof}

 \begin{Cor}\label{cor:RegNorQCopenForFiniteRank}
     Let $V$ be a valuation ring of finite rank.
     %that satisfies the equivalent conditions in Theorem \ref{thm:RegQCopen}.
     Let $X$ be a flat $V$-scheme of finite presentation, or an integral $V$-scheme of finite type.
     Then $\operatorname{Reg}(X)$ and $\operatorname{Nor}(X)$ are quasi-compact opens.
 \end{Cor}
 \begin{proof}
 $V$ satisfies the equivalent conditions in Lemma \ref{lem:J2isEveryFiniteJ0} by Lemma \ref{lem:J1.5classes}\eqref{J1.5:frk}.
    Condition \ref{condit:IsolatedPrimeInV} is trivial for $V$ as $\Spec V$ is finite.
    We conclude by Theorem \ref{thm:RegQCopen},  Corollary \ref{cor:NorQCopen}, and Corollary \ref{cor:RegNorQCopenForFlat}.
 \end{proof}
\begin{Rem}\label{rem:RegNorQCopenForWellOrdered}
    The proof tells us Corollary \ref{cor:RegNorQCopenForFiniteRank} is true for any valuation ring $V$ such that for every $\fp\in\Spec V$ there exists an $f\not\in\fp$ so that $V_\fp=V_f$,
    or equivalently, $\Spec V$ is well-ordered by specialization, that is, there is no infinite descending chains $\fp_0\supsetneq\fp_1\supsetneq\ldots$.
    In particular, the value group of $V$ can be $\bigoplus_{\beta<\alpha}S_\beta$ in lexicographic order where $\alpha$ is an arbitrary ordinal and $S_\beta$ in an arbitrary subgroup of $\bR$ for every $\beta<\alpha$.
\end{Rem}

 \section{Big Cohen--Macaulay modules and algebras}\label{sec:BCM}
In this section, we consider a condition (Theorem \ref{thm:BCMisflat}\ref{BCM=flat:BCM}) that resembles (balanced) BCM modules over Noetherian rings.
For a local ring $(A,\fm_A)$ essentially finitely presented over a valuation ring $(V,\fm_V)$, we replace (part of) a system of parameters by (part of) a system of parameters of the fiber plus a nonzero element of $V$.
We show that when $\operatorname{w.dim}A<\infty$ this condition characterizes flatness, Theorem \ref{thm:BCMisflat}.

Our condition behaves well with respect to approximation, so we can use BCM algebras for Noetherian local rings \cite{Bhatt-int-clos-BCM,BMPSTWW} and a naive Noetherian approximation to get desirable ``BCM algebras'' in our case, see \S\ref{subsec:ConstructBCM}.
Such construction is essential in the proof of the direct summand and Kunz's theorems in our case in \S\ref{sec:SplinterandKunz}.

\subsection{Systems of parameters}

\begin{Def}\label{def:relativeSOP}
    Let $R$ be a ring and $S$ an essentially finitely generated $R$-algebra.
    Let $Q\in\Spec S, P=Q\cap R$.
    We say a sequence of elements $\underline{y}=y_1,y_2,\ldots,y_m\in Q$ a \emph{system of parameters at $Q$ over $R$}
    if $\underline{y}$ is a system of parameters of the Noetherian local ring $T:=S_Q/PS_Q$.
    We say an element $y\in Q$ is \emph{a parameter at $Q$ over $R$} if it is part of a system of parameters at $Q$ over $R$.

A sequence of elements $\underline{y}=y_1,y_2,\ldots,y_m\in Q$ is part of a system of parameters at $Q$ over $R$ if and only if $\dim T/\underline{y}T=\dim T-m$,
if and only if for all $1\leq i_1<\ldots< i_l\leq m$ we have $\dim T/(y_{i_1},\ldots,y_{i_l})\leq\dim T-l$, cf. \citestacks{00KW}.    
\end{Def}

\begin{Lem}\label{lem:CombineSOP}
    Let $R$ be a ring and $S$ an essentially finitely generated $R$-algebra.
    Let $Q\in\Spec S, P=Q\cap R$.
    Let $\underline{x}=x_1,\ldots,x_n\in P$ be part of a system of parameters of $R_P$ and let $\underline{y}=y_1,y_2,\ldots,y_m\in Q$ be part of a system of parameters at $Q$ over $R$.

    If $R_P$ is Noetherian and $R_P\to S_Q$ is flat, then $\underline{x},\underline{y}$ is part of a system of parameters of $S_Q$.
\end{Lem}
\begin{proof}
    We have $\dim R_P/\underline{x}R_P=\dim R_P-n$ and $\dim T/\underline{y}T=\dim T-m$ by assumptions, where $T=S_Q/PS_Q$.
    Hence $\dim S_Q/(\underline{x},\underline{y})\leq \dim R_P-n+\dim T-m$ \citestacks{00OM}.
    On the other hand, $\dim S_Q=\dim R_P+\dim T$ by flatness \citestacks{00ON}.
    Therefore
    $\underline{x},\underline{y}$ is part of a system of parameters of $S_Q$.
\end{proof}

\begin{Lem}\label{lem:baseChangeSOP}
Let
    \[\begin{CD}
       R@>{\varphi}>> S\\
       @VVV @VVV\\
       R'@>>> S'
    \end{CD}\]
be a cocartesian diagram of rings in which $\varphi$ is essentially finitely generated.
Let $Q'\in\Spec S$ and let $Q=Q'\cap S$.
A sequence of elements $\underline{y}=y_1,y_2,\ldots,y_m\in Q$ is part of a system of parameters at $Q$ over $R$ if and only if it is part of a system of parameters at $Q'$ over $R'$.
\end{Lem} 
\begin{proof}
Let $P'=Q'\cap R'$ and let $P=P'\cap R$.
    Observe that $S'_{Q'}/P'S'_{Q'}$ is a localization of $S'\otimes_{R'}\kappa(P')=S\otimes_R \kappa(P')$, which is flat over $S\otimes_R \kappa(P)$ as the field extension $\kappa(P')/\kappa(P)$ is flat.
    Therefore $S'_{Q'}/P'S'_{Q'}$ is flat over $S_Q/PS_Q$, and our result follows from \citestacks{00ON}.
\end{proof}

\subsection{Flatness of BCM modules}
It is classical that a balanced BCM module over a regular local ring is flat \cite[(6.7)]{HH92-BCMisFlat}. In our situation we have
\begin{Thm}\label{thm:BCMisflat}
    Let $(V,\fm_V)$ be a valuation ring and let $(A,\fm_A)$ be a local ring essentially finitely presented over $V$.
    Let $M$ be an $A$-module.
    
    Assume that $\operatorname{w.dim}A<\infty$. Then the following are equivalent.
    \begin{enumerate}[label=$(\roman*)$]
        \item\label{BCM=flat:BCM} For every sequence $\underline{x}$ in $\fm_A$ that is part of a system of parameters at $\fm_A$ over $V$ and every $0\neq \pi\in V$,
        $\underline{x},\pi$ is a possibly improper regular sequence in $M$.
        \item\label{BCM=flat:flat} $M$ is flat over $A$.
    \end{enumerate}
\end{Thm}
\begin{proof}
As seen in Discussion \ref{discu:NorRegloci}, $A$ is an integral domain.
We may assume $V\to A$ is flat local (cf. Discussion \ref{discu:DomainsOverPrufer}).
$\overline{A}:=A/\fm_V A$ is a (Noetherian) Cohen--Macaulay local ring \cite[Th.~2.9]{Knaf-RLR-over-Prufer}. 
 
 To show \ref{BCM=flat:flat} implies \ref{BCM=flat:BCM},
 we may assume $M=A$.
 We know $\underline{x}$ is a regular sequence in $\overline{A}$ as $\overline{A}$ is Cohen--Macaulay, hence $\underline{x}$ is a regular sequence in $A$ and $A/\underline{x}A$ is flat over $V$ by the local criterion for flatness \citestacks{0470}.
 Therefore $\underline{x},\pi$ is a regular sequence in $A$ (improper if $\pi\not\in\fm_V$).

 Assume \ref{BCM=flat:BCM}. By \S\ref{subsec:ReducedToSmallV}, Lemma \ref{lem:baseChangeSOP}, and \citestacks{05UU},
 in proving \ref{BCM=flat:flat} 
 we may assume $V$ has finite rank, in particular $\dim A<\infty$.
 %, and there exists $\pi\in V$ so that $\fm_V=\sqrt{\pi V}$.
 We perform induction on $\dim A$.
 By \citetwostacks{05IR}{045U}, $A$ is the localization of a flat finitely presented $V$-algebra $B$ such that the fibers of $V\to B$ are Cohen--Macaulay.
 Lemma \ref{lem:LiftSOP} below and the induction hypothesis then ensure that $M_\fp$ is flat over $A_\fp$ for all non-maximal $\fp\in\Spec A$.
  Therefore $M_f$ is flat over $A_f$ for all $f\in\fm_A$.
 
 Let $m=\dim \overline{A}$ and let $x_1,\ldots,x_m$ be elements in $\fm_A$ that are linearly independent in $\fm_A\overline{A}/\fm^2_A\overline{A}$ and form a regular sequence in $\overline{A}$.
    Then $\underline{x}$ is a regular sequence in $A$ and $A/\underline{x}A$ is flat over $V$ \citestacks{0470}.
    By \cite[Th.~2]{Knaf-RLR-over-Prufer},
    $\operatorname{w.dim}(A/\underline{x}A)<\infty$, hence $A/\underline{x}A$ is a valuation ring 
    whose value group is a finite index extension of that of $V$
    (\cite[Cor.~4.3]{Bertin-def-Regular} and Lemma \ref{lem:QFNormal=Val}).
By \ref{BCM=flat:BCM}, $M/\underline{x}M=M\otimes^L_{A}A/\underline{x}A$ is flat over $V$, hence flat over  $A/\underline{x}A$.
Therefore $M$ is flat \citestacks{0H7P}.
\end{proof}
The following lemma is the key technical point in the proof of the theorem above.
The statement and proof are inspired by \cite[Lem.~A.2]{MurFinj}.
\begin{Lem}\label{lem:LiftSOP}
    Let $V$ be a valuation ring and let $B$ be a flat finitely presented $V$-algebra.
    Assume that the fibers of $V\to B$ are Cohen--Macaulay.
    Let $\fq\subseteq\fQ$ be prime ideals of $B$, and let $\underline{y}=y_1,\ldots,y_n\in\fq$ be part of a system of parameters at $\fq$ over $V$.
    Then there exists $\underline{z}=z_1,\ldots,z_n\in\fQ$, part of a system of parameters at $\fQ$ over $V$,
    such that $(z_1,\ldots,z_j)B_\fq=(y_1,\ldots,y_j)B_\fq$ for all $1\leq j\leq n$.
\end{Lem}
\begin{proof}
Let $\fp=\fq\cap V$ and $\fP=\fQ\cap V$.
    Note that any such choice of $\underline{z}$ is a regular sequence in $B_\fQ/\fP B_\fQ$,
    hence a regular sequence in $B_\fQ$ and $B_\fQ/\underline{z}B_\fQ$ is flat over $V$ \citestacks{0470}.
    As $B_\fQ/((\underline{z})+\fP) B_\fQ$ is also Cohen--Macaulay, we know after a principal localization near $\fQ$, $V\to B/\underline{z}B$ is flat with Cohen--Macaulay fibers, see \citetwostacks{05IR}{045U}.
Therefore, an easy induction tells us it suffices to prove the result in the case $n=1$, and we write $y=y_1$.
By part of the same reasoning we know $B_\fq/yB_\fq$ is flat over $V$.

By \citestacks{00RL} we may assume all fibers of $V\to B$ are equidimensional of the same dimension $d$.
Consider $J=yB_\fq\cap B$.
Then $B/J\subseteq B_\fq/yB_\fq$, hence $B/J$ is flat over $V$.
Let $K=\Frac V$.
Then $J=JB_K\cap B$ as $B/J$ is flat over $V$.
As $B_K$ is Noetherian, it now follows from primary decomposition of $JB_K\subseteq B_K$ that 
\[
J=\fr_1\cap \fr_2\cap\ldots\cap\fr_s
\]
where $\fr_i$ are primary ideals of $B$ that satisfy $\sqrt{\fr_i}=\sqrt{\fr_i}B_K\cap B$, 
hence $B/\sqrt{\fr_i}$ is flat over $V$.
As $J=JB_\fq\cap B$ we may assume $\fr_i\subseteq\fq$ for all $i$.
Similarly, as $B/\fP B$ is flat over $V/\fP$,
we have
\[
\sqrt{\fP B}=\fN_1\cap \fN_2\cap\ldots\cap\fN_t
\]
where $\fN_j$ are minimal prime divisors of $\fP B$.

By \citestacks{00QK}, all fibers of $V\to B/\sqrt{\fr_i}$ are equidimensional of the same dimension $d_i$.
As $y\in J\subseteq \sqrt{\fr_i}\subseteq\fq$ is a parameter at $\fq$ over $V$, we see $d_i<d$.
It therefore follows that $\sqrt{\fr_i}\not\subseteq\fN_j$ for any $i,j$,
as $B/\fN_j\otimes_{V}\kappa(\fP)$ has dimension $d$.

Note that we also have $\fQ\not\subseteq\fN_j$ for any $j$, because otherwise $\fQ=\fN_j$ and $\Spec(B_\fQ)\to\Spec V$ would be injective by Lemma \ref{lem:QFNormal=Val}, so $\dim (B_\fq/\fp B_\fq)=0$ and $y$ could not exist.
Therefore $J^2\cap \fQ\not\subseteq\fN_j$ for any $j$ as $\fN_j$ is prime.
By prime avoidance and a theorem of Davis \cite[Th.~124]{kaplansky}, there exists $b\in J^2\cap \fQ$ so that $z:=y+b\not\in \fN_j$ for any $j$,
so $z\in\fQ$ is a parameter at $\fQ$ over $V$.
Finally, the choice of $J$ tells us $bB_\fq\subseteq y^2B_\fq$, so $z=y(1+yc)\in B_\fq$ for some $c\in B_\fq$, therefore $zB_\fq=yB_\fq$ as $1+yc$ is invertible.
\end{proof}

Theorem \ref{thm:BCMisflat} has the following consequence via standard dimension shifting.
\begin{Cor}\label{cor:BCMisTorDim1}
    Let $(V,\fm_V)$ be a valuation ring and let $(A,\fm_A)$ be a local ring essentially finitely presented over $V$.
    Let $M$ be an $A$-module.
    
    Assume that $\operatorname{w.dim}A<\infty$
    and assume that every sequence $\underline{x}$ in $\fm_A$ that is part of a system of parameters at $\fm_A$ over $V$ is a possibly improper regular sequence in $M$.
    Then $\operatorname{w.dim}_AM\leq 1$.
\end{Cor}
\begin{proof}
    Let $F$ be a flat $A$-module mapping surjectively to $M$.
    Then $\ker(F\to M)$ is flat by Theorem \ref{thm:BCMisflat}.
\end{proof}

\subsection{Construction of BCM algebras}\label{subsec:ConstructBCM}
We give our construction in the case we have a local domain essentially finitely presented flat over a valuation ring of positive or mixed characteristic.
This case is enough for the direct summand and perfectoid Kunz's theorems below.
\begin{Discu}\label{discu:BCMby+}
    Let $(V,\fm_V)\to (A,\fm_A)$ be an essentially finitely presented flat local map of local rings.
    Assume that $V$ is a valuation ring, and assume there exists a prime number $p$ not invertible in $V$.
    
    Write $V=\bigcup_{\lambda\in\Lambda}R_{\lambda}$,
    where the union is filtered and all $R_{\lambda}$ are excellent Noetherian local rings (for example essentially finitely generated over $\bZ$) and the transition maps are local.

    Assume $A$ is an integral domain.
    As seen in Discussion \ref{discu:DomainsOverPrufer},
    $A=B_\fq$ for some integral domain $B$ flat and finitely presented over $V$.
    By \citestacks{02JO}, there exists a $\lambda_0\in \Lambda$ and a flat finitely presented $R_{\lambda_0}$-algebra $B_{\lambda_0}$ whose base change to $V$ is $B$.
    Let $B_{\lambda}=B_{\lambda_0}\otimes_{R_{\lambda_0}}R_{\lambda}$.
    As each $R_{\lambda}$ is a subring of $V$ and as $B_{\lambda_0}$ is flat over $R_{\lambda_0}$,
    we see $B=\bigcup_{\lambda\geq\lambda_0}B_{\lambda}$,
    in particular every $B_{\lambda}$ is an integral domain.
    Therefore $A=\bigcup_{\lambda\geq\lambda_0}A_{\lambda}$, where $A_\lambda=(B_{\lambda})_{\fq\cap B_\lambda}$.
    We also write $\fm_\lambda$ for the maximal ideal of $A_\lambda$.

    We consider the absolute integral closure of the rings $A_\lambda$ inside a fixed algebraic closure of $\Frac A$.
    This gives us $A^+=\bigcup_{\lambda\geq\lambda_0}A_{\lambda}^+$.
    We therefore have a commutative diagram
    \[
    \begin{CD}
        R_\lambda @>>> A_\lambda @>>> A_{\lambda}^+ @>>> \widehat{A_{\lambda}^+}^p\\
        @VVV @VVV @VVV @VVV\\
        V @>>> A @>>> A^+ @>>> \colim_{\lambda\geq\lambda_0}\widehat{A_{\lambda}^+}^p,
    \end{CD}
    \]
    where $\widehat{(-)}^p$ is the $p$-adic completion.

\end{Discu}
\begin{Discu}\label{discu:BCMby+CONCLUSIONS}
Continuing Discussion \ref{discu:BCMby+},
    let $C=\colim_{\lambda\geq\lambda_0}\widehat{A_{\lambda}^+}^p$.
    Consider a sequence $\underline{x}$ in $\fm_A$ that is part of a system of parameters at $\fm_A$ over $V$ (Definition \ref{def:relativeSOP}),
    and any element $\pi\in\fm_V$.
    There exists an index $\lambda_1$ so that $\underline{x}$ is in the subring $A_{\lambda_1}$ and that $\pi\in R_{\lambda_1}$.
    As $A$ is a localization of $V\otimes_{R_{\lambda}}A_{\lambda}$,
    it follows from Lemma \ref{lem:baseChangeSOP} that $\underline{x}$ is part of a system of parameters at $\fm_\lambda$ over $R_\lambda$ for all $\lambda\geq\lambda_1$,
    so Lemma \ref{lem:CombineSOP} tells us $\underline{x},\pi$ is part of a system of parameters in $A_\lambda$.

    By \cite[Cor.~2.10]{BMPSTWW}, all permutations of the sequence $\underline{x},\pi$ are a regular sequence in $\widehat{A_{\lambda}^+}^p$ for all $\lambda\geq\lambda_1$,
    hence all permutations of the sequence $\underline{x},\pi$ are a regular sequence in $C$.
    Furthermore, $\fm_A=\bigcup_{\lambda\geq\lambda_0}\fm_\lambda$,
    so $$C/\fm_AC=\colim_{\lambda\geq\lambda_0}\widehat{A_{\lambda}^+}^p/\fm_\lambda\widehat{A_{\lambda}^+}^p=\colim_{\lambda\geq\lambda_0}{A_{\lambda}^+}/\fm_\lambda{A_{\lambda}^+}=A^+/\fm_A A^+\neq 0,$$
    where we used that $p\in\fm_\lambda$ (as $p\in\fm_V$) and that the fiber $\Spec(A^+/\fm_A A^+)$ is nonempty \citestacks{00GQ}.
\end{Discu}

\begin{Rem}
    We have taken full advantage of the fact that the maps $A_\lambda\to A$ are injective, so we have full functoriality of the absolute integral closure, and $C$ is well-defined.
    We could also use different BCM algebras for each $A_{\lambda}$ with no functoriality at all, by replacing the colimit with an ultralimit.
    Further discussion in this direction may appear in future work.
\end{Rem}

\section{Direct summand and Kunz's theorems}\label{sec:SplinterandKunz}

\subsection{The direct summand theorem over valuation rings}

Recall
\begin{Def}[cf. {\cite[Def.~2.1]{Lyu-Splinter}}]\label{def:Splinters}
    A ring $R$ is a \emph{splinter} if every finite ring map $R\to S$ that induces a surjective map on spectra is pure, \emph{i.e.}, the $R$-module map $R\to S$ is universally injective \citestacks{058I}.
\end{Def}

The following result seems to be well-known.
\begin{Lem}\label{lem:Splinter=R+pure}
    Let $R$ be an integral domain.
    Then $R$ is a splinter if and only if the map $R\to R^+$ is pure,
    where $R^+$ is an absolute integral closure of $R$.
\end{Lem}
\begin{proof}
Assume $R$ is a splinter.
    As $R^+$ is integral over $R$, it is the filtered union of subalgebras $S$ finite over $R$,
    and $\Spec S\to\Spec R$ is surjective as $R\to S$ is injective and finite \citestacks{00GQ}.
    As filtered colimits are exact and commute with tensor products,
    we see $R\to R^+$ is pure, cf. \citestacks{058J}.

    Conversely, assume $R\to R^+$ is pure.
    Let $R\to S$ be a finite ring map that induces a surjective map on spectra.
    Let $\fq\in\Spec S$ be above $0\in\Spec R$.
    Then $S/\fq$ is isomorphic to a subalgebra of $R^+$,
    hence $R\to S$ is pure as the composition $R\to S\to S/\fq\to R^+$ is.
\end{proof}

\begin{Thm}\label{thm:RLRisSplinter}
    Let $V$ be a valuation ring and let $A$ be a local ring essentially finitely presented over $V$.
    If $\operatorname{w.dim}A<\infty$,
    then $A$ is a splinter.
\end{Thm}
\begin{proof}
    By \cite[Cor.~4.3]{Bertin-def-Regular}, $A$ is a normal domain,
    so the result is true if $A$ contains $\bQ$, see for example \cite[Lemmas 2 and 3]{Hoc73-splinter}.
    We may now assume $A$ does not contain $\bQ$.
    We may assume $V\to A$ is flat and local.
    As $A$ does not contain $\bQ$,
    there exists a prime number $p$ not invertible in $V$.
    We may therefore construct the algebra $C$ as in Discussion \ref{discu:BCMby+}.

    As $\operatorname{w.dim}A<\infty,$
    Theorem \ref{thm:BCMisflat} and
 Discussion \ref{discu:BCMby+CONCLUSIONS} tells us $C$ is faithfully flat over $A$.
    Therefore $A\to C$ is pure \citestacks{05CK},
    hence so is $A\to A^+$ as we have a factorization $A\to A^+\to C$.
    We conclude by Lemma \ref{lem:Splinter=R+pure}.
\end{proof}

\begin{Ques}
    Is every coherent local ring, regular in the sense of Bertin \cite{Bertin-def-Regular}, a splinter?
\end{Ques}

\subsection{Perfectoid Kunz's thoerem over valuation rings}
In \cite{Kunz-Bhatt-Ma} the relation between regularity of a Noetherian ring and perfetoid algebras over the ring is discussed. To get perfectoid rings we need to take completion.
\begin{Lem}[cf. {\cite[Th.~2.9]{BMPSTWW}}]\label{lem:RegSeqStaysInCompletion}
    Let $R$ be a ring, $x_1,\ldots,x_m,z\in R$.
    %Assume that $z\in\sqrt{yR}$.

    Let $M$ be an $R$-module such that all permutations of $x_1,\ldots,x_m,z$ are  $M$-regular sequences.
    Then all permutations of $x_1,\ldots,x_m,z$ are $\widehat{M}^z$-regular sequences, where $\widehat{M}^z$ is the $z$-adic completion of $M$.
\end{Lem}
\begin{proof}
    First, let us show $z$ is a regular element on $\widehat{M}^z$.
    Let $a_n$ be a sequence in $M$ so that $\lim_n za_n=0$ in the $z$-adic topology.
    Then for fixed $N$ and large $n$ we have $za_n\in z^NM$, so $a_n\in z^{N-1}M$ as $z$ is a regular element on $M$.
    Therefore $\lim_n a_n=0$, so $z$ is a regular element on $\widehat{M}^z$.
    This finishes the proof in the case $m=0$.
    We may therefore assume, by induction on $m$,
    that our results hold for all smaller $m$.

    Next, %let us show $x_1\widehat{M}^z$ is closed in $\widehat{M}^z$.
    %By completeness, it suffices to show the $z$-adic topology on $x_1\widehat{M}^z$ is the same as that induced from $\widehat{M}^z$.
    %Hence it suffices to 
    We show for all $c$,
    \begin{align}\label{Completion:artinrees}
        x_1\widehat{M}^z\cap z^c\widehat{M}^z=z^cx_1\widehat{M}^z.
    \end{align}
    As ``$\supseteq$'' is trivial we only need to show ``$\subseteq$.''
    A typical element in $x_1\widehat{M}^z$ %\cap z^c\widehat{M}^z$ 
    is of the form $a=x_1\lim_n a_n=\lim_n x_1a_n$, where $a_n\in M$ form a Cauchy sequence.
    If $a\in z^c \widehat{M}^z$,
    then $a=\lim_n z^cb_n$ for some Cauchy sequence $(b_n)_n$.
    Passing to subsequences we may assume $x_1a_n-z^cb_n\in z^{n+c}M$,
    in particular $x_1a_n\in z^cM$.
    As $z,x_1$ is an $M$-regular sequence by assumption, so is $z^c,x_1$ \citestacks{07DV},
    hence $a_n\in z^cM,$
    say $a_n=z^ca'_n$.
    As $z$ is a nonzerodivisor on $M$ and as $(a_n)_n$ is a Cauchy sequence,
    $(a'_n)_n$ is a Cauchy sequence.
    Therefore $a=z^cx_1\lim_n a'_n\in z^cx_1\widehat{M}^z$, as desired.
    Moreover, the same argument tells us if $a=0$,
    then for every $c$, $a_n\in z^c M$ for all large $n$,
    in other words, $\lim_n a_n=0$,
    so $x_1$ is a nonzerodivisor on $\widehat{M}^z$.

    By \eqref{Completion:artinrees}, $x_1\widehat{M}^z$ is closed in $\widehat{M}^z$, hence
    $\widehat{M}^z/x_1\widehat{M}^z$ is  the $z$-adic completion of $M/x_1M$.
    By the induction hypothesis,
    we now know that every permutation of $x_1,\ldots,x_m,z$ that still starts with $x_1$ is an $\widehat{M}^z$-regular sequence.
    By symmetry,
    every permutation of $x_1,\ldots,x_m,z$ that does not start with $z$ is an $\widehat{M}^z$-regular sequence.
    On the other hand, as $z$ is a nonzerodivisor on $\widehat{M}^z$ and as $\widehat{M}^z/z\widehat{M}^z=M/zM$ \citestacks{05GG},
    every permutation of $x_1,\ldots,x_m,z$ that starts with $z$ is an $\widehat{M}^z$-regular sequence,
    and the result follows.
\end{proof}

Now we can show the following characterization.
\begin{Thm}\label{thm:PerfdKunz}
        Let $(V,\fm_V,k_V)$ be a valuation ring  and let $(A,\fm_A,k_A)$ be an essentially finitely presented local $V$-algebra of residue characteristic $p$.

    Then the following are equivalent.
    \begin{enumerate}[label=$(\roman*)$]
        \item\label{KunzPFD-fwgdim} $\operatorname{w.dim} A<\infty$.
        %\item\label{KunzPFD-RLR} For every finitely generated ideal $I$ of $A$, $\operatorname{w.dim}_AA/I<\infty$; that is, $A$ is regular in the sense of Bertin \cite[D\'ef.~3.5]{Bertin-def-Regular}.
        \item\label{KunzPFD-plus-flat} $A$ is an integral domain and 
        \begin{align*}
            \operatorname{w.dim}_A\widehat{A^+}^p=\begin{cases}
                0 & \text{the image of } V_{\fm_A\cap V}\text{ in }A\text{ is }p\text{-adically separated}\\
                %\footnote{Equivalently, the radical of the ideal generated by $p$ has height $\leq 1$.}\\
                1 & \text{otherwise}
            \end{cases}
        \end{align*}
        %is a  perfectoid algebra flat over $A$.
        \item\label{KunzPFD-plus-ffd} $A$ is an integral domain and  $\operatorname{w.dim}_A\widehat{A^+}^p<\infty$.
        \item\label{KunzPFD-some-perf-ffd} There exists an $A$-algebra $S$ which is perfectoid,
        and an $S$-module $U$
        that satisfies $U/\sqrt{\fm_A S}U\neq 0$ and $\operatorname{w.dim}_AU<\infty$. 
    \end{enumerate}
\end{Thm}
We see \ref{KunzPFD-plus-flat} trivially implies \ref{KunzPFD-plus-ffd},
and \ref{KunzPFD-plus-ffd} implies \ref{KunzPFD-some-perf-ffd} as $\widehat{A^+}^p$ is perfectoid \cite[Example 3.8(2)]{Kunz-Bhatt-Ma}
  and as $\widehat{A^+}^p/\fm_A\widehat{A^+}^p=A^+/\fm_A A^+\neq 0$, so we can take $S=U=\widehat{A^+}^p$.
It remains to show \ref{KunzPFD-some-perf-ffd} implies \ref{KunzPFD-fwgdim} and \ref{KunzPFD-fwgdim} implies \ref{KunzPFD-plus-flat}.

\begin{proof}[Proof of \ref{KunzPFD-some-perf-ffd} implies \ref{KunzPFD-fwgdim}]
Write $V=\bigcup_{\lambda}V_\lambda$ and $A=\bigcup_{\lambda\geq\lambda_0}(B_{\lambda})_{\fm_A \cap B_{\lambda}}$ as in \S\ref{subsec:ReducedToSmallV},
so every $(B_{\lambda})_{\fm_A \cap B_{\lambda}}$ satisfies \ref{KunzPFD-some-perf-ffd}.
If we can show every $(B_{\lambda})_{\fm_A \cap B_{\lambda}}$ satisfies \ref{KunzPFD-fwgdim},
then every coherent $A$-module has finite tor dimension as the transition maps are flat, see \cite[Th.~6.2.2]{Coherent-Rings},
so $\operatorname{w.dim}A<\infty$, see Discussion \ref{discu:NorRegloci}.
Therefore, we may replace $V$ by $V_\lambda$ and $A$ by $(B_\lambda)_{\fm_A\cap B_\lambda}$
to assume $V$ has finite rank,
in particular
$\fm_A=\sqrt{\fa}$ for some finitely generated ideal $\fa$ of $A$ (as $\Spec A$ is a Noetherian topological space),
and $\sqrt{\fm_A S}=\sqrt{\fa S}$.

By \cite[Lem.~3.7]{Kunz-Bhatt-Ma},
$J:=\sqrt{\fa S}\supseteq\fm_AS$ has finite tor dimension over $S$,
and $U/JU\neq 0$ by assumption.
As $\operatorname{w.dim}_AU<\infty$,
it follows from \cite[Th.~2.1]{Kunz-Bhatt-Ma}\footnote{The statement includes the assumption $R$ is Noetherian, which is not necessary for the proof.} that $k_A\otimes^L_A M$ is bounded for all $M\in\Coh(A)$.
By Theorem \ref{thm:ALLwdimFacts}\ref{wdim:testk},
    $M$ has finite tor dimension,
    as desired.
\end{proof}

\begin{proof}[Proof of \ref{KunzPFD-fwgdim} implies \ref{KunzPFD-plus-flat}]
As seen in Discussion \ref{discu:NorRegloci}, $A$ is an integral domain.
We may assume $V\to A$ is flat local (cf. Discussion \ref{discu:DomainsOverPrufer}).

 Let $C$ be as in Discussion \ref{discu:BCMby+}.
 As seen in the proof of Theorem \ref{thm:RLRisSplinter},
 $C$ is faithfully flat over $A$,
 and $C=A^+$ if $p=0$ in $A$,
 so we are done in the characteristic $p$ case.

 Assume $p\neq 0\in A$.
 By construction, the canonical map $A^+\to \widehat{A^+}^p$ factors through $A^+\to C$,
 hence $\widehat{A^+}^p$ is a direct summand of $\widehat{C}^p$.
 Let us now show 
 \begin{align}\label{tordimC}
\operatorname{w.dim}_A\widehat{C}^p\leq\begin{cases}
                0 & V \text{ is } p \text{-adically separated}\\
                1 & \text{otherwise}
            \end{cases}
        \end{align}
        By Discussion \ref{discu:BCMby+CONCLUSIONS} and Lemma \ref{lem:RegSeqStaysInCompletion},
 for a sequence $\underline{x}$ in $\fm_A$ that is part of a system of parameters at $\fm_A$ over $V$,
 (every permutation of) $\underline{x},p$ is a $\widehat{C}^p$-regular sequence.
 We see $\operatorname{w.dim}_A\widehat{C}^p\leq1$ by Corollary \ref{cor:BCMisTorDim1}.
 If the valuation ring $V$ is  $p$-adically separated,
 then every nonzero element $\pi\in V$ divides a power of $p$,
 hence $\underline{x},\pi$ is a $\widehat{C}^p$-regular sequence.
 Thus $\widehat{C}^p$ is flat over $A$ by Theorem \ref{thm:BCMisflat}.
 \eqref{tordimC} is therefore proved.

 As $\widehat{A^+}^p$ is a direct summand of $\widehat{C}^p$,
 we see it remains to show if $\widehat{A^+}^p$ is flat over $A$ then $V$ is $p$-adically separated.
 If $\widehat{A^+}^p$ is flat over $A$,
 then it is faithfully flat over  $A$ as $\widehat{A^+}^p/\fm_A\widehat{A^+}^p\neq 0$.
 Hence $A\to \widehat{A^+}^p$ is injective, and so is $V\to \widehat{A^+}^p$,
 hence $V$ is $p$-adically separated, as desired.
\end{proof} 

\subsection{Kunz's theorem over valuation rings}
Specializing to characteristic $p>0$,
we have the parallel statement of Kunz's Theorem \cite{Kunz1969-reg-F-flat} in our case.
\begin{Thm}\label{thm:KunzALL}
    Let $(V,\fm_V,k_V)$ be a valuation ring of characteristic $p>0$ and let $(A,\fm_A,k_A)$ be an essentially finitely presented local $V$-algebra.
    Then the following are equivalent.
    \begin{enumerate}[label=$(\roman*)$]
        \item\label{KunzALL-fwgdim} $\operatorname{w.dim} A<\infty$.
        %\item\label{Kunz-RLR} For every finitely generated ideal $I$ of $A$, $\operatorname{w.dim}_A A/I<\infty$; that is, $A$ is regular in the sense of Bertin \cite[D\'ef.~3.5]{Bertin-def-Regular}.
        \item\label{KunzALL-all-F-flat} $F^e_*A$ is flat over $A$ for all $e$.
        \item\label{KunzALL-some-F-flat} $F^e_*A$ is flat over $A$ for some $e>0$.
        \item\label{KunzALL-perf-flat} The map $A\to A_{perf}=\colim_e F^e_*A$ is flat.
        \item\label{KunzALL-plus-flat}
        $A$ is an integral domain and the map $A\to A^+$ is flat.
        \item\label{KunzALL-some-perf-ffd} There exists an $A$-algebra $S$ which is perfect,
        and an $S$-module $U$
        that satisfies $U/\sqrt{\fm_A S}U\neq 0$ and $\operatorname{w.dim}_AU<\infty$. 
    \end{enumerate}
\end{Thm}
\begin{proof}
    %By Theorem \ref{thm:Kunz}, \ref{KunzALL-fwgdim}\ref{KunzALL-all-F-flat}\ref{KunzALL-some-F-ffd} are equivalent.
    By Theorem \ref{thm:PerfdKunz},
    \ref{KunzALL-fwgdim}\ref{KunzALL-plus-flat}\ref{KunzALL-some-perf-ffd} are equivalent.
    We also have, trivially,  
    \ref{KunzALL-all-F-flat} implies  \ref{KunzALL-some-F-flat} implies \ref{KunzALL-perf-flat}.
    %\ref{KunzALL-all-F-flat} implies \ref{KunzALL-perf-flat} and 
    Note that \ref{KunzALL-perf-flat} implies \ref{KunzALL-some-perf-ffd} as we can take $S=U=A_{perf}$; the condition $U/\sqrt{\fm_A S}U\neq 0$ follows from the fact $\Spec S\to \Spec A$ is a homeomorphism.
    Therefore, it remains to show \ref{KunzALL-fwgdim} implies \ref{KunzALL-all-F-flat}.

    Assume \ref{KunzALL-fwgdim}.
    As seen in Discussion \ref{discu:NorRegloci}, $A$ is a normal domain.
We may assume $V\to A$ is flat local (cf. Discussion \ref{discu:DomainsOverPrufer}).
Write $V=\bigcup_{\lambda}V_\lambda$ and $A=\bigcup_{\lambda\geq\lambda_0}(B_{\lambda})_{\fm_A \cap B_{\lambda}}$ as in \S\ref{subsec:ReducedToSmallV},
we know every $(B_\lambda)_{\fm_A\cap B_\lambda}$ satisfies \ref{KunzALL-fwgdim} by flatness,
and \ref{KunzALL-all-F-flat} for $A$ follows from \ref{KunzALL-all-F-flat} for  all $(B_\lambda)_{\fm_A\cap B_\lambda}$ as $A=\bigcup_{\lambda\geq\lambda_0} (B_\lambda)_{\fm_A\cap B_\lambda}$.
Therefore, we may replace $V$ by $V_\lambda$ and $A$ by $(B_\lambda)_{\fm_A\cap B_\lambda}$
to assume $\dim A<\infty$.
We perform induction on $\dim A$ and assume that the result is true in all smaller dimensions.

Let $\overline{A}=A/\fm_VA$, an essentially finitely generated local $k_V$-algebra.
 If $\dim \overline{A}=0$, then
$A$ is a valuation ring (Lemma \ref{lem:QFNormal=Val}).
In particular, $A\to F^e_*A$ is flat for all $e$ by Theorem \ref{thm:ModAndAlgOverV}\ref{OverV:torsion-free-is-flat}.
So we may assume %$\overline{A}$ is not finite over $k_V$, hence 
$\dim \overline{A}>0$.
Therefore $\fm_A\overline{A}\neq \fm_A^2\overline{A}$ by Nakayama's Lemma,
so in particular
we can take an element $t\in \fm_A$ that is not in $\fm_A^2$. %and is a nonzerodivisor in $\overline{A}$.
Let $B=A/tA$.
By \cite[Th.~2]{Knaf-RLR-over-Prufer}, $B$ satisfies \ref{KunzALL-fwgdim}.
As $B$ is a proper quotient of $A$ and $A$ is an integral domain, $\dim B<\dim A$,
so by induction we know $B\to F^e_*B$ is flat for all $e$.

For an arbitrary $e$, consider the commutative diagram of rings
\[\begin{CD}
A@>>> A'@>{\varphi}>> F^e_*A    \\
@. @V{\psi}VV @VVV\\
@. B@>>> F^e_*B
\end{CD}\]
where $A'=A[Y]/(Y^{p^e}-t)$ is a finite free $A$-algebra, $\varphi$ maps  $Y$ to $F^e_*t$,
and $\psi$ maps $Y$ to zero.
It is clear that $\psi$ is surjective with kernel $YA'$ and the square is cocartesian.
As $t$ is a regular element in $A$,
it is a regular element in $A'$ and thus so is $Y$;
we also have $F^e_*t$ is a regular element in $F^e_*A$.
As $B\to F^e_*B$ is flat,
we see from the local criterion \citestacks{051C} that $F^e_*A/F^e_*t^{p^e}F^e_*A=F^e_*A/tF^e_*A$ is flat over $A'/Y^{p^e}A'=A'/tA'$,
hence flat over $A/tA$.

Finally, the principal localization $A_t$ has finite weak global dimension and $\dim A_t<\dim A$ as $A$ is local.
By induction applied to the localizations of $A_t$, $A_t\to F^e_*A_t$ is flat.
We conclude that $A\to F^e_*A$ is flat by \citestacks{0H7N}.
\end{proof}

\section{Cotangent complexes, I}\label{sec:Cotangent1}
We use cohomological conventions.
Therefore, the Andr\'e--Quillen homology ($H_n(A,B,W)$ in \cite{Andre-cotangent-complex} and $\operatorname{D}_n(B\mid A,W)$ in \cite{avramov-ci}) is written as $H^{-n}(L_{B/A}\otimes^L_B W)$.
\subsection{Regular local rings have nice cotangent complexes}
The following result is due to Gabber--Romero \cite[Th.~6.5.8(ii) and Cor.~6.5.21]{GR-almost-ring-theory}.
\begin{Thm}\label{thm:Val-Has-Flat-L}
    Let $k$ be a perfect field and let $V$ be a valuation ring containing $k$.
    Then the cotangent complex $L_{V/k}$ has tor-amplitude in $[0,0]$;
    in other words, $L_{V/k}$ is quasi-isomorphic to a flat $V$-module placed in degree $0$.
\end{Thm}
It is not difficult to generalize this result to finitely presented algebras.
\begin{Thm}\label{thm:Regular-Has-Flat-L-special}
    Let $k$ be a perfect field and let $V$ be a valuation ring containing $k$.
    Let $A$ be an essentially finitely presented local $V$-algebra.
    If $\operatorname{w.dim}A<\infty$,
    then the cotangent complex $L_{A/k}$ has tor-amplitude in $[0,0]$.
\end{Thm}
\begin{proof}
    By \S\ref{subsec:ReducedToSmallV}, Theorem \ref{thm:ALLwdimFacts}\ref{wdim:flatDescent}, and the compatibility of the cotangent complex with filtered colimits,
    we may assume $V$ has finite rank, in particular $\dim A<\infty$.
    We perform induction on $\dim A$.

    As seen in Discussion \ref{discu:NorRegloci}, $A$ is an integral domain.
We may assume $V\to A$ is flat local (cf. Discussion \ref{discu:DomainsOverPrufer}).
Let $\fm_V$ (resp. $\fm_A$) be the maximal ideal of $V$ (resp. $A$) and let $\overline{A}=A/\fm_V A$.
If $\dim \overline{A}=0$ then $A$ is a valuation ring (Lemma \ref{lem:QFNormal=Val}),
so we win by Theorem \ref{thm:Val-Has-Flat-L}.
We may therefore assume $\dim \overline{A}>0$.
Since $\overline{A}$ is Noetherian, Nakayama's Lemma tells us $\fm_A\overline{A}\neq \fm_A^2\overline{A}$, in particular there exists an $a\in \fm_A\setminus\fm_A^2$.
Let $B=A/aA$, so $\operatorname{w.dim}B<\infty$ \cite[Th.~2]{Knaf-RLR-over-Prufer}.
Consider the Jacobi--Zariski triangle
\[\begin{CD}
    L_{A/k}\otimes_A^L B@>>> L_{B/k}@>>> L_{B/A}@>>>+1.
\end{CD}\]
As $A$ is an integral domain we have $L_{B/A}=(aA/a^2A)[1]$. Therefore we have a distinguished triangle
\[\begin{CD}
    aA/a^2A@>>> L_{A/k}\otimes_A^L B@>>> L_{B/k}@>>>+1.
\end{CD}\]
By the  induction hypothesis, $L_{B/k}$ has tor-amplitude in $[0,0]$, hence so does $L_{A/k}\otimes_A^L B$.
On the other hand, $L_{A/k}\otimes_A^L A_a=L_{A_a/k}$ also has tor-amplitude in $[0,0]$ by the  induction hypothesis.
Therefore $L_{A/k}$ has tor-amplitude in $[0,0]$ by \citestacks{0H85}.
\end{proof}

We can generalize Theorem \ref{thm:Regular-Has-Flat-L-special} a bit more.
Note that in the theorem below we do not require $V$ to contain $k$.
\begin{Thm}\label{thm:Regular-Has-Flat-L}
Let $(V,\fm_V,k_V)$ be a valuation ring,
$(A,\fm_A,k_A)$ an essentially finitely presented $V$-algebra.
    Assume $A$ contains a field $k$ of characteristic $p\geq 0$.
    
    Assume that $p>0$ and $\operatorname{w.dim}(A\otimes_k F)<\infty$ for every finite extension $F/k$ inside $k^{1/p}$,
    or that $p=0$ and $\operatorname{w.dim}A<\infty$.
    Then the cotangent complex $L_{A/k}$ has tor-amplitude in $[0,0]$.
\end{Thm}
\begin{proof}
We may replace $V$ by a localization of the image of $V$ in $A$ (Discussion \ref{discu:DomainsOverPrufer}) to assume $V\to A$ is flat and local, in particular $V$ contains a field.

First, we reduce to the case $k$ is finitely generated (over the prime field).
If $p=0$ then let $(k_{\alpha})_\alpha$ be the collection of all finitely generated subfields of $k$;
if $p>0$ then fix a $p$-basis $\Lambda$ of $k$ and let $(k_{\alpha})_\alpha$ be the collection of subfields of the form $\bF_p(\alpha)$ where $\alpha$ is a finite subset of $\Lambda$.
Then the family $(k_{\alpha})_\alpha$ is filtered, the inclusion maps are separable, and $L_{k/\cup_\alpha k_\alpha}=0$.
Then $k_\alpha\to A$ satisfies the same assumptions as, when $p>0$, for a finite extension $F_\alpha$ of $k_\alpha$ in $k_\alpha^{1/p}$, $F_\alpha\otimes_{k_\alpha}k$ is a finite extension of $k$ in $k^{1/p}$.
Moreover $H^i(L_{A/k}\otimes_A^LW)=\colim_\alpha H^i(L_{A/k_\alpha}\otimes_A^LW)$ for all $A$-modules $W$.
Therefore it suffices to show the result for $k_\alpha\to A$, hence we may assume $k$ is finitely generated.

Let $k_0$ be the prime subfield of $k$ and examine the Jacobi--Zariski triangle
\[\begin{CD}
    L_{k/k_0}\otimes_k^L A@>>> L_{A/k_0}@>>> L_{A/k}@>>>+1.
\end{CD}\]
As $k_0$ is perfect and $k$ is finitely generated we know $L_{k/k_0}$ is represented by a finite free module in degree $0$.
We also know $L_{A/k_0}$ is represented by a flat module in degree $0$ by Theorem \ref{thm:Regular-Has-Flat-L-special}, since $k_0$ is automatically contained in $V$.
By \cite[Prop.~3.4.36]{Epstein-Datta-Tucker-all-chars-OHM-RUSH} %(an alternative version of \citestacks{00ME})
it therefore suffices to show $H^{-1}(L_{A/k}\otimes_A^Lk_A)=0$.
If $p=0$, then $k_A/k$ is separable, so %this follows 
from the Jacobi--Zariski triangle $L_{A/k}\otimes_A^Lk_A\to L_{k_A/k}\to L_{k_A/A}\to +1$
we get $H^{-1}(L_{A/k}\otimes_A^Lk_A)=H^{-2}(L_{k_A/A})$.
As $k_A/k_0$ is also separable we have
$0=H^{-1}(L_{A/k_0}\otimes_A^Lk_A)=H^{-2}(L_{k_A/A})$, as desired. %and $H^{-1}(L_{A/k_0}\otimes_A^Lk_A)=0$
We may therefore assume $p>0$.

Let $A_F=A\otimes_k F$ for a field extension $F/k$ and let $k_{A_F}$ be the residue field of $A_F$.
We write $A_1$ for $A_{k^{1/p}}$ and $k_{A_1}$ for the residue field of $A_1$.
We need to show $H^{-1}(L_{A/k}\otimes_A^Lk_A)=0$,
which is equivalent to $H^{-1}(L_{A/k}\otimes_A^Lk_{A_1})=0$.
By 
\cite[Th.~7.26]{Andre-cotangent-complex} we have
\[
H^{-1}(L_{A/k}\otimes_A^Lk_{A_1})=H^{-2}(L_{k_{A_1}/A_1})=\underset{F\subseteq k^{1/p}\text{ finite}}{\operatorname{colim}} H^{-2}(L_{k_{A_F}/A_F})\otimes_{k_{A_F}}k_{A_1}.
\]
Therefore, it suffices to show $H^{-2}(L_{k_{A_F}/A_F})=0$ for all $F\subseteq k^{1/p}$ finite.
Since $\operatorname{w.dim}A_F<\infty$ and $k_{A_F}$ is separable over $k_0=\bF_p$ we have, similar to the $p=0$ case,
$0=H^{-1}(L_{A_F/k_0}\otimes_{A_F}^Lk_{A_F})=H^{-2}(L_{k_{A_F}/A_F})$,
as desired.
\end{proof}

 \begin{Rem}
     Let us give a quick explanation of the isomorphism $H^{-1}(L_{A/k}\otimes_A^Lk_{A_1})=H^{-2}(L_{k_{A_1}/A_1})$ from
     \cite[Th.~7.26]{Andre-cotangent-complex}.
     We are looking at the Jacobi--Zariski triangle (noting that $L_{A/k}\otimes_A^LA_1=L_{A_1/k^{1/p}}$)
     \[\begin{CD}
     L_{A/k}\otimes_A^L k_{A_1}@>>> L_{k_{A_1}/k^{1/p}}@>>> L_{k_{A_1}/A_1}@>>>+1.
 \end{CD}\]
As $L_{k_{A_1}/k^{1/p}}$ is in $D^{[-1,0]}$ it suffices to show the canonical map $H^{-1}(L_{A/k}\otimes_A^L k_{A_1})\to H^{-1}(L_{k_{A_1}/k^{1/p}})$ is zero.
This map factors through the canonical map
$H^{-1}(L_{k_A/k}\otimes_{k_A}^L k_{A_1})\to H^{-1}(L_{k_{A_1}/k^{1/p}})$, which in turn fits into the commutative diagram with exact rows
\[
\begin{CD}
    0@>>>H^{-1}(L_{k_A/k}\otimes_{k_A}^L k_{A_1})@>>> \Omega_{k/\bF_p}\otimes_k k_{A_1}\\
    @. @VVV @VVV\\
    0@>>>H^{-1}(L_{k_{A_1}/k^{1/p}})@>>> \Omega_{k^{1/p}/\bF_p}\otimes_{k^{1/p}} k_{A_1},
\end{CD}
\]
hence it suffices to show the canonical map $\Omega_{k/\bF_p}\otimes_k k^{1/p}\to \Omega_{k^{1/p}/\bF_p}$ is zero.
This is trivial as every element of $k$ is a $p$th power in $k^{1/p}$.
 \end{Rem}

\subsection{Converse -- first case}
We treat the first case of the converse to Theorem \ref{thm:Regular-Has-Flat-L}.
The other case will be treated later with a different method, Theorem \ref{thm:Regular-Has-Flat-L-m-fg}.
\begin{Thm}\label{thm:Regular-Has-Flat-L-m-not-fg}
    Let $(V,\fm_V,k_V)$ be a valuation ring,
$(A,\fm_A,k_A)$ an essentially finitely presented $V$-algebra.
    Assume $A$ contains a field $k$.
    
    Assume $\fm_V$ is \emph{not} finitely generated, and assume
    $H^{-1}(L_{A/k}\otimes^L_Ak_A)=0$.
    Then for every finite extension $F/k$, $\operatorname{w.dim}A_F<\infty$,
    where $A_F=A\otimes_k F$.
\end{Thm}
\begin{proof}
By Theorem \ref{thm:ALLwdimFacts}\ref{wdim:flatDescent}\ref{wdim:flatAscent} it suffices to show $\operatorname{w.dim}A_F<\infty$
for every finite purely inseparable extension $F/k$.
$A_F$ is a local ring essentially finitely presented over $V$,
and $L_{A_F/F}\otimes^L_{A_F}k_{A_F}=L_{A/k}\otimes^L_{A}A_F\otimes^L_{A_F}k_{A_F}=L_{A/k}\otimes^L_{A}k\otimes^L_{k}k_{A_F}$, where $k_{A_F}$ is the residue field of $A_F$.
Therefore $H^{-1}(L_{A_F/F}\otimes^L_{A_F}k_{A_F})=0$.
Consequently, it suffices to show $\operatorname{w.dim}A<\infty$.

We may assume $V\to A$ local.
Let $\overline{A}=A/\fm_V A$.
    By \cite[Prop. 15.18]{Andre-cotangent-complex} or \citestacks{09DB}, there exists an exact sequence
    \[
    \begin{CD}
        H^{-2}(L_{k_V/V}\otimes_{k_V}^L k_A)@>>> H^{-2}(L_{\overline{A}/A}\otimes_{\overline{A}}^L k_A)@>>>\Tor_1^V(k_V,A)\otimes_A k_A@>>>\\ H^{-1}(L_{k_V/V}\otimes_{k_V}^L k_A)@>>> H^{-1}(L_{\overline{A}/A}\otimes_{\overline{A}}^L k_A)@>>> 0.
    \end{CD}
    \]

    Note that as $\fm_V$ is a filtered union of principal ideals we have $L_{k_V/V}=\fm_V/\fm_V^2[1]$.
    Since $\fm_V$ is not finitely generated, $\fm_V=\fm_V^2$, so $L_{k_V/V}=0$.
    We therefore have
    \begin{align}
       H^{-2}(L_{\overline{A}/A}\otimes_{\overline{A}}^L k_A)&\simeq\Tor_1^V(k_V,A)\otimes_A k_A, \label{H2=tor1}\\
       H^{-1}(L_{\overline{A}/A}\otimes_{\overline{A}}^L k_A)&=0. \label{H1=0}
    \end{align}
Consider the Jacobi--Zariski long exact sequence
\[\begin{CD}
        %H^{-2}(L_{A/k}\otimes_{A}^L k_A)
        .@.
        H^{-2}(L_{\overline{A}/k}\otimes_{\overline{A}}^L k_A)@>>> 
        H^{-2}(L_{\overline{A}/A}\otimes_{\overline{A}}^L k_A)@>>>\\ H^{-1}(L_{A/k}\otimes_{A}^L k_A)@>>> H^{-1}(L_{\overline{A}/k}\otimes_{\overline{A}}^L k_A)@>>> H^{-1}(L_{\overline{A}/A}\otimes_{\overline{A}}^L k_A).
    \end{CD}
    \]
\eqref{H1=0} and the assumption $H^{-1}(L_{A/k}\otimes^L_Ak_A)=0$ tells us $H^{-1}(L_{\overline{A}/k}\otimes^L_{\overline{A}}k_A)=0$.
Therefore the Noetherian local ring $\overline{A}$ is regular and $H^{-2}(L_{\overline{A}/k}\otimes^L_{\overline{A}}k_A)=0$, cf. \cite[Th.~16.18]{Andre-cotangent-complex}.
Now the sequence tells us $H^{-2}(L_{\overline{A}/A}\otimes_{\overline{A}}^L k_A)=0$,
which gives $\Tor_1^V(k_V,A)\otimes_A k_A=0$ by \eqref{H2=tor1}.

At this point, we note that $\Tor_1^V(k_V,A)$ is a finitely generated $A$-module.
Indeed, $A=P/J$ where $P$ is a localization of a polynomial $V$-algebra and $J$ is a finitely generated ideal;
therefore, writing $\overline{P}=P/\fm_V P$, we have $k_V\otimes^L_VA=\overline{P}\otimes^L_PA\in D^{\leq 0}_{\Coh}(\overline{P})$ (cf. \S\ref{subsec:DofCohRing}).
Therefore $\Tor_1^V(k_V,A)$ is a finitely generated $\overline{P}$-module, hence a finitely generated $A$-module.
Consequently, $\Tor_1^V(k_V,A)=0$ by Nakayama's Lemma.
As $\operatorname{w.dim}V\leq 1$,
we get $\overline{A}=k_V\otimes^L_VA$ and $\operatorname{w.dim}_A\overline{A}\leq 1$.
Since $\overline{A}$ is regular, we see $\operatorname{w.dim}A<\infty$ from Theorem \ref{thm:ALLwdimFacts}\ref{wdim:quotient}.
\end{proof}
\section{Approximation of valuation rings by regular local rings}\label{sec:approximation-regular-v}
In this section, we discuss a much finer approximation of valuation rings by Noetherian rings than we used in \S\ref{subsec:ConstructBCM}.
We will consider valuation rings $V$ that are  a filtered colimit of (Noetherian) regular local rings $R_{\lambda}$,
and descend objects over $V$ to objects over some $R_{\lambda}$ preserving homological properties.

Some fundamental ideas here are similar to those in the forthcoming joint work \cite{Lyu-Bogdan-approximation}, or earlier versions of it, where we discuss systematic approximation of a \emph{Noetherian} local ring by other Noetherian local rings.

In future work, the author will consider further homological properties of finitely presented algebras over valuation rings, in which Noetherian approximation as in this section will be further developed and will play a critical role.

\subsection{Valuation rings as filtered colimits of regular local rings}
The key result follows from Gabber's local uniformization theorem \cite[Exp.~VII, Th.~1.1]{Gabber-ILO},
cf. \cite[Cor.~3.4]{formal-lifting-excellence-Gabber} and \cite[\S4.2]{ADVal}.
We remind the reader that the words ``quasi-excellent,'' ``excellent,'' and ``regular'' for a local ring are reserved for Noetherian local rings.

\begin{Thm}\label{thm:AbsClosedVisIndRegular}
Let $T\to V$ be a local map of local rings.
Assume that $T$ is quasi-excellent and that $V$ is an absolutely integrally closed valuation ring.
Then $V=\colim_{\lambda\in\Lambda}R_{\lambda}$,
where $\Lambda$ is a directed set, every $R_{\lambda}$ is a regular local ring essentially finitely generated over $T$,
and the transition maps are local.
If $T\to V$ is injective, then we may arrange that $R_{\lambda}\to V$ is injective for all $\lambda\in\Lambda$.
\end{Thm}

Similarly, in equal characteristic zero, the following result follows essentially from  \cite{Zariski-local-uniformization}, or the stronger \cite{Hironaka}.

\begin{Thm}\label{thm:VoverQisIndRegular}
    Let $V$ be a valuation ring containing $\bQ$.
    Then $V=\colim_{\lambda\in\Lambda}R_{\lambda}$,
where $\Lambda$ is a directed set, every $R_{\lambda}$ is a regular local ring essentially finitely generated over $\bQ$,
and the transition maps are local.
\end{Thm}

\subsection{Regular local rings and valuation rings}
The following observation is important in the proof of Theorem \ref{thm:TorIndependentForIndRegularV} below
in the case $R_{\lambda}\to R$ not injective; otherwise it is trivial.

\begin{Lem}\label{lem:descend-injectivity}
    Let $A\to B$ be a map of rings, where
    $A$ is a Noetherian integral domain and 
    %$A$ is a regular local ring and 
    $B$ is reduced. 
    Let $\varphi:\cF\to \cG$ be a morphism of coherent $\cO_{\bP^N_{A}}$-modules.

    Assume both $\cF$ and $\cG$ are flat over $A$ and that $\varphi_B:\cF\otimes_A B\to \cG\otimes_A B$ is an injective morphism of $\cO_{\bP^N_{B}}$-modules.
    Then $\varphi$ is injective.
\end{Lem}
\begin{proof}
    Let $\fq$ be a minimal prime of $B$ and let $\fp=\fq\cap A$.
    As $\cF$ is flat over $A$ and as $A\to A_\fp$ is injective, we may apply the base change $A\to A_\fp$ and $B\to B_\fq$ to assume that $A$ is local, $B$ is a field, and $A\to B$ is local.
    Then $B$ is faithfully flat over the residue field $k$ of $A$, so we may assume $B=k$.
    
    Let $x\in \bP^N_A$ be a closed point, which lies above the maximal ideal $\fm\in\Spec A$.
    The map $\varphi_x:\cF_x\to \cG_x$ of stalks is a map between finite $\cO_{\bP^N_A,x}$-modules flat over $A$.
    The assumption tells us the mod $\fm$ map $\cF_x/\fm \cF_x\to \cG_x/\fm \cG_x$ is injective.
    Hence $\varphi_x$ is injective by \citestacks{00ME}.
    Since $x$ was arbitrary, $\varphi$ is injective.
    %Since $A\to B$ is local, there exists a closed point $y\in \bP^N_B$ lying above $x$, and $C:=\cO_{\bP^N_A,x}\to \cO_{\bP^N_B,y}=:D$ 
    %is a local map of local rings where $C$ is an integral domain.
    %As $\cF$ is flat over $A$, we may apply the base change $A\to A_\fp$ and $B\to B_\fq$ to assume that $\dim B=0$,
    %complete intersection morphism of Noetherian local rings by \cite[Lem.~2.2.5]{Lyu-Bogdan-approximation}.
    %Finally, let $M=[\cF_x\xrightarrow{\varphi_x}\cG_x]\in D^{[-1,0]}_{\Coh}(C)$.
    %As $\cF$ and $\cG$ are flat over $A$ and as $\varphi_B$ is injective, we have $M\otimes^L_{C}D\in D^{[0,0]}(D)$, hence $M\in D^{[0,0]}(C)$ \cite[Lem.~2.1.5]{Lyu-Bogdan-approximation}.
    %As this is true for all closed points $x$, we see $\varphi$ is injective.
\end{proof}

We consider the following
\begin{Situ}\label{situ:approximationOf1scheme1module}
    Let $\Lambda$ be a directed set with a minimal element $0$,
    $(R_{\lambda})_{\lambda}$ be a direct system of local maps of local rings with colimit $R$.
    Let $X_0$ be an $R_0$-scheme of finite presentation, $\cF_0$ a finitely presented $\cO_{X_0}$-module.
    Denote by $X_{\lambda}$ (resp. $X$) the base change of $X_0$ to $R_\lambda$ (resp. $R$),
    and by $\cF_{\lambda}$ (resp. $\cF$) the base change of $\cF_0$ to $X_\lambda$ (resp. $X$).
    Denote by $v_\lambda:X\to X_{\lambda}$ (resp. $v_{\lambda',\lambda}:X_{\lambda'}\to X_{\lambda}$) the projection map (resp. transition map).
\end{Situ}

\begin{Thm}\label{thm:TorIndependentForIndRegularV}
In Situation \ref{situ:approximationOf1scheme1module},
assume that $R$ is a valuation ring and that every $R_{\lambda}$ is a regular local ring.
Then there exists $\lambda\in\Lambda$ so that for all $\lambda'\geq\lambda$,
the rings $R_{\lambda'}$ and $R$ are Tor-independent with both $X_{\lambda}$ and $\cF_{\lambda}$ over $R_{\lambda}$.
\end{Thm}
\begin{proof}
Passing to an affine open cover of $X_0$ may assume $X_0$ affine,
hence a closed subscheme of $\bA^N_{R_0}$.
Replacing $X_0$ and $\cF_0$ by
$\bA^N_{R_0}$ and $\cO_{X_0}\oplus \cF_0$ we may assume $X_0=\bA^N_{R_0}$.
By \citestacks{0G41}, $\cF_0$ extends to a coherent sheaf on $\bP^N_{R_0}$, so we may assume $X_0=\bP^N_{R_0}$.

Assume $X_0=\bP^N_{R_0}$, so Tor-independence with $X_\lambda$ is trivial.
    As $\cF_0$ is coherent there exists a surjective morphism $\cE_0\to\cF_0$ where $\cE_0$ is a vector bundle.
    Denote by $\cE_{\lambda}$ (resp. $\cE$) the base change of $\cE_0$ to $X_\lambda$ (resp. $X$).

    Let $\cK=\ker(\cE\to \cF)$. By Theorem \ref{thm:ModAndAlgOverV}\ref{OverV:fp-is-coherent} $\cK$ is finitely presented.
    By \citestacks{01ZR}, there exists $\lambda\in\Lambda$ and a morphism $\psi_{\lambda}:\cK_{\lambda}\to \cE_{\lambda}$ in $\Coh(\bP^N_{R_\lambda})$ whose base change to $R$ is the inclusion map $\cK\to\cE$.
    By Theorem \ref{thm:ModAndAlgOverV}\ref{OverV:torsion-free-is-flat}, $\cK$ is flat over $V$,
    so we may assume $\cK_{\lambda}$ is flat over $R_{\lambda}$ \citestacks{05LY}.
    Since the vector bundle $\cE_{\lambda}$ is also flat over $R_{\lambda}$,
    we see from Lemma \ref{lem:descend-injectivity} that $\psi_{\lambda}$ is injective,
    and the same is true by the same reasoning for all further base changes $\psi_{\lambda'}:\cK_{\lambda'}\to \cE_{\lambda'}$ to $R_{\lambda'}\ (\lambda'\geq \lambda)$.
    This tells us that $Lv^*_{\lambda',\lambda}\coker(\psi_{\lambda})%\otimes^L_{R_{\lambda}}R_{\lambda'}
    =\coker(\psi_{\lambda'})$ and $Lv^*_{\lambda}\coker(\psi_{\lambda})=\cF$,
    so $\coker(\psi_{\lambda'})$ is Tor-independent with $R_{\lambda''}\ (\lambda''\geq\lambda')$ and $R$.
    Finally, the uniqueness part of \citestacks{01ZR} tells us $\coker(\psi_{\lambda'})\cong\cF_{\lambda'}$ for all large $\lambda'$, as desired.
\end{proof}

\begin{Cor}\label{cor:SeparateCohomology}
    In Situation~\ref{situ:approximationOf1scheme1module}, let $M_0\in D^-_{\Coh}(X_0)$ and set $M_\lambda = Lv_{\lambda, 0}^* M_0\in D^-_{\Coh}(X_\lambda)$ and $M= Lv_{0}^* M_0 \in D^-_{\Coh}(X)$. Then, for any integer $c$, there is an element $\lambda \in \Lambda$ such that the natural morphisms
\[
v_{\lambda', \lambda}^*\big( \mathcal{H}^c(M_\lambda) \big)\to \mathcal{H}^c(M_{\lambda'}) \text{ and } v_{\lambda}^* \big(\mathcal{H}^c(M_{\lambda})\big) \to \mathcal{H}^c(M)
\]
are isomorphism for any $\lambda'\geq \lambda$.
\end{Cor}
\begin{proof}
    Let $b:=b(M_\bullet)$ be the minimal element of $\bZ \cup \{-\infty\}$ such that 
    there exists
    some $\lambda\in \Lambda$ so that $M_\lambda\in D^{\leq b}_{\Coh}(X_\lambda)$.
    If $b<c$, then the contention is trivial. Therefore, we can assume that $b\geq c$ and assume that the statement is true for all $M_0$ with $b(M_\bullet)<b$. 
    By our choice of $b$, there exists $\lambda\in \Lambda$ such that $M_{\lambda'}\in D^{\leq b}_{\Coh}(X_{\lambda'})$ and $M \in D^{\leq b}_{\Coh}(X)$ for every $\lambda'\geq \lambda$. In particular,
    \[
    v_{\lambda', \lambda}^*\big( \mathcal{H}^b(M_\lambda) \big)\to \mathcal{H}^b(M_{\lambda'}) \text{ and } v_{\lambda}^* \big(\mathcal{H}^b(M_{\lambda})\big) \to \mathcal{H}^b(M)
    \]
    are isomorphisms. If $b=c$, this finishes the proof. So we can assume that $b>c$. Now Theorem \ref{thm:TorIndependentForIndRegularV} implies that, after enlarging $\lambda$, we can also assume that 
    \[
        Lv_{\lambda', \lambda}^*\big( \mathcal{H}^b(M_\lambda) \big)\to \mathcal{H}^b(M_{\lambda'}) \text{ and } Lv_{\lambda}^* \big(\mathcal{H}^b(M_{\lambda})\big) \to \mathcal{H}^b(M)
    \]
    are isomorphisms as well. Therefore, the distinguished triangles $\tau^{<b} M_{\lambda'} \to M_{\lambda'} \to \mathcal{H}^b(M_{\lambda'})[-b]$ for $\lambda'\ge \lambda$ and $\tau^{<b} M \to M \to \mathcal{H}^b(M)[-b]$ imply that the natural maps
    \[
            Lv_{\lambda', \lambda}^*\big( \tau^{<b}M_\lambda \big)\to \tau^{<b}M_{\lambda'} \text{ and } Lv_{\lambda}^* \big(\tau^{<b}M_{\lambda}\big) \to \tau^{<b}M
    \]
    are isomorphisms for $\lambda'\ge\lambda$. Since $c<b$ by our assumption, we can replace $\Lambda$ with $\Lambda_{\geq\lambda}$ and $M_0$ with $\tau^{<b}M_\lambda$ to assume that $b(M_\bullet)<b$. 
    We win by induction.
    %This case has been already handled by the induction hypothesis. 
\end{proof}

\begin{Thm}\label{thm:descend-any-Hc}
Let $R$ be an excellent regular local ring, $V$ is a valuation ring, and let $R\to V$ be a local map.
Let $X$ be an $R$-scheme of finite type and $Y=X\times_{\Spec R}\Spec V$.
Denote by $w$ the projection morphism $Y\to X$.
Let $y\in Y$ and let $x=v(y)\in X$.
Assume that $X$ and $V$ are Tor-independent over $R$.

    Let $M\in D_{\Coh}(\cO_{X,x})$.
    For any integer $c$,
    if $H^c(M\otimes_{\cO_{X,x}}^L \cO_{Y,y})=0$,
    %\in D^{[a,b]}(\cO_{Y,y})$ for some integers $a\leq b$,
    then $H^c(M)=0$.
    %$M\in D^{[a,b]}(\cO_{X,x})$.
\end{Thm}
\begin{proof}
As $X$ and $V$ are Tor-independent over $R$, we may shrink $X$ and replace $X$ by an affine space over $R$ to assume $X$ is smooth over $R$.
    As an extension of valuation rings is  faithfully flat, we may assume $V$ is absolutely integrally closed.
    As $R$ has finite global dimension,
    Tor-indepence tells us
    the functor $-\otimes_{\cO_{X,x}}^L \cO_{Y,y}$ has bounded cohomological dimension, so we may assume $M\in D^b_{\Coh}(\cO_{X,x})$.
    %as $\tau_{\geq c}M\otimes_{\cO_{X,x}}^L \cO_{Y,y}=M\otimes_{\cO_{X,x}}^L \cO_{Y,y}$ for all $c<a$.
    Shrinking $X$, we may also assume $M$ is the stalk of an object $K\in D^b_{\Coh}(X)$, cf. \cite[Lem.~2.2]{Lyu-dual-complex-lift}.
    Furthermore, as $N:=Lw^*K$ is pseudo-coherent, $\mathcal{H}^c(N)$ is finitely presented (\S\ref{subsec:DofCohRing}) and its stalk at $y$ is zero, so we may assume $\mathcal{H}^c(N)=0$.

Write $V=\colim_{\gamma\in\Gamma}R_{\gamma}$ as in Theorem \ref{thm:AbsClosedVisIndRegular} with $T=R$. 
Let $Y_{\gamma}=X\times_{\Spec R}\Spec R_{\gamma}$ and denote by $w_\gamma$ the projection morphism $Y_{\gamma}\to X$.
 Denote by $v_\gamma:Y\to Y_{\gamma}$ (resp. $v_{\gamma',\gamma}:Y_{\gamma'}\to Y_{\gamma}$) the projection map (resp. transition map).
Let $N_{\gamma}=Lw_{\gamma}^*K$.
By Corollary \ref{cor:SeparateCohomology},
we may assume $v_{\gamma',\gamma}^*\mathcal{H}^c(N_{\gamma})=\mathcal{H}^c(N_{\gamma'})$ for all $\gamma'\geq\gamma$ and $v_{\gamma}^*\mathcal{H}^c(N_{\gamma})=\mathcal{H}^c(N)=0$.
By the uniqueness part of \citestacks{01ZR} we see for large $\gamma$ we have $\mathcal{H}^c(N_{\gamma})=0$.
Finally, %as $R\to S_{\gamma}$ is lci by 
$w_\gamma$ is lci by \cite[(5.4)]{avramov-ci},
as $X$ is smooth over $R$,
so  $\mathcal{H}^c(N_{\gamma})=0$ implies $\mathcal{H}^c(K)_{p}=0$ for all $p\in w_{\gamma}(Y_{\gamma})$, cf. \cite[Lem.~3.1]{Lyu-dual-complex-lift}.
In particular, $\mathcal{H}^c(K)_{x}=0$, that is, $H^c(M)=0$.
\end{proof}

\begin{Rem}
    Using techniques in \cite{Lyu-Bogdan-approximation}, we can generalize Theorem \ref{thm:descend-any-Hc} to the case $R$ is any regular local ring, not necessarily excellent.
    Moreover, the general case  of Theorem \ref{thm:descend-any-Hc}  follows from the case  $R$ is essentially finitely generated over $\bZ$, 
    and in this case, 
    Theorem \ref{thm:AbsClosedVisIndRegular} follows from \cite{deJong-alterations}.
\end{Rem}

\subsection{Regular sequences}
As a quick application we show that the following characterization of regular sequence, well-known for Noetherian rings, holds in our setup.
\begin{Thm}\label{thm:Reg=H1reg}
    Let $(A,\fm_A)$ be a local ring essentially finitely presented over a valuation ring $V$.
    Let $\underline{x}=x_1,x_2,\ldots,x_n$ be a sequence in $\fm_A$.
    Let $M$ be a coherent $A$-module.
    
    Then $\underline{x}$ is an $M$-regular sequence if and only if $\underline{x}$ is an $M$-Koszul-regular sequence, if and only if $\underline{x}$ is an $M$-$H_1$-regular sequence.
    In particular, a permutation of an $M$-regular sequence is an $M$-regular sequence.
\end{Thm}
\begin{proof}
The ``in particular'' statement is trivial as a permutation of an $M$-$H_1$-regular sequence is an $M$-$H_1$-regular sequence.
    From \citetwostacks{062F}{0CEM} we see it suffices to show an $M$-$H_1$-regular sequence is an $M$-regular sequence.
    By flat descent we may assume $V$ is absolutely integrally closed.

    Let $B$ be a finitely presented $V$-algebra with a prime ideal $\fq$ so that $B_\fq=A$.
    Localize near $\fq$ we may assume $M=N_\fq$ for some coherent $B$-module $N$ \citestacks{05N7}.
    Write $V=\colim_{\lambda\in\Lambda}R_\lambda$ as in Theorem \ref{thm:AbsClosedVisIndRegular} for $T=\bZ$.
    We may assume that $\Lambda$ has a minimal element $0$,
    $B$ and $N$ comes from a finitely presented $R_0$-algebra $B_0$ and a finite $R_0$-module $N_0$,
    and we may assume $B_0,N_0$ are Tor-independent with every $R_\lambda$ and $V$ by Theorem \ref{thm:TorIndependentForIndRegularV}.
    Then $N_\lambda:=N_0\otimes^L_{B_0}B_{\lambda}$ is a coherent $B_{\lambda}$-module and $N=\colim_{\lambda}N_\lambda$.
    We may also assume that every $x_i$ is the image of some element in $B_0$,
    which we also denote by $x_i$ by abuse of notations.

    Let $K_0$ be the Koszul complex of $\underline{x}$ with respect to $N_0$.
    Then $K_{\lambda}:=K_0\otimes^L_{B_0}B_{\lambda}$ (resp. $K:=K_0\otimes^L_{B_0}B$) is the Koszul complex of $\underline{x}$ with respect to $N_\lambda$ (resp. $N$).
    Note that $H^{-1}(K)_\fq=0$ by our assumption, and $H^{-1}(K)$ is coherent as $K\in D^b_{\Coh}(B)$ (\S\ref{subsec:DofCohRing}).
    Therefore after localing near $\fq$ we may assume $H^{-1}(K)=0$.
    From Corollary \ref{cor:SeparateCohomology} (and the uniqueness part of \citestacks{01ZR}) we then get $H^{-1}(K_{\lambda})=0$ for all large $\lambda$, in other words $\underline{x}$ is $(N_{\lambda})_{\fq\cap B_{\lambda}}$-$H_1$-regular.
    From the Noetherian case \citestacks{09CC} we see $\underline{x}$ is $(N_{\lambda})_{\fq\cap B_{\lambda}}$-regular.
    As this is true for all large $\lambda$ we see $\underline{x}$ is $M$-regular,
    since $M=\colim_{\lambda} (N_{\lambda})_{\fq\cap B_{\lambda}}$.
\end{proof}

\begin{Cor}
In Situation \ref{situ:approximationOf1scheme1module},
assume that $R$ is a valuation ring and that every $R_{\lambda}$ is a regular local ring.  
Let $x\in X$ and write $x_\lambda=v_\lambda(x)$.
Let $\underline{y}$ be a sequence in $\cO_{X_0,x_0}$.
Then $\underline{y}$ is a regular sequence in $\cF_x$ if and only if $\underline{y}$ is a regular sequence in $(\cF_\lambda)_{x_\lambda}$ for all large $\lambda$.
\end{Cor}
\begin{proof}
    As regular sequence can be detected from the Koszul complex (Theorem \ref{thm:Reg=H1reg} and \citestacks{09CC}) the result follows,
    see the last paragraph of the proof of Theorem \ref{thm:Reg=H1reg}. 
\end{proof}

\section{Refinement of Kunz's Theorem over valuation rings}\label{sec:KunzRefined}
\begin{Thm}\label{thm:Kunz-KL-over-val-ring}
    Let %$(V,\fm_V,k_V)$ 
    $V$ be a local ring of characteristic $p$ and let $(A,\fm)$ be an essentially finitely presented local $V$-algebra.
    Assume that either  $V$ is a valuation ring, or that $V$ is Noetherian.

Let $M\in D^b_{\Coh}(A)$.
    Assume that there exists an infinite sequence of integers $1\leq e_1<e_2<...$ such that $M\otimes^L_AF^{e_m}_*A$ is bounded for every $m$.
    Then $M$ has finite tor dimension over $A$.
\end{Thm}
\begin{proof}
First, as $M\in D^b_{\Coh}(A)$,
there exists a distinguished triangle
\[
\begin{CD}
    F@>>> M@>>> C@>>> +1
\end{CD}
\]
where $F$ is perfect and $C$ is concentrated in a single degree.
Therefore we may always assume $M$ is a module.

If $V$ is Noetherian, then $A$ is Noetherian, so the result follows from \cite[Prop.~2.6]{Koh-Lee-exactness-complex-Kunz-theorem}.
We may therefore assume $V$ is a valuation ring. We will reduce the problem to the case $V$ is Noetherian, which concludes the proof.

    If $B$ is a faithfully flat $A$-algebra, then $F^e_*B$ is a faithfully flat $(F^e_*A)$-algebra for all $e$.
    The statement is therefore preserved along the base change $A\to B$, so we may assume $V$ is absolutely integrally closed.

    In this case $V=\colim_{\lambda\in\Lambda}R_{\lambda}$ as in Theorem \ref{thm:AbsClosedVisIndRegular} with $T=\bF_p$;
    by %\citestacks{02JO} and 
    Theorem \ref{thm:TorIndependentForIndRegularV} we may assume $\Lambda$ has a minimal element $0$ and that there exists a finitely generated $R_0$-algebra $B_0$ that is tor-independent with $V$ over $R_0$ such that $A$ is a localization of $B_0\otimes_{R_0}V$.
We can spread the module $M$ out to a principal localization of $B_0\otimes_{R_0}V$ by \citestacks{05N7},
    so we may assume, after replacing $\Lambda$ by $\Lambda_{\geq \lambda}$ for some $\lambda$ and $B_\lambda$ by a principal localization,
     that $M$ comes from some finite $B_0$-module $N_0$.
     %$N\in D^b_{\Coh}(A_0\otimes_{R_0}V)$.
     By Theorem \ref{thm:TorIndependentForIndRegularV}, after replacing $\Lambda$ by $\Lambda_{\geq \lambda}$ for some $\lambda$,
     %$N$ is the base change of some $N_0\in D^b_{\Coh}(A_0)$ (same argument as in the proof of \cite[Lem.~2.5.1]{Lyu-Bogdan-approximation}).
     %In particular, 
     we may assume $M=N_0\otimes_{B_0}^LA$.

     At this point, we can apply Theorem \ref{thm:descend-any-Hc} to $R=F^{e_m}_*R_0,V=F^{e_m}_*V$,
     $X=\Spec( F^{e_m}_*B_0)$,
     $\cO_{X,x}=F^{e_m}_*(B_0)_{\fm\cap B_0},\cO_{Y,y}=F^{e_m}_*A$ (for a suitable choice of $y$),
     %, possible as $A$ is a localization of $A_0\otimes_{R_0}V$,
     %so $F^{e_m}_*A$ is a localization of $F^{e_m}_*A_0\otimes_{F^{e_m}_*R_0}F^{e_m}_*V$),
     and $M=N_0\otimes^L_{B_0}F^{e_m}_*(B_0)_{\fm\cap B_0}$.
     It follows that $N_0\otimes^L_{B_0}F^{e_m}_*(B_0)_{\fm\cap B_0}$ is bounded for every $m$.
     On the other hand, if $N_0\otimes^L_{B_0}(B_0)_{\fm\cap B_0}$ has finite tor dimension over $(B_0)_{\fm\cap B_0}$,
     then $M$ has finite tor dimension over $A$ by base change, as $M=N_0\otimes_{B_0}^LA$, see \citestacks{066L}.
     %for every $m$, if $N_0\otimes^L_{A_0}F^{e_m}_*A_0$  has finite tor dimension over $F^{e_m}_*A_0$, then it follows from base change that $M\otimes^L_A F^{e_m}_*A$ has finite tor dimension over $F^{e_m}_*A$.
     Therefore, we have reduced the statement to the case $V=R_0$ is Noetherian, as desired.
\end{proof}

\begin{Cor}\label{cor:OneFeFiniteTorDim=Regular}
    Let $V$ be a valuation ring containing $\bF_p$, and let $A$ be an essentially finitely presented local $V$-algebra.
    If $\operatorname{w.dim}_AF^e_*A<\infty$ for some $e\in\bZ_{>0}$, then $\operatorname{w.dim}A<\infty$.
\end{Cor}
\begin{proof}
    We have $\operatorname{w.dim}_AF^{me}_*A\leq m\operatorname{w.dim}_AF^e_*A$ for all $m\in\bZ_{>0}$
    as the ring maps $F^{(m-1)e}_*A\to F^{me}_*A$ are abstractly isomorphic to $A\to F^e_*A$.
    By Theorem \ref{thm:Kunz-KL-over-val-ring},
    all $M\in\Coh(A)$ have finite tor dimension, so $\operatorname{w.dim}A<\infty$, see Discussion \ref{discu:NorRegloci}.
\end{proof}

\section{Cotangent complexes, II}\label{sec:Cotangent2}
\subsection{Regular immersions}
\begin{Thm}\label{thm:lci=Lin-1}
    Let $V$ be a valuation ring, $(A,\fm_A,k_A)$ an essentially finitely presented local $V$-algebra.
    Let $I$ be a finitely generated ideal of $A$.
    The following are equivalent.
    \begin{enumerate}[label=$(\roman*)$]
        \item\label{RegImm:AllAreRegSeq} Every minimal set of generators of $I$ is a regular sequence in $A$.
        \item $I$ is generated by a regular sequence in $A$.
        \item $L_{(A/I)/A}$ has tor-amplitude in $[-1,-1]$.
        \item\label{RegImm:H2=0} $H^{-2}(L_{(A/I)/A}\otimes^L_{A/I}k_A)=0$.
    \end{enumerate}
\end{Thm}
\begin{proof}
In view of \citestacks{08SJ}, the only nontrivial implication is \ref{RegImm:H2=0} implies \ref{RegImm:AllAreRegSeq}.   
Assume \ref{RegImm:H2=0}.
By flat base change \citestacks{08QQ} and descent \citestacks{00LM}, in showing \ref{RegImm:AllAreRegSeq}
we may assume $V$ is absolutely integrally closed.

Let $B$ be a finitely presented $V$-algebra with a prime ideal $\fq$ so that $B_\fq=A$,
and let $J$ be a finitely generated ideal of $B$ so that $JA=I$.
    %Localize near $\fq$ we may assume $M=N_\fq$ for some coherent $B$-module $N$.
    Write $V=\colim_{\lambda\in\Lambda}R_\lambda$ as in Theorem \ref{thm:AbsClosedVisIndRegular}.
    We may assume that $\Lambda$ has a minimal element $0$,
    $B$ and $J$ comes from a finitely presented $R_0$-algebra $B_0$ and a finitely generated ideal $J_0$ of $B_0$,
    and we may assume $B_0,B_0/J_0$ are Tor-independent with every $R_\lambda$ and $V$ by Theorem \ref{thm:TorIndependentForIndRegularV}.
    %Then $N_\lambda:=N_0\otimes^L_{B_0}B_{\lambda}$ is a coherent $B_{\lambda}$-module and $N=\colim_{\lambda}N_\lambda$.
    %We may also assume that every $x_i$ is the image of some element in $B_0$,
    %which we also denote by $x_i$ by abuse of notations.
    Let $B_{\lambda}=B_0\otimes_{R_0}R_\lambda$ and let $J_{\lambda}=JB_\lambda$.
    Then \citestacks{08QQ} applies by Tor-independence and \ref{RegImm:H2=0} can be rewritten as
    $H^{-2}(L_{(B_\lambda/J_{\lambda})/B_\lambda}\otimes^L_{B_\lambda/J_{\lambda}}k_A)=0$.
    As $k_A/\kappa(\fq\cap B_{\lambda})$ is a field extension
    we have $H^{-2}(L_{(B_\lambda/J_{\lambda})/B_\lambda}\otimes^L_{B_\lambda/J_{\lambda}}\kappa(\fq\cap B_{\lambda}))=0$.
    The Noetherian case of the theorem \cite[Th.~6.25]{Andre-cotangent-complex}
    tells us every minimal set of generators of $J_\lambda(B_\lambda)_{\fq\cap B_{\lambda}}$ is a regular sequence in $(B_\lambda)_{\fq\cap B_{\lambda}}$.

    Finally, for every minimal set $\underline{x}$ of generators of $I$,
    we may replace $\Lambda$ by some $\Lambda_{\geq\lambda}$ to assume $\underline{x}$ comes from a sequence $\underline{y}$ in $J_0$.
    We may replace $J_0$ by the ideal generated by $\underline{y}$ to assume that $\underline{y}$ generate $J_0$.
    The image of $\underline{y}$ in every $(B_\lambda)_{\fq\cap B_{\lambda}}$ is automatically a  minimal set   of generators of $J_\lambda(B_\lambda)_{\fq\cap B_{\lambda}}$,
    therefore a regular sequence in $(B_\lambda)_{\fq\cap B_{\lambda}}$. 
    As $A=\colim_\lambda (B_\lambda)_{\fq\cap B_{\lambda}}$ we get \ref{RegImm:AllAreRegSeq}.
\end{proof}

\begin{Thm}\label{thm:Regular-Has-Flat-L-m-fg}
    Let $(V,\fm_V,k_V)$ be a valuation ring,
$(A,\fm_A,k_A)$ an essentially finitely presented $V$-algebra.
    Assume $A$ contains a field $k$.
    
    Assume $\fm_V$ \emph{is} finitely generated, and assume
    $H^{-1}(L_{A/k}\otimes^L_Ak_A)=0$.
    Then for every finite extension $F/k$, $\operatorname{w.dim}A_F<\infty$,
    where $A_F=A\otimes_k F$.
\end{Thm}
\begin{proof}
Same as in the proof of Theorem \ref{thm:Regular-Has-Flat-L-m-not-fg} we may assume $V\to A$ is local and it suffices to show $\operatorname{w.dim} A<\infty.$
As $A/\fm_V A$ is Noetherian we have $\fm_A$ finitely generated.
As $k_A/k$ is a field extension, $L_{k_A/k}$ has tor-amplitude in $[-1,0]$,
hence the assumption $H^{-1}(L_{A/k}\otimes^L_Ak_A)=0$
 gives $H^{-2}(L_{k_A/A})=0$,
 via the Jacobi--Zariski distinguished triangle $L_{A/k}\otimes^L_Ak_A\to L_{k_A/k}\to L_{k_A/A}\to +1$.
 As $\fm_A$ is finitely generated, it is generated by a regular sequence by Theorem \ref{thm:lci=Lin-1},
 so $\operatorname{w.dim}A<\infty$ by Theorem \ref{thm:ALLwdimFacts}\ref{wdim:quotient}.
\end{proof}
 \subsection{Geometrically regular fibers}
 We discuss a ``relative'' version of Theorems \ref{thm:Regular-Has-Flat-L-m-not-fg} and \ref{thm:Regular-Has-Flat-L-m-fg}.
\begin{Thm}\label{thm:RegularMapsByL}
    Let \[\begin{CD}
        (V,\fm_V,k_V) @>>> (W,\fm_W,k_W)\\
        @VVV @VVV\\
        (A,\fm_A,k_A)@>>> (B,\fm_B,k_B)
    \end{CD}\]
be a commutative diagram of local maps of local rings, where $V$ and $W$ are valuation rings and the vertical arrows are essentially finitely presented.
Assume $H^{-1}(L_{B/A}\otimes_B^Lk_B)=0$.
Then %the following hold.
    for every finite extension $F/k_A$, $B_F=B\otimes_A F$ is essentially finitely presented over a valuation ring, and $\operatorname{w.dim}B_F<\infty$.
    %\item $\Tor_1^A(k_A,B)=0$.
\end{Thm}
\begin{proof}
    By \cite[Prop.~7.31]{Andre-cotangent-complex} or \citestacks{09DB},
    we have a surjective map $H^{-1}(L_{B/A}\otimes_B^Lk_B)\to H^{-1}(L_{B_{k_A}/k_A}\otimes_{B_{k_A}}^Lk_B)$.
    %Therefore $H^{-1}(L_{B_{k_A}/k_A}\otimes_{B_{k_A}}^Lk_B)=0$,
    %so for every $F/k_A$ finite,
    %$H^{-1}(L_{B_{F}/F}\otimes_{B_F}^Lk_{B_F})=0$,
    %where $k_{B_F}$ is the residue field of $B_F$.
    %For the prime field $k_0\subseteq F$ we then have $H^{-1}(L_{B_{F}/k_0}\otimes_{B_F}^Lk_{B_F})=0$, as $L_{F/k_0}$ has tor-amplitude in $[0,0]$.
    By Theorems \ref{thm:Regular-Has-Flat-L-m-not-fg} and \ref{thm:Regular-Has-Flat-L-m-fg},
    we only need to prove $B_F$ is essentially finitely presented over a valuation ring.
    As $B_F$ is essentially finitely presented over $B/\fm_A B$,
    hence over $B/\fm_V B$ as $A/\fm_V A$ is Noetherian,
    %to prove \ref{RegMaps:fiberGeomReg} 
    it suffices to show $B/\fm_V B$ is essentially finitely presented over a valuation ring.

    If $\fm_V$ is not finitely generated, then  $\fm_V=\fm_V^2$.
    Now $\fm_V W$ is a prime ideal by Lemma \ref{lem:SquareIsPrimeInVal}.
    In this case, $B/\fm_V B$ is essentially finitely presented over $W/\fm_V W$, as desired.
    Now assume $\fm_V$ is finitely generated.
    Then $B/\fm_V B$ is finitely presented over $B$, hence essentially finitely presented over $W$,
     we win.
\end{proof}

\begin{Rem}
    What is the ultimate source of the exact sequences \cite[Prop. 15.18 and 7.31]{Andre-cotangent-complex} (or \citestacks{09DB}) used in the proof of Theorems \ref{thm:Regular-Has-Flat-L-m-not-fg} and \ref{thm:RegularMapsByL}?
    In fact, they come from the Jacobi--Zariski sequence associated with the maps $C\xrightarrow{\varphi} B\otimes_A^LC\xrightarrow{\psi} B\otimes_A C$ of connective $\mathbb{E}_{\infty}$-rings (or animated rings),
    where $L_\varphi=L_{B/A}\otimes_A^LC$, and $L_\psi$ is controlled by \cite[Th.~7.4.3.1]{HA}: the nonzero homotopy of $L_\psi$ at the smallest (homological) degree is $\operatorname{Tor}_k^A(B,C)$  in  degree $k+1$, provided $\operatorname{Tor}_j^A(B,C)=0$ for $0<j<k$ and $\operatorname{Tor}_k^A(B,C)\neq 0$.
\end{Rem}

\begin{Rem}
    In contrast to the Noetherian case (cf. \cite[Th.~16.18]{Andre-cotangent-complex}), the condition in Theorem \ref{thm:RegularMapsByL} does not imply flatness of $B$ over $A$.
    In fact, consider the simplest example $A=V$ and $B=W=k_V$, where $\fm_V$ is not finitely generated.
    Then $L_{B/A}=0$ as seen in the proof of Theorem \ref{thm:Regular-Has-Flat-L-m-not-fg},
    but $A\to B$ is not flat.
\end{Rem}

\section{Vanishing theorems}\label{sec:VanishingTh}

\subsection{Vanishing in equal characteristic zero}
We prove a version of the Kodaira's vanishing theorem in our setup.
We have not defined or studied dualizing complexes and upper shriek in our setup, so we use an ad hoc notion.
\begin{Thm}\label{thm:VanishingOverQ}
    Let $V$ be a valuation ring of residue characteristic $0$,
    and let $X$ be a closed subscheme of $\bP^r_V$ for some $r$.
    Assume that $X$ is flat over $\Spec V$ of relative dimension $d$
    and assume that $\operatorname{Reg}(X)=X$.
    %$\operatorname{w.dim}\cO_{X,x}<\infty$ for all $x\in X$.

    Let $\omega_X=\mathcal{E}xt^{r-d}_{\bP^r_V}(\cO_X,\cO_{\bP^r_V}(-r-1))$.
    Then $\omega_X$ is an invertible sheaf, and for every ample invertible sheaf $\cL$ on $X$ and every $j>0$ we have $H^j(X,\omega_X\otimes\cL)=0$.
\end{Thm}
\begin{proof}
    Write $V=\colim_{\lambda\in\Lambda}R_\lambda$ as in Theorem \ref{thm:VoverQisIndRegular}.
    Note that $X$ is of finite presentation over $V$ \citestacks{053G}.
    We may assume that $\Lambda$ has a minimal element $0$ and
    that $X$ and $\cL$ come from a closed subscheme $X_0$ of $\bP^r_{R_0}$ and an ample invertible sheaf $\cL_0$ on $X_0$, see \citetwostacks{0B8W}{09MT}.
    By \citestacks{05LY} we may assume $X_0$ is flat over $R_0$.

    Let $x_0\in X_0$ be an arbitrary closed point.
    As $X_0\to \Spec R_0$ is proper and as $R_0\to V$ is local,
    there exists a closed point $x\in X$ lying over $x_0$.
    Then $\cO_{X,x}$ is a localization of $\cO_{X_0,x_0}\otimes_{R_0} V$.
    As $\cO_{X_0,x_0}$ is flat over $R_0$ and as $\operatorname{w.dim}_{R_0}V<\infty$ (since $R_0$ is a regular local ring),
    it follows from \citetwostacks{066M}{066K} that $\operatorname{w.dim}_{\cO_{X_0,x_0}}\cO_{X,x}<\infty$.
    As $\operatorname{w.dim}\cO_{X,x}<\infty$, another application of \citestacks{066K} shows $\operatorname{w.dim}_{\cO_{X_0,x_0}}\kappa(x)<\infty$.
    But $\kappa(x)/\kappa(x_0)$ is a field extension, so  $\operatorname{w.dim}_{\cO_{X_0,x_0}}\kappa(x_0)<\infty$,
    hence $\cO_{X_0,x_0}$ is a regular local ring.
    As $x_0$ was arbitrary, we see $X_0$ is regular.

Write $f_0$ for the morphism $X_0\to\Spec R_0$.
As $X_0$ and $R_0$ are regular and as $f_0$ is flat we see all fibers of $f_0$ are Cohen--Macaulay,
hence $X_0=W\sqcup W'$, where $W\to \Spec R_0$ is of relative dimension $d$ and at no point of $W'$ the fiber dimension over $\Spec R_0$ is $d$, see \citestacks{02NM}.
It is then clear that $W'\times_{\Spec R_0}\Spec V=\emptyset$, so $W'=\emptyset$ by properness and $f_0$ has relative dimension $d$.
It now follows that
$\omega_{X_0}:=\mathcal{E}xt^{r-d}_{\bP^r_{R_0}}(\cO_{X_0},\cO_{\bP^r_{R_0}}(-r-1))$
is the unique nonzero cohomology sheaf of $f_0^!R_0$, \emph{i.e.}, the dualizing sheaf of $X_0$ compatible with $R_0$.
As $R_0$ is essentially finitely generated over $\bQ$, the classical Kawamata--Viehweg vanishing theorem \cite[Th.~1-2-3]{KMM87-MMP} tells us $H^j(X_0,\omega_{X_0}\otimes\cL_0)=0$ for all $j>0$.

As $X_0$ and $\bP^r_{R_0}$ are flat over $R_0$,
we have $L(\bP^r_{V}\to \bP^r_{R_0})^*\cO_{X_0}=\cO_X$,
so by \citestacks{0A6A} we have $$L(\bP^r_{V}\to \bP^r_{R_0})^*R\mathcal{H}om_{\cO_{\bP^r_{R_0}}}(\cO_{X_0},\cO_{\bP^r_{R_0}}(-r-1))=R\mathcal{H}om_{\cO_{\bP^r_V}}(\cO_X,\cO_{\bP^r_V}(-r-1)).$$
As $\omega_{X_0}$ is an invertible sheaf on $X_0$ and the only nonzero cohomology sheaf of $R\mathcal{H}om_{\cO_{\bP^r_{R_0}}}(\cO_{X_0},\cO_{\bP^r_{R_0}}(-r-1))$,
it now follows that 
$(X\to X_{0})^*\omega_{X_0}=\omega_X$,
in particular $\omega_X$ is an invertible sheaf.
Hence \citestacks{07VK} tells us $R\Gamma(X_0,\omega_{X_0}\otimes\cL_0)\otimes^L_{R_0}V=R\Gamma(X,\omega_{X}\otimes\cL)$.
As $R\Gamma(X_0,\omega_{X_0}\otimes\cL_0)\in D^{\leq 0}$,
we conclude that $R\Gamma(X,\omega_{X}\otimes\cL)\in D^{\leq 0}$,
as desired.
\end{proof}

\subsection{Product of local rings}
We make some preparation for vanishing in large residue characteristics.
The intuition comes from the consideration of ultraproduct of local rings, which are simultaneously localizations and quotients of (infinite) direct product of local rings, cf. \cite{SchoutensBOOK} and \cite[Appendix A]{Lyu-thesis}.
\begin{Lem}\label{lem:ProdofLocalRings}
    Let $(R_\alpha,\fm_\alpha,k_\alpha)_{\alpha\in\Gamma}$ be a family of local rings,
    and let $\mathcal{R}=\prod_\alpha R_\alpha$.
    Then the following hold.
    \begin{enumerate}[label=$(\roman*)$]
        \item\label{prodlocal:Jac} $\prod_\alpha \fm_\alpha$ is the Jacobson radical of $\mathcal{R}$,
        and $V(\prod_\alpha \fm_\alpha)=\Max(\mathcal{R})$.
        \item\label{prodlocal:Principal} Let $M_\alpha$ be the preimage of $\fm_\alpha$ in $\mathcal{R}$.
        Then $M_\alpha\in\Max(\mathcal{R})$ and $\mathcal{R}_{M_\alpha}=R_\alpha$.
        Moreover, the singleton $\{M_\alpha\}$ is a principal open subset of $\Max(\mathcal{R})=\Spec(\prod_\alpha k_\alpha)$.
\item\label{prodlocal:dense} The set of all $M_\alpha$ is dense in the constructible topology of $\Max(\mathcal{R})=\Spec(\prod_\alpha k_\alpha)$.
    \end{enumerate}
\end{Lem}
\begin{proof}
    We know $\prod_\alpha \fm_\alpha$ is contained in the Jacobson radical of $\mathcal{R}$ by \citestacks{0AME}.
    It follows from \citetwostacks{092G}{092F} that $\dim\prod_\alpha k_\alpha=0$, showing \ref{prodlocal:Jac}.
    Writing $\mathcal{R}=R_{\alpha}\times\prod_{\beta\neq\alpha}R_\beta$ we get \ref{prodlocal:Principal}.
    To show \ref{prodlocal:dense},
    we may assume $R_\alpha=k_\alpha$ for every $\alpha$.
    Then every element of $\mathcal{R}$ is the product of a unit and an idempotent,
    so every principal open is defined by an idempotent.
    Therefore principal opens themselves form a Boolean algebra, hence there are no other constructible subsets.
    As every nonempty principal open contains at least one $M_\alpha$,
    we get \ref{prodlocal:dense}.
\end{proof}

\begin{Cor}\label{cor:CharofProd}
    Notations and assumptions as in Lemma \ref{lem:ProdofLocalRings}.
    Assume that for every prime $p$ there are only finitely many $\alpha$ with $\operatorname{char}(k_\alpha)=p$.
    Then for every $M\in\Max(\mathcal{R})\setminus\{M_\alpha\mid \alpha\in\Gamma\}$,
    we have $\operatorname{char}(\mathcal{R}/M)=0$.
\end{Cor}
\begin{proof}
    Let $p$ be a prime number. 
    Our assumption says that for some $s$ the constructible subset $(V(p)\cap\Max(\mathcal{R}))\setminus\{M_{\alpha_1},\ldots,M_{\alpha_s}\}$ of $\Max(R)$ contains no $M_\alpha$,
    so it is empty by Lemma \ref{lem:ProdofLocalRings}\ref{prodlocal:dense}, showing that $p\not\in M$.
\end{proof}

The following result follows from \cite[Lem.~6.3.6 and 6.3.7]{Coherent-Rings} in view of \cite[Cor.~4.2.19]{Coherent-Rings}.
We give an elementary proof.
\begin{Thm}\label{thm:ProdofVsemihereditary}
    Let $\{V_\alpha\}_{\alpha\in \Gamma}$ be a family of valuation rings and let $\mathcal{V}=\prod_\alpha V_\alpha$.
    Then $\mathcal{V}$ is semihereditary, that is, every finitely generated ideal of $\mathcal{V}$ is projective.
    %Then every finitely generated ideal of the direct product $\mathcal{V}:=\prod_\alpha V_\alpha$ is coherent, and $\mathcal{V}_\fp$ is a valuation ring for all $\fp\in\Spec\mathcal{V}$.
    %In other words, $\mathcal{V}$ is semihereditary \cite[Cor.~4.2.19]{Coherent-Rings},
    %hence all essentially finitely presented algebras over $\mathcal{V}$ are coherent \cite[Cor.~4.2.19]{Coherent-Rings}
\end{Thm}
\begin{proof}
    Let $f_i=(f_{i\alpha})_\alpha\ (1\leq i\leq s)$ be a finite collection of elements of $\mathcal{V}$.
    As each $V_\alpha$ is a valuation ring we see $f_{i\alpha}\ (1\leq i\leq s)$ generates a principal ideal of $V_\alpha$, so $(f_{1\alpha},\ldots,f_{s\alpha})V_\alpha =f_{i\alpha}V_\alpha$ for some $i=i(\alpha)$.
    We can then write $\Gamma$ as a finite union $\Gamma=\Gamma_1\cup\ldots\cup\Gamma_s$
    so that $(f_{1\alpha},\ldots,f_{s\alpha})V_\alpha =f_{i\alpha}V_\alpha$ for all $\alpha\in\Gamma_i$.
    Replacing $\Gamma_i$ with $\Gamma_i\setminus\cup_{j<i}\Gamma_{j}$ we may assume the union is disjoint,
    so $\mathcal{V}=\prod_{i}\prod_{\alpha\in\Gamma_i}V_\alpha$,
    and we have $(f_1,\ldots,f_s)\prod_{\alpha\in\Gamma_i}V_\alpha=f_i\prod_{\alpha\in\Gamma_i}V_\alpha$.
    This shows that, Zariski-locally,
    $(f_1,\ldots,f_s)\mathcal{V}$ is principal.
    Further partition $\Gamma_i$ according to $f_{i\alpha}=0$ or not, we see, Zariski-locally,
    $(f_1,\ldots,f_s)\mathcal{V}$ is zero or generated by a nonzerodivisor.
    Therefore $(f_1,\ldots,f_s)\mathcal{V}$ is projective.
\end{proof}

\begin{Cor}
    \label{cor:ProdofVunifCoherentAndLocallyValRing}
    Let $\{V_\alpha\}_{\alpha\in \Gamma}$ be a family of valuation rings and let $\mathcal{V}=\prod_\alpha V_\alpha$.
    Then $\mathcal{V}_\fp$ is a valuation ring for every $\fp\in\Spec\mathcal{V}$, and every essentially finitely presented $\mathcal{V}$-algebra is coherent.
\end{Cor}
\begin{proof}
    See \cite[Cor.~4.2.19 and 7.3.4]{Coherent-Rings},
     as well as \cite[Th.~2.4.1 and 2.4.2]{Coherent-Rings}.
\end{proof}

\subsection{Vanishing in large residue characteristics}
We are able to deduce a version of Kodaira's vanishing theorem over a valuation ring of large residue characteristic relative to a ``bounded family.''
Note that even for Noetherian $V$, the proof involves non-Noetherian rings, namely other localizations of $\mathcal{V}$.
\begin{Thm}\label{thm:VanishingForLargep}
Let $r\in\bZ_{>0}$ and let $H\in\bQ[t]$.
Let $d=\deg H$.
There exists a constant $C=C(r,H)$ with the following property.

    Let $V$ be a valuation ring of residue characteristic $p>C$,
    and let $X$ be a closed subscheme of $\bP^r_V$.
    Assume that $X$ is flat over $\Spec V$ of relative dimension $d$ and Hilbert polynomial $H$,
    and assume that $\operatorname{Reg}(X)=X$.
    %$\operatorname{w.dim}\cO_{X,x}<\infty$ for all $x\in X$.

    Let $\omega_X=\mathcal{E}xt^{r-d}_{\cO_{\bP^r_V}}(\cO_X,\cO_{\bP^r_V}(-r-1))$.
    Then %for every ample invertible sheaf $\cL$ on $X$ and every $j>0$ we have 
    $H^j(X,\omega_X(1))=0$ for all $j>0$.
\end{Thm}
\begin{proof}
    As $X\subseteq \bP^r_V$,
    we have $H^j(X,\cF)=0$ for all $j>r$ and all quasi-coherent $\cF$.
    Therefore, it suffices to find $C=C(r,H,j)$ for a fixed $j$.

Assume, to the contrary, that such a $C$ does not exist.
Then for each $n\in\bZ_{>0}$, there exists a valuation ring $(V_n,\fm_n)$ of residue characteristic $>n$,
a closed subscheme $X_n$ of $\bP^r_{V_n}$ flat over $\Spec V_n$ of relative dimension $d$ and Hilbert polynomial $H$,
satisfying $\operatorname{Reg}(X_n)=X_n$
and $H^j(X_n,\omega_{X_n}(1))\neq 0$,
where $\omega_{X_n}=\mathcal{E}xt^{r-d}_{\cO_{\bP^r_{V_n}}}(\cO_{X_n},\cO_{\bP^r_{V_n}}(-r-1))$.
Note that $X_n$ is of finite presentation over $V_n$ by flatness \citestacks{053G}.

    Let $\mathcal{H}$ be the Hilbert scheme of $\bP^r_\bZ$ with Hilbert polynomial $H$.
    Then $\mathcal{H}$ is projective over $\bZ$ (see \cite[Chap.~5]{FGA-explained}; properness \citestacks{0DPH}, or even qcqs, is sufficient) and each $X_n$ comes from a morphism $\varphi_n:\Spec V_n\to \mathcal{H}$.
    By \cite[Th.~1.3]{Bha16-prod}, letting $\mathcal{V}=\prod_n V_n$,
    there is a (unique) morphism $\varphi:\Spec\mathcal{V}\to \mathcal{H}$ compatible with all $\varphi_n$.
    Then $\varphi$ gives rise to a closed subscheme $\mathcal{X}$ of $\bP^r_{\mathcal{V}}$ such that $\mathcal{X}\times_{\Spec\mathcal{V}}\Spec V_n= X_n$.
    Moreover, $\mathcal{X}$ is flat of finite presentation over $\mathcal{V}$ with Hilbert polynomial $H$.
    By Corollary \ref{cor:ProdofVunifCoherentAndLocallyValRing}, $\cO_{\mathcal{X}}(\mathcal{U})$ is coherent for all affine opens $\mathcal{U}$ of $\mathcal{X}$.
    We will establish two more properties of $\mathcal{X}$ and then derive a contradiction.

    First, we claim that the flat morphism $\pi:\mathcal{X}\to\Spec\mathcal{V}$ has relative dimension $d$ (\emph{i.e.} there are no connected components of smaller dimension).
    To see this, by \citetwostacks{045U}{02NM},
    the locus $\mathcal{W}$ of $\mathcal{X}$ that $\pi$ is Cohen--Macaulay of relative dimension $d$ is open.
    Since the construction of $\mathcal{W}$ is compatible with base change, Noetherian approximation \citetwostacks{01ZM}{05LY} tells us $\mathcal{W}$ is quasi-compact,
    in other words, $\mathcal{X}\setminus\mathcal{W}$ is a constructible closed subset.
    Write $M_n$ for the preimage of $\fm_n$ in $\mathcal{V}$.
    As $X_n\to \Spec V_n$ has relative dimension $d$ and has Cohen--Macaulay fibers \cite[Th.~2.9]{Knaf-RLR-over-Prufer},
    we see $M_n\not\in\pi(\mathcal{X}\setminus\mathcal{W})$ for all $n$.
    Note that $\pi(\mathcal{X}\setminus\mathcal{W})$ is constructible by Chevalley's Theorem \citestacks{054K}.
    By Lemma \ref{lem:ProdofLocalRings}\ref{prodlocal:dense},
    $\Max(\mathcal{V})\cap\pi(\mathcal{X}\setminus\mathcal{W})=\emptyset$.
    Since $\pi$ is proper, we see $\mathcal{X}\setminus\mathcal{W}=\emptyset$,
    as desired.

    Next, we claim that $\operatorname{w.dim}\cO_{\mathcal{X},x}\leq d+1$ for all $x\in \mathcal{X}$.
    Since every coherent $\cO_{\mathcal{X},x}$-module comes from a coherent sheaf in an affine open neighborhood \citestacks{01ZR} and therefore a coherent sheaf on $\mathcal{X}$ \citestacks{0G41},
    it suffices to show every $\mathcal{F}\in\Coh(\mathcal{X})$ has a resolution by vector bundles of length $\leq d+1$.
    As $\cO_{\mathcal{X}}(1)$ is ample there exists an exact sequence
    \[
    0\rightarrow\mathcal{K}\rightarrow\mathcal{E}_d\rightarrow\ldots\rightarrow\mathcal{E}_0\rightarrow\mathcal{F}\rightarrow0
    \]
    of coherent sheaves where each $\mathcal{E}_k$ is a vector bundle of rank $r_k$.
    %of the form $\cO_{\mathcal{X}}(-m_k)^{\oplus r_k}$.
    Let $\mathcal{U}_r$ be the locus where $\mathcal{K}$ is locally free of rank $r$.
    Then $\mathcal{U}_r$ is a quasi-compact open, as its complement is the union of  closed subschemes cut out by the finitely generated (coherent) ideals $\operatorname{Fit}_0(\operatorname{Fit}_{r-1}(\mathcal{K}))$ and $\operatorname{Fit}_{r}(\mathcal{K})$ \citestacks{07ZD}.
    Here we used the fact that the 
    Fitting ideals of a finitely presented module is finitely generated and the zeroth Fitting ideal cuts out the support, see \citestacks{07ZA}.
    Therefore, the locus $\mathcal{U}$ where $\mathcal{K}$ is locally free is the quasi-compact open $\mathcal{U}_0\cup\ldots\cup\mathcal{U}_{r_d}$.
    Since $\operatorname{w.dim}\cO_{\mathcal{X},x}\leq d+1$ for all $x\in \mathcal{X}$ above any $M_n$ (see \cite[Th.~1]{Knaf-RLR-over-Prufer}),
    we see $M_n\not\in\pi(\mathcal{X}\setminus\mathcal{U})$,
    so similar to the previous paragraph we get $\mathcal{U}=\mathcal{X}$.

    Finally, consider $\omega_{\mathcal{X}}=\mathcal{E}xt^{r-d}_{\cO_{\bP^r_{\mathcal{V}}}}(\cO_{\mathcal{X}},\cO_{\bP^r_{\mathcal{V}}}(-r-1))\in\Coh(\mathcal{X})$.
    For every $M\in \Max(\mathcal{V})\setminus\{M_1,M_2,\ldots\}$,
    we know $\mathcal{V}_M$ is a valuation ring of residue characteristic zero (Corollaries \ref{cor:ProdofVunifCoherentAndLocallyValRing} and \ref{cor:CharofProd}).
    Hence $\omega_{\mathcal{X}}$ restricts to an invertible sheaf on $\mathcal{X}\times_{\Spec\mathcal{V}}\Spec\mathcal{V}_M$ and $H^j(\mathcal{X},\omega_{\mathcal{X}}(1))_M=0$, see Theorem \ref{thm:VanishingOverQ};
    here we use  $\operatorname{w.dim}\cO_{\mathcal{X},x}\leq d+1$ for all $x\in \mathcal{X}$,
    hence the same is true for $\mathcal{X}\times_{\Spec\mathcal{V}}\Spec\mathcal{V}_M$.
    Letting $\mathcal{Z}$ be the constructible closed subset of $\mathcal{X}$ cut out by %the first Fitting ideal of 
    $\operatorname{Fit}_1(\omega_{\mathcal{X}})$,
    we see $\pi(\mathcal{Z})\cap \Max(\mathcal{V})\subseteq \{M_1,M_2,\ldots\}$.\footnote{Using \cite[Th.~2]{Knaf-RLR-over-Prufer} and Lemma \ref{lem:QFNormal=Val}, it can be shown directly that the ideal of $X_n$ is locally generated by a regular sequence of length $r-d$, so $\omega_{X_n}$ is invertible; then $\pi(\mathcal{Z})\cap \Max(\mathcal{V})=\emptyset$, making the reduction following unnecessary.}
    %Indeed, any truncation of $L_{X_n/\bP^r_{V_n}}$ is a perfect complex, so the same technique as in the proof of Theorem \ref{thm:lci=Lin-1} together with \cite[Th.~B]{Briggs-Iyenger-Cotangent-Complex} imply that $X_n\to \bP^r_{V_n}$ is a regular immersion.}
    Note that $\pi(\mathcal{Z})\cap \Max(\mathcal{V})$ and every $\{M_n\}$  is constructible in $\Max(\mathcal{V})$ (\citestacks{054K} and Lemma \ref{lem:ProdofLocalRings}\ref{prodlocal:Principal}),
    hence by compactness \citestacks{0901} $\pi(\mathcal{Z})\cap \Max(\mathcal{V})\subseteq \{M_1,M_2,\ldots,M_n\}$ for some $n$.
    Therefore, we may remove finitely many $V_n$ to assume $\pi(\mathcal{Z})\cap \Max(\mathcal{V})=\emptyset$, so $\mathcal{Z}=\emptyset$ by properness and thus $\omega_\mathcal{X}$ is invertible.
    It now follows that $R\Gamma(\mathcal{X},\omega_\mathcal{X})$ is a perfect object in $D(\mathcal{V})$ whose formation commutes with arbitrary base change \citestacks{0B91}.
    In particular, $T:=H^j(\mathcal{X},\omega_\mathcal{X})$ is a coherent $\mathcal{V}$-module and $T_{M}=0$ for all $M\in \Max(\mathcal{V})\setminus\{M_1,M_2,\ldots\}$, but $T_{M_n}\neq 0$ for all $n$.
    As $\operatorname{Supp}(T)$ is constructible closed subset \citestacks{051B} and as each $\{M_n\}$ is constructible in $\Max(\mathcal{V})$ (Lemma \ref{lem:ProdofLocalRings}\ref{prodlocal:Principal}) this is a contradiction to compactness \citestacks{0901}.
\end{proof}

\begin{Rem}\label{rem:VanishingEasyForSmooth}
    If, instead of $\operatorname{Reg}(X)=X$, we assume $X\to \Spec V$ is smooth, then Theorem \ref{thm:VanishingForLargep} is trivial, and we may even allow $V$ to be arbitrary local rings.
    Indeed, let $\mathcal{H}_0$ be locus of $\mathcal{H}$ over which the universal family $\pi_0:\mathcal{Y}_0\subseteq \bP^r_{\mathcal{H}_0}\to \mathcal{H}_0$ is smooth of relative dimension $d$,
    an open subscheme of $\mathcal{H}$.
    We can similarly define $\omega_{\mathcal{Y}_0/\mathcal{H}_0}=\mathcal{E}xt^{r-d}_{\cO_{\bP^r_{\mathcal{H}_0}}}(\cO_{Y_0},\cO_{\bP^r_{\mathcal{H}_0}}(-r-1))$, the unique nonzero cohomology sheaf of $\pi_0^!\cO_{\mathcal{H}_0}$.
    Then $\omega_{\mathcal{Y}_0/\mathcal{H}_0}$ is an invertible sheaf, so
    $T_0:=R\pi_{0*}(\omega_{\mathcal{Y}_0/\mathcal{H}_0}(1))$ is a perfect object in $D(\mathcal{H}_0)$ whose formation commutes with arbitrary base change.
    Therefore $T_0\otimes^L_{\cO_{\mathcal{H}_0}}\kappa(s)\in D^{\leq 0}$, for all $s\in \mathcal{H}_0$ lying above $0\in\Spec\bZ$, by the classical Kawamata--Viehweg. 
It follows from Nakayama's Lemma $\tau^{>0}T_0$ restricts to $0$ on $\mathcal{H}_0\times_{\Spec \bZ}\Spec\bQ$, so some nonzero integer annihilates all $\tau^{>0}T_0$, which implies the result by base change,
    as the image of $\Spec V$ in $\mathcal{H}$ is contained in $\mathcal{H}_0$.

    In particular, Theorem \ref{thm:VanishingForLargep} is trivial if we restrict to the case where $V$ is a perfect field.
    When $V$ is a(n imperfect) field, we can show $X$ is smooth when $\operatorname{char}V$ is large with respect to $r$ and $H$, using an analogous argument to the proof of Theorem \ref{thm:VanishingForLargep}.
    However, the author does not have an easier proof for Theorem \ref{thm:VanishingForLargep} for non-fields, e.g. $V=\bZ_p$.
\end{Rem}

The same method yields
\begin{Thm}\label{thm:VanishingForLargep-fixdegree}
Let $r,m,n,d\in\bZ_{>0}$.
There exists a constant $C=C(r,m,n)$ with the following property.

    Let $V$ be a valuation ring of residue characteristic $p>C$,
    and let $X$ be a closed subscheme of $\bP^r_V$ defined by at most $m$ homogeneous polynomials of degree at most $n$.
    Assume that $X$ is flat over $\Spec V$ of relative dimension $d$,
    and assume that $\operatorname{Reg}(X)=X$.

    Let $\omega_X=\mathcal{E}xt^{r-d}_{\cO_{\bP^r_V}}(\cO_X,\cO_{\bP^r_V}(-r-1))$.
    Then %for every ample invertible sheaf $\cL$ on $X$ and every $j>0$ we have 
    $H^j(X,\omega_X(1))=0$ for all $j>0$.
\end{Thm}
\begin{proof}
    %Identical to the proof of Theorem \ref{thm:VanishingForLargep}, except that the Hilbert scheme $\mathcal{H}$ is replaced by a finite disjoint union of finite products of affine spaces parameterizing the homogeneous polynomials defining $X$.
Without loss of generality, we may assume $X$ is defined by $m$ homogeneous polynomials of fixed degrees $n_1,\ldots,n_m\leq n$.
Then there exists an affine space $\bA^N_\bZ$ parametrizing such polynomials and $X$ comes from a morphism $\varphi:\Spec V\to \bA^N_\bZ$.
More precisely, there exists a closed subscheme $\mathcal{Z}\subseteq\bA^N_\bZ\times \bP^r_\bZ$
such that $X\subseteq\bP^r_V$ is the base change of $\mathcal{Z}$ along $\varphi\times\operatorname{id}:\bP^r_V\to \bA^N_\bZ\times \bP^r_\bZ$.
As we only consider flat $X$ we may therefore replace $\mathcal{Z}\subseteq\bA^N_\bZ\times \bP^r_\bZ$ by $Z\subseteq H\times \bP^r_\bZ$
where $Z\to H$ is the universal flattening of $\mathcal{Z}\to\bA^N_\bZ$ \citestacks{05UH}.
Now the proof of Theorem \ref{thm:VanishingForLargep} works the same way with this $\mathcal{H}$ in place of the Hilbert scheme.
\end{proof} 
\printbibliography
\end{document}